\documentclass[reqno,10pt]{amsart}

\usepackage[a4paper,left=35mm,right=35mm,top=30mm,bottom=30mm,marginpar=25mm]{geometry}

\usepackage{amsmath}
\usepackage{amssymb}
\usepackage{amsthm}

\usepackage{graphicx}
\usepackage[colorlinks=true, pdfstartview=FitV, 
 linkcolor=blue, citecolor=blue, urlcolor=blue]{hyperref}
\usepackage[dvipsnames]{xcolor}
\usepackage[nocompress, space]{cite}

\usepackage{mathtools}
\mathtoolsset{showonlyrefs=false}

\usepackage{enumitem}

\allowdisplaybreaks

\newcommand{\supp}{\operatorname{supp}}
\newcommand{\esssup}{\operatorname*{esssup}}

\renewcommand{\div}{\operatorname{div}}

\newcommand{\RR}{{\mathbb{R}}}

\newcommand{\Nn}{{\mathbb{N}}}

\newcommand{\Tt}{{\mathbb{T}}}

\newcommand{\epsi}{\varepsilon}

\def\leq{\leqslant}
\def\geq{\geqslant}

\numberwithin{equation}{section}

\newtheoremstyle{thmlemcorr}{10pt}{10pt}{\itshape}{}{\bfseries}{.}{10pt}{{\thmname{#1}\thmnumber{
        #2}\thmnote{ (#3)}}}

\theoremstyle{thmlemcorr}
\newtheorem{theorem}     {Theorem}
\numberwithin{theorem}{section} 
\newtheorem{lemma}      [theorem]{Lemma}
\newtheorem{corollary}   [theorem]{Corollary}
\newtheorem{proposition}  [theorem]{Proposition}
\newtheorem{assumption}        {Assumption} 

\newtheoremstyle{defi}{10pt}{10pt}{\itshape}{}{\bfseries}{.}{10pt}{{\thmname{#1}\thmnumber{
        #2}\thmnote{ (#3)}}}

\theoremstyle{defi}
\newtheorem{definition}  [theorem]{Definition}
\newtheorem{problem}     {Problem}  

\newtheoremstyle{remexample}{10pt}{10pt}{}{}{\bfseries}{.}{10pt}{{\thmname{#1}\thmnumber{
        #2}\thmnote{ (#3)}}}

\theoremstyle{remexample}
\newtheorem{remark}     [theorem]{Remark}
\newtheorem{example}     [theorem]{Example}

\makeatletter
\newcommand{\myulabel}[2]{%
#1\protected@edef\@currentlabel{#1}\label{#2}%
}
\makeatother

\begin{document}
     
     \title[First-Order Time-Dependent MFGs with Local Couplings]{Existence and Structure for First-Order Time-Dependent Mean-Field Games with Local Couplings}
     
\author{Diogo Gomes}
\address[D. Gomes]{
     King Abdullah University of Science and Technology (KAUST),
     CEMSE Division, Thuwal 23955-6900, Saudi Arabia.}
\email{diogo.gomes@kaust.edu.sa}

\author{Melih \"U\c{c}er}
\address[M. \"U\c{c}er]{
 King Abdullah University of Science and Technology (KAUST),
     CEMSE Division, Thuwal 23955-6900, Saudi Arabia.
     }
\email{melih.ucer@kaust.edu.sa}
     


\thanks{The research reported in this publication was supported by funding from King Abdullah University of Science and Technology (KAUST)}

\date{\today}

\begin{abstract}
We develop a Banach-space framework for first-order time-dependent
mean-field games with local couplings, using monotone operator theory and
low-order \(p\)-Laplacian regularization to avoid high-order elliptic
smoothing. Under monotonicity and power-growth assumptions, together with
either a Lagrangian lower bound or strict positivity of the initial density,
we prove existence of weak variational-inequality solutions by Minty's
method. The constructed solutions satisfy uniform \(L^\beta\) estimates on
the density, \(L^\alpha\) estimates on the spatial gradient of the value
function, and space-time shift estimates sufficient to identify the limiting
PDE system. We prove that any variational-inequality solution satisfying
these bounds, regardless of how it is obtained, is a MFG solution satisfying
the Hamilton--Jacobi and transport equations in the \(BV\) sense. This
separates the construction of VI-solutions from the verification of the PDE
system, a feature not directly available in the existing stationary
Banach-space framework. Finally, for each fixed density \(m\), we establish
a maximal value function among Hamilton--Jacobi subsolutions; every MFG
value function coincides with this maximal representative on \(\{m>0\}\)
and initially on \(\{m_0>0\}\). Under semi-strict monotonicity, the density
\(m\) itself is unique. Our results apply to non-separable Hamiltonians with
power growth and impose no dimension restrictions.
\end{abstract}

\maketitle
 
\section{Introduction}\label{sec:intro}

Mean-field game (MFG) theory, introduced independently in~\cite{ll1,ll2} and~\cite{Caines2}, provides a mathematical framework for modeling strategic interactions among large populations of rational agents. The equilibrium of a mean-field game is characterized by a coupled system of nonlinear partial differential equations: a Hamilton--Jacobi equation for the agents' value function~$u$ and a transport (or Fokker--Planck) equation for the population density~$m$. The theory has found applications ranging from price formation in financial markets to crowd dynamics in urban planning; we refer to~\cite{lasryMeanFieldGames2007, LCDF} and the monograph~\cite{GPV} for a comprehensive account.

In this paper, we study the existence, uniqueness, and structural regularity properties of solutions to a first-order, time-dependent, MFG problem, which models the local interactions of agents over a finite-time horizon $[0,T]$, on the $d$-dimensional flat torus $\Tt^d$ ($d\in\Nn$) as the state space. More precisely, we consider the following problem.

\begin{problem}[Mean Field Game]\label{prob.mfg}
Let $Q := (0,T)\times\Tt^d$.
Consider a Hamiltonian function, $H\colon Q\times \RR^d \times \RR^+_0 \to \RR$, which is Lebesgue measurable and such that the map $(p,m) \mapsto H(t,x,p,m)$ and its $p$-gradient $D_pH(t,x,p,m)$ are continuous on $\RR^d \times \RR^+_0$ for a.e.~$(t,x)\in Q$.
Given an initial distribution $m_0\in C(\Tt^d;\RR^+_0)$ and a terminal cost $u_T\in C^1(\Tt^d;\RR)$,
we seek pairs \((m,u)\), with suitable weak regularity specified in
Definition~\ref{def.strong.sol}, solving the system
  \begin{equation}\label{mfg}
     \begin{aligned}
   \begin{cases}
        -u_t + H(t,x,Du,m) =0,\\
        m_t - \div(m D_pH(t,x, Du,m)) =0,\\
        m(0,x)=m_{0}(x),\enspace u(T,x)=u_T(x),
   \end{cases}
     \end{aligned}
  \end{equation}
\end{problem}

The analysis below establishes existence and studies the uniqueness and structural properties of such pairs. Here, $u$ is the value function of a representative agent, $m$ is the distribution of the population, and the Hamiltonian~$H$ encodes the running cost and the agents' dynamics. The first equation is a Hamilton--Jacobi equation governing the optimal strategy of each agent, while the second is a transport equation describing the evolution of the population distribution. The initial distribution $m_0$ and the terminal cost~$u_T$ are prescribed data.
Under structural assumptions on the Hamiltonian~$H$, stated in Section~\ref{sec:assume}, we prove existence of solutions to Problem~\ref{prob.mfg}. The key component of our analysis is to consider an associated variational inequality problem (Problem~\ref{prob.vi}), 
establish that it admits solutions, and prove that these solutions also solve Problem~\ref{prob.mfg} in the sense defined below. We then prove the uniqueness of density and establish a corresponding canonical maximal value function.

While the stationary counterpart of this system has been extensively studied (see, e.g.,~\cite{FG2, FGT1, ferreiraSolvingMeanFieldGames2025} and references therein), the time-dependent problem is substantially harder. The first-order (non-diffusive) setting, in which the individual agents' dynamics are deterministic, presents specific analytical challenges: the value function $u$ has limited a priori regularity, which forces us to seek solutions in the framework of BV functions rather than Sobolev spaces.

Research on first-order MFG systems has developed along two main threads. The first thread uses variational methods rooted in optimal transport: existence and uniqueness for first-order systems with separable Hamiltonians was proved in~\cite{cardaliaguetMeanFieldGames2014}; weak solutions with local coupling were treated in~\cite{Card1order}; and Sobolev regularity results were obtained in~\cite{graber2017sobolev}. Both classical and weak solutions were established through elliptic regularity techniques in~\cite{munozClassicalWeakSolutions2022}, while interior H\"{o}lder continuity for weak solutions was proved in~\cite{alharbiRegularityWeakSolutions2025} without requiring monotonicity or convexity. In the time-dependent first-order setting, the planning problem in one dimension was studied via a potential approach in~\cite{DRT2021Potential}, and the case with Neumann boundary conditions was addressed in~\cite{gomesTimeDependentFirstOrder2024}. We also note the extensive literature on classical solutions for parabolic (second-order) MFGs~\cite{gomesTimeDependentMeanFieldGames2015, gomesTimedependentMeanfieldGames2016a, Gomes2015b, GM, GPatVrt, PV15, CirGo20} and for weak solutions to second-order systems~\cite{porretta, porretta2, bocorsporr, cgbt, MR3691806, AMFS}.

The second thread, and the one most relevant to our work, exploits the monotonicity condition on the Hamiltonian---originally central to the uniqueness proofs in~\cite{lasryMeanFieldGames2007}---as a tool for proving existence. A monotone operator approach for stationary MFGs through variational inequalities was introduced in~\cite{FG2}. This was extended to Dirichlet boundary conditions in~\cite{FGT1} and to the time-dependent setting in~\cite{FeGoTa21}, where existence of weak solutions was established via high-order elliptic regularization and Minty's method. The role of various monotonicity conditions in this context was analyzed in~\cite{graberMonotonicityConditionsMean2023}. More recently, a Banach space framework for stationary MFGs was developed in~\cite{ferreiraSolvingMeanFieldGames2025}, introducing a low-order $p$-Laplacian regularization that avoids the high-order smoothing of earlier approaches and yields strong solutions for models with power growth and congestion. The monotone operator viewpoint has also inspired numerical methods~\cite{almullaTwoNumericalApproaches2017, GJS2, nurbekyanMonotoneInclusionMethods2024}. We remark that results for non-monotone MFGs are typically local in nature, relying on contraction arguments valid only over short time intervals~\cite{ciranttonon, Ambrose}.

Despite these advances, existing monotonicity-based proofs for time-dependent MFGs~\cite{FeGoTa21} rely on Hilbert space methods with high-order regularization, which can be technically cumbersome, yield only weak solutions in the monotone operator sense, and are not well suited for the design of modern numerical algorithms. The Banach space framework of~\cite{ferreiraSolvingMeanFieldGames2025}, which overcomes these limitations, has so far been developed only for the stationary case. Extending it to the time-dependent setting poses genuinely new challenges: the value function has a BV structure whose singular part must be carefully controlled, and the initial and terminal conditions must be incorporated into the variational inequality formulation. Moreover, the space-time setting requires new uniform a priori estimates that have no direct stationary counterpart.

We work with two concepts of solutions, which we respectively call MFG-solutions and Variational Inequality (VI) solutions. To avoid the terminological ambiguity present in prior literature---where solutions analogous to our MFG-solutions have been variously referred to as strong solutions in the variational sense \cite{FG2, FGT1, FeGoTa21, ferreiraSolvingMeanFieldGames2025} or weak solutions \cite{cardaliaguetMeanFieldGames2014}---we adopt this distinct nomenclature. The VI-solutions correspond to the weak solutions used in \cite{ferreiraSolvingMeanFieldGames2025}.
The notation \(u_t^{\text{ac}}\), \(u_t^{\text{s}}\), \(Du\), and the
temporal traces of \(BV\)-functions are recalled in
Section~\ref{sec:note}; in the definitions below, \(Du\) denotes the
absolutely continuous spatial derivative.
\begin{definition}[MFG-Solution]\label{def.strong.sol}
 A pair $(m,u) \in L^1(Q; \RR^+_0) \times BV(Q)$ is a \emph{MFG-solution} to Problem~\ref{prob.mfg} if
\begin{enumerate}
  \item We have
  \begin{equation}\label{basic-integrability}
     \begin{aligned}
       & Du,\enskip mD_pH(t,x,Du,m)\in L^1(Q; \RR^d), \\
       & m(u_t^{\text{ac}} - D_pH(t,x,Du,m)\cdot Du)\in L^1(Q);
     \end{aligned}
  \end{equation}\smallskip

  \item The Hamilton--Jacobi equation in~\eqref{mfg} holds in the sense that
  \begin{equation}\label{def:hjb-strong}
     \begin{alignedat}{2}
       & -u_t + H(t,x,Du,m) \leq 0 \qquad && \text{in } Q = (0,T)\times\Tt^d, \\
       & -u_t^{\text{ac}} + H(t,x,Du,m) = 0 \qquad && \text{a.e.~in the set } \{m>0\},
     \end{alignedat}
  \end{equation}
    where the inequality is understood in the sense of distributions. By the Lebesgue decomposition of the measure $du_t$, this distributional inequality is equivalent to the pair of conditions
 \[
 u_t^{\text{s}} \geq 0 \qquad\text{and}\qquad -u_t^{\text{ac}} + H(t,x,Du,m) \leq 0 \quad\text{a.e.~in } Q.
 \]
  \smallskip

  \item The transport equation in~\eqref{mfg} with the initial condition $m(0,x) = m_0(x)$ holds in the sense that
  \begin{equation}\label{def:transport-strong}
     \int_Q m\bigl(-\varphi_t + D_pH(t,x,Du,m)\cdot D\varphi\bigr) \,dxdt = \int_{\Tt^d} m_0(x)\varphi(0,x)\,dx
  \end{equation}
  for all  $\varphi\in C^1(\bar{Q})$ satisfying $\varphi(T,x)=0$;\smallskip

  \item The terminal condition holds as an inequality $u(T,x) \leq u_T(x)$ in the sense of trace;\smallskip

  \item Moreover, we have \begin{equation}\label{def:singular-and-terminal}
     \begin{aligned}
       &\int_Q m\bigl(u_t^{\text{ac}} + D_pH(t,x,Du,m)\cdot (Du_T-Du)\bigr) \,dxdt \\
       & \qquad\quad \geq \int_{\Tt^d} m_0(x)(u_T(x)-u(0,x))\,dx.
     \end{aligned}
  \end{equation}
\end{enumerate}
\end{definition}

\begin{remark}
    If we formally integrate the transport equation $m_t - \div(m D_p H) = 0$ against the test function $\varphi(t,x) = u_T(x) - u(t,x)$, we obtain boundary terms at both $t=0$ and $t=T$. This yields the identity
    \begin{equation*}
        \begin{aligned}
        & \int_Q m\,du_t + \int_Q m D_pH(t,x,Du,m)\cdot (Du_T-Du)\,dxdt \\
        & \qquad = \int_{\Tt^d} m_0(x)(u_T(x)-u(0,x))\,dx - \int_{\Tt^d} m(T,x)(u_T(x)-u(T,x))\,dx.
        \end{aligned}
    \end{equation*}
    Decomposing the time derivative as $du_t = u_t^{\text{ac}}\,dxdt + du_t^{\text{s}}$ and rearranging, this becomes
    \begin{equation*}
        \begin{aligned}
        & \int_Q m\bigl(u_t^{\text{ac}} + D_pH(t,x,Du,m)\cdot (Du_T-Du)\bigr) \,dxdt - \int_{\Tt^d} m_0(x)(u_T(x)-u(0,x))\,dx \\
        & \qquad = -\int_Q m\,du_t^{\text{s}} - \int_{\Tt^d} m(T,x)(u_T(x)-u(T,x))\,dx.
        \end{aligned}
    \end{equation*}
    Condition~\eqref{def:singular-and-terminal} dictates that the left-hand side of this identity is non-negative. Therefore, the right-hand side must also be non-negative, which implies
    \[
        \int_Q m\,du_t^{\text{s}} + \int_{\Tt^d} m(T,x)(u_T(x)-u(T,x))\,dx \leq 0.
    \]
    Since the inequalities given earlier establish that $u_t^{\text{s}} \geq 0$ and $u(T,x) \leq u_T(x)$, both terms on the left-hand side are non-negative and must identically vanish. 
Consequently, condition~\eqref{def:singular-and-terminal} encodes that $u_t^{\text{s}} = 0$ in the set $\{m>0\}$ and $u(T,x) = u_T(x)$ in the set $\{m(T,x) > 0\}$. We emphasize that this derivation is purely formal---since $u$ lacks the regularity to be used as a test function and the product $m\,du_t^{\text{s}}$ is not rigorously defined---but it serves to motivate the inclusion of condition~\eqref{def:singular-and-terminal}.
\end{remark}

Our analysis relies on the observation that solutions to Problem~\ref{prob.mfg} in the sense of Definition~\ref{def.strong.sol} also solve a variational inequality associated with the formal operator
\[
A\begin{bmatrix}\mu\\ \upsilon\end{bmatrix}
=
\begin{bmatrix}
\upsilon_t-H(t,x,D\upsilon,\mu)\\
\mu_t-\div\bigl(\mu D_pH(t,x,D\upsilon,\mu)\bigr)
\end{bmatrix}.
\]
With the appropriate initial and terminal conditions, this operator is
monotone under Assumption~\ref{onH.monotone}. This observation is proved in
Lemma~\ref{strong.is.weak}. It motivates the following problem and the
corresponding notion of solution.

\begin{problem}[Variational Inequality]\label{prob.vi}
 Consider the setting of Problem~\ref{prob.mfg}. Find pairs \((m,u)\)
 satisfying the monotonicity-based variational inequality associated with
 \eqref{mfg}, in the precise sense of Definition~\ref{def.weak.sol}.
\end{problem}

\begin{definition}[VI-Solution]\label{def.weak.sol}
A pair $(m,u) \in L^1(Q; \RR^+_0) \times BV(Q)$ is a \emph{VI-solution} to Problem~\ref{prob.vi} if $Du\in L^1(Q;\RR^d)$ and
\begin{equation}
\label{weak-sol}
  \begin{aligned}
     & 0\leq \int_Q \bigl((\mu-m)(\upsilon_t-H(t,x,D\upsilon,\mu)) + \mu D_pH(t,x,D\upsilon,\mu)\cdot (D\upsilon-Du)\bigr) \,dxdt \\
     & \qquad + \int_Q \mu du_t- \int_Q \mu\upsilon_t \,dxdt \\
     & \qquad - \int_{\Tt^d} m_0(\upsilon(0,x)-u(0,x))\,dx + \int_{\Tt^d} \mu(T,x)(u_T(x)-u(T,x))\,dx,
  \end{aligned}
  \end{equation}
for every $(\mu, \upsilon) \in C(\bar{Q}; \RR^+_0)\times C^1(\bar{Q})$ satisfying $\mu(0,x)=m_0(x)$ and $\upsilon(T,x) = u_T(x)$, where all terms in \eqref{weak-sol} are required to be well-defined.
\end{definition}

\begin{remark}
The first space-time integral in \eqref{weak-sol} is understood as an
extended Lebesgue integral, required to be well-defined with value in
\((-\infty,+\infty]\). All remaining terms are finite and are understood in
the usual sense.
\end{remark}


Our main results are given in the following three theorems and a corollary. Theorem~\ref{weak.exists} establishes the existence of VI-solutions satisfying further integrability conditions. Theorem~\ref{weak.is.strong} shows that any such VI-solution is an MFG-solution; this equivalence is independent of the method used to construct the VI-solution, and in particular does not rely on the specific regularization scheme of this paper. Theorem~\ref{maximal-u-exists} establishes, for each density $m$ for which solutions exist, the existence of a unique maximal value function $u^*$ dominating all Hamilton--Jacobi subsolutions, and proves that all MFG value functions for that density agree with $u^*$ on $\{m>0\}$ and initially on $\{m_0>0\}$; moreover, under a semi-strict monotonicity condition, the density $m$ itself is uniquely determined. As a direct corollary, there exists a unique special MFG solution in the sense made precise by Corollary~\ref{strong.exists}.

\begin{theorem}[Existence of VI-Solutions]\label{weak.exists}
Consider the setting of Problem~\ref{prob.mfg}. Suppose that the Hamiltonian~$H$ satisfies the monotonicity condition (Assumption~\ref{onH.monotone}) and the basic power growth (Assumption~\ref{onH.powergrowth}).  Furthermore, assume that either the Lagrangian lower bound condition (Assumption~\ref{onH.Lagrangianlowerbound}) or the strict positivity of the initial density (Assumption~\ref{onm0.lowerbound}) holds.

Then, there exists a solution $(m,u)$ to Problem~\ref{prob.vi} in the sense of Definition~\ref{def.weak.sol}, which satisfies the following integrability bounds:
  \begin{align}
     & m\in L^{\beta}(Q; \RR^+_0), \label{integrability.further.m} \\
     & Du\in L^{\alpha}(Q;\RR^d), \label{integrability.further.du} \\
     & (u_t)^{-}\in L^{\beta'}(Q), \label{integrability.further.ut} \\
     & \sup_{(s,y)\in \bar{Q}} \left(\int_s^T\int_{\Tt^d} m(t-s,x-y)|Du(t,x)|^\alpha\,dxdt\right) < \infty. \label{integrability.further.mdu}
  \end{align}

Moreover, if the condition of no density dependence in high-order momentum terms (Assumption~\ref{onH.high-order-p.has-no-m}) also holds, this solution additionally satisfies:
\begin{equation}\label{integrability.further.H}
     (-u_t+H(t,x,Du,0))^+ \in L^{\beta'}(Q).
  \end{equation}
\end{theorem}

\begin{theorem}[Equivalence of VI-Solutions and MFG-Solutions]\label{weak.is.strong}
  Consider the setting of Problem~\ref{prob.mfg}. Suppose that the Hamiltonian~$H$ satisfies the monotonicity condition (Assumption~\ref{onH.monotone}) and the basic power growth (Assumption~\ref{onH.powergrowth}). Furthermore, assume that the condition of no time/space/density dependence in the highest-order momentum term (Assumption~\ref{onH.highest-order-p.has-no-txm}) holds.
  
  Let $(m,u)$ be a solution to Problem~\ref{prob.vi} in the sense of Definition~\ref{def.weak.sol} which satisfies~\eqref{integrability.further.m}, \eqref{integrability.further.du}, \eqref{integrability.further.ut}, \eqref{integrability.further.mdu}, and~\eqref{integrability.further.H}. Then $(m,u)$ is a solution to Problem~\ref{prob.mfg} in the sense of Definition~\ref{def.strong.sol}.
\end{theorem}

We next record structural properties of MFG-solutions, including regularity, stability, comparison, maximality, and uniqueness. To state these results precisely, we work under the assumptions of Theorem~\ref{weak.is.strong} and introduce two classes of value functions associated with a fixed density \(m\).
\begin{definition}\label{def:u-classes-at-given-m}
Given \(m\in L^\beta(Q;\RR^+_0)\), define the following two classes.
\begin{enumerate}
\item[\myulabel{$\mathbf{(\mathcal{U})}$}{def:Um}]
\(\mathcal{U}(m)\), the class of \emph{MFG value functions associated with \(m\)}, denotes the set of \(u\in BV(Q)\) which satisfies
    \eqref{integrability.further.du}, \eqref{integrability.further.ut}, \eqref{integrability.further.mdu}, \eqref{integrability.further.H}, and the following two equivalent conditions (see Theorem~\ref{weak.is.strong} and Lemma~\ref{strong.is.weak}):
    \begin{itemize}[label=$\bullet$]
        \item $(m,u)$ is a solution to Problem~\ref{prob.mfg} in the sense of Definition~\ref{def.strong.sol}.
        \item $(m,u)$ is a solution to Problem~\ref{prob.vi} in the sense of Definition~\ref{def.weak.sol}.
    \end{itemize}
    
\item[\myulabel{$\mathbf{(\mathcal{S})}$}{def:Um-tilde}]
\(\mathcal{S}(m)\), the class of \emph{admissible MFG subsolutions associated with \(m\)}, denotes the set of \(u\in BV(Q)\) which satisfies
    \eqref{integrability.further.du}, \eqref{integrability.further.ut}, \eqref{integrability.further.mdu}, \eqref{integrability.further.H}, and the following two conditions:
    \begin{itemize}[label=$\bullet$]
        \item the inequality part of~\eqref{def:hjb-strong}, namely \( -u_t + H(t,x,Du,m) \leq 0\),
        \item the terminal inequality \(u(T,x)\leq u_T(x)\).
    \end{itemize}
\end{enumerate}
\end{definition}

From the definition, the inclusion $\mathcal{U}(m)\subset\mathcal{S}(m)$ is clear. Moreover, we can read the conclusion of Theorem~\ref{weak.exists} as follows: there exists $m\in L^\beta(Q)$ such that $\mathcal{U}(m)$ is non-empty.

\begin{theorem}[Maximal MFG-Solutions]\label{maximal-u-exists}
    Consider the setting of Problem~\ref{prob.mfg}. Suppose that the Hamiltonian~$H$ satisfies the monotonicity condition (Assumption~\ref{onH.monotone}) and the basic power growth (Assumption~\ref{onH.powergrowth}). Furthermore, assume that the condition of no time/space/density dependence in the highest-order momentum term (Assumption~\ref{onH.highest-order-p.has-no-txm}) holds. Recall the notation introduced in Definition~\ref{def:u-classes-at-given-m}.
    \begin{enumerate}[label=(\Roman*)]
        \item\label{thmitem:maximal.u}
        Let $m\in L^\beta(Q;\RR^+_0)$ be such that $\mathcal{U}(m)$ is non-empty.
        
        If the conditions of no density dependence in high-order momentum terms (Assumption~\ref{onH.high-order-p.has-no-m}) and no time/space dependence in high-order momentum terms (Assumption~\ref{onH.remedy-bound}) also hold, there exists (a necessarily unique) $u^*\in \mathcal{U}(m)$ such that
    \begin{itemize}
        \item $u\leq u^*$ a.e.~for all $u\in \mathcal{S}(m)$,
        \item $\mathcal{U}(m)$ consists exactly of those $u\in \mathcal{S}(m)$ which satisfies $u=u^*$ a.e.~in the set $\{m>0\}$ and $u(0,\cdot) = u^*(0,\cdot)$ a.e.~in the set $\{m_0 > 0\}$.
        
    \end{itemize}

    \item\label{thmitem:unique.m} If the conditions of semi-strict monotonicity (Assumption~\ref{onH.strict-monotone}) and no time/space dependence or inter-dependence in high order momentum/density terms (Assumption~\ref{onH.high-order-pm.has-no-tx}) also hold, there exists at most one $m\in L^\beta(Q;\RR^+_0)$ such that $\mathcal{U}(m)$ is non-empty.
        
    \end{enumerate} 
\end{theorem}

As a direct corollary of the above theorems, we obtain the existence of a unique special solution to Problem~\ref{prob.mfg}. 
As we discuss in the next section, Remark~\ref{rmk:last-deviation-implies-the-earlier-two}, under Assumption~\ref{onH.powergrowth},
Assumption~\ref{onH.high-order-pm.has-no-tx} implies
Assumptions~\ref{onH.high-order-p.has-no-m}, \ref{onH.remedy-bound}, and
\ref{onH.highest-order-p.has-no-txm}.

\begin{corollary}[Existence of unique maximal MFG-Solutions]\label{strong.exists}
  Consider the setting of Problem~\ref{prob.mfg} and suppose that Assumptions~\ref{onH.strict-monotone}, \ref{onH.powergrowth}, \ref{onH.high-order-pm.has-no-tx}, and one of \ref{onH.Lagrangianlowerbound}/\ref{onm0.lowerbound} hold. Then there exists a unique pair $(m,u)$ with the following properties:
  \begin{itemize}
      \item $(m,u)$ is a solution to Problem~\ref{prob.mfg} in the sense of Definition~\ref{def.strong.sol}.
      \item $(m,u)$ satisfies \eqref{integrability.further.m}, \eqref{integrability.further.du}, \eqref{integrability.further.ut}, \eqref{integrability.further.mdu}, 
      and \eqref{integrability.further.H}.
      \item If $(m',u')$ is a solution to Problem~\ref{prob.mfg} in the sense of Definition~\ref{def.strong.sol}, which satisfies \eqref{integrability.further.m}, \eqref{integrability.further.du}, \eqref{integrability.further.ut}, \eqref{integrability.further.mdu}, and \eqref{integrability.further.H}, then
      \[\begin{aligned}
       & m' = m \text{ a.e. in } Q, \\
       & u'\leq u \text{ a.e. in } Q, \quad u'= u \text{ a.e. in the set } \{m>0\},\\ 
       & u'(0,\cdot) = u(0,\cdot) \text{ a.e. in the set } \{m_0>0\}.    
      \end{aligned}
      \]
  \end{itemize}
\end{corollary}

The hypothesis of Corollary~\ref{strong.exists} is satisfied by, for example, the Hamiltonian
\[H(t,x,p,m) = |p|^4 - m^3 + w_1(t,x)\frac{|p|^2}{1+m} - w_2(t,x),\]
for $w_1 \in L^5(Q;\RR^+_0)$ and $w_2\in L^2(Q;\RR^+_0)$. Note that this example is non-separable, need not possess continuity or boundedness in time and space, and imposes no restriction on the dimension of the domain in terms of the growth parameters. Another interesting example is
\[H(t,x,p,m) = |p|^\alpha- m^{\beta-1} + w(t,x)\]
with $\alpha > 1$, $\beta>1$, and $w\in L^{\infty}(Q)$,
which shows that the standard range \(\alpha>1\), \(\beta>1\) of power-growth exponents can be treated. For a more complete discussion on the scope of our analysis, see Section~\ref{sec:example}.

The main improvement over previous monotonicity-based approaches to
time-dependent first-order MFGs is that the existence proof is carried
out in a Banach-space framework with a low-order $p$-Laplacian
regularization, rather than through high-order elliptic regularization.
This allows us to obtain the BV compactness and space-time estimates
needed for local first-order couplings while keeping the variational
inequality formulation close to the original MFG system. 
Consequently, Theorem~\ref{weak.is.strong} establishes the equivalence between VI-solutions satisfying natural integrability bounds and MFG-solutions, under structural assumptions on the Hamiltonian alone, independently of how the VI-solution was obtained. This equivalence is entirely absent from previous works, including the stationary Banach space framework of~\cite{ferreiraSolvingMeanFieldGames2025}. Furthermore,
the comparison and lattice arguments leading to
Theorem~\ref{maximal-u-exists} identify a canonical maximal value
function: for a fixed density $m$ for which solutions exist, all
admissible Hamilton--Jacobi subsolutions lie below this maximal
representative, and the possible MFG value functions are characterized
by equality with it on the set where agents are present. 

Our approach also goes beyond traditional variational methods. Those methods
typically rely on a global minimization structure and are therefore most
directly suited to potential games, in particular to separable Hamiltonians.
The monotonicity-based Banach-space framework used here does not require such
a variational structure, and therefore accommodates genuinely non-separable
Hamiltonians.

Several natural questions remain open. First, regarding the interpretation of the maximal value function $u^*$: the classical Hamilton--Jacobi viscosity theory typically requires Lipschitz continuity of the solution (see, e.g.,~\cite{MR2153516}), and even the one-sided characterization of~\cite{Card1order} requires the dimension bound $d<\alpha/(\beta-1)$ to ensure H\"older continuity of $u^*$, whereas our results hold in all dimensions. The broader question of characterizing a maximal subsolution in viscosity or optimal control terms, when one is known to exist but may not even be continuous, appears to be open and is naturally motivated by the present work. Closely related is the derivation of $\epsi$-Nash equilibrium properties for the corresponding $N$-player games from the MFG solutions constructed here, which we leave for future investigation.

Second, regarding the scope of the assumptions: while~\cite{ferreiraSolvingMeanFieldGames2025} treats congestion models and Hamiltonians with weak growth in the stationary setting, extending these results to the time-dependent case requires further analysis. Also, we accommodate initial distributions that vanish by relying on the Lagrangian lower bound, but it would be desirable to completely drop both Assumption~\ref{onH.Lagrangianlowerbound} and Assumption~\ref{onm0.lowerbound}, as one of the two currently plays a necessary role in establishing the global a priori estimates.

The remainder of this paper is organized as follows.
In Sections~\ref{sec:assume} and~\ref{sec:example}, we present our assumptions on the
Hamiltonian and the data, including the monotonicity condition and growth
hypotheses, as well as their consequences and sufficient conditions ensuring them. Sections~\ref{sec:monotone} and~\ref{sec:meas-bv} collect
the tools from monotone operator theory and the theory of BV functions
used in our proofs. Section~\ref{sec:proof-weak-exists} is devoted to
the proof of Theorem~\ref{weak.exists}: we introduce the regularized
problem, derive uniform estimates, establish compactness, and pass to
the limit. In Section~\ref{sec:proof-weak-is-strong}, we prove
Theorem~\ref{weak.is.strong} by enlarging the class of admissible test
functions. Finally, in Section~\ref{sec:maximal-u}, we prove
Theorem~\ref{maximal-u-exists} using stability, comparison, and lattice
arguments.

\section{Preliminaries}

In this section, we collect the notation, structural assumptions, and analytical tools used throughout the paper. We first fix notation and state the assumptions on the Hamiltonian and the data. We then recall the monotone-operator and BV facts used in our proofs.

\subsection{Notation}\label{sec:note}

We denote the set of non-negative real numbers by $\RR^+_0 := [0, \infty)$,
and we write $\bar{Q} := [0,T]\times\Tt^d$ for the closure of the space-time
cylinder $Q = (0,T)\times\Tt^d$. For a real-valued function $v$, we denote its
positive and negative parts by $v^+ := \max(v,0)$ and $v^- := \max(-v,0)$; we
use the same notation $v^+$, $v^-$ for the positive and negative parts in the
Jordan decomposition of a signed Radon measure $v$. In the same vein, we denote $|v| := v^+ + v^-$ for the absolute value of a real-valued function or the total variation measure of a signed Radon measure.

For a bounded-variation function $u = u(t,x)$ on a space-time domain, we denote its distributional
derivative in time by $u_t$ and its spatial distributional derivative by $Du$.
The latter is, a priori, an $\RR^d$-valued Radon measure; throughout we work
under the standing condition $Du \in L^1(Q;\RR^d)$, so that this measure is
absolutely continuous with respect to the $(d+1)$-dimensional Lebesgue measure $dxdt$, its singular part vanishes, and we
denote its density again by $Du$. In particular, whenever $Du$ appears as an
argument of the Hamiltonian~$H$, it is this $L^1$ density that is meant.

By the Lebesgue decomposition theorem, we uniquely write
$du_t = u_t^{\text{ac}} \,dxdt + du_t^{\text{s}}$, where
$u_t^{\text{ac}} \in L^1(Q)$ is the density of the absolutely continuous part
with respect to the $(d+1)$-dimensional Lebesgue measure $dxdt$, and
$du_t^{\text{s}}$ is the mutually singular part. We write $du_t^{\text{s}}$ for
this measure inside integrals and simply $u_t^{\text{s}}$ otherwise (so that, for
instance, $u_t^{\text{s}} \geq 0$ means the singular measure is non-negative).
Point evaluations at the temporal boundaries, such as $u(0,\cdot)$ and
$u(T,\cdot)$ for $u \in BV(Q)$, are understood as interior traces.

Finally, $C$ denotes a positive constant, independent of the relevant
parameters, that may change from line to line; all inequalities are written so as
to remain valid when $C$ is replaced by any larger value.


\subsection{Assumptions}\label{sec:assume}

Now, we present our structural assumptions on the Hamiltonian~$H$ and the initial data in the setting of Problem~\ref{prob.mfg}. To facilitate the analysis, our hypotheses fall into four functional categories: (1) classical monotonicity conditions, which provide the fundamental variational structure and ensure uniqueness; (2) basic power growth conditions, which guarantee coercivity and global \textit{a priori} estimates; (3) novel deviation bounds, which relax strict classical separability assumptions to safely control error terms when approximating non-smooth $BV$ test functions; and (4) mutually independent lower-bounding conditions on the Lagrangian or initial density to secure global \emph{a priori} estimates.

\begin{assumption}[Monotonicity]
\label{onH.monotone}
The following monotonicity condition holds for a.e.~$(t,x)\in Q$ and for all $(p_1,m_1),\, (p_2,m_2) \in\RR^d\times\RR^+_0$:
     \begin{equation}
  \label{hmon}
   \begin{aligned}
        &\bigl(- H(t,x,p_1,m_1) + H(t,x,p_2,m_2)\bigr) (m_1 - m_2)\\ &\quad +
        \bigl( m_1 D_p H(t,x,p_1,m_1) - m_2 D_p H(t,x,p_2,m_2)\bigr) \cdot (p_1-p_2) \geq 0.
   \end{aligned}
     \end{equation}
\end{assumption}

This is a classical monotonicity condition from \cite{lasryMeanFieldGames2007}, which underlies the variational inequality formulation throughout this paper. Unpacking Definition~\ref{def:prelim} below, it states that for $K = \RR^+_0\times\RR^d\subset \RR^{d+1}$, the map $A\colon K \to \RR^{d+1}$  defined by
\[A(m,p) = \bigl(-H(t,x,p,m),\, mD_pH(t,x,p,m)\bigr)\]
is monotone for a.e~$(t,x)\in Q$, with respect to the standard bilinear pairing of $\RR^{d+1}$ with itself. We now point out some consequences of this fundamental monotonicity assumption.

\begin{remark}\label{rmk:cxty}
Suppose that~$H$ satisfies Assumption~\ref{onH.monotone}.
It follows by taking $p_1=p_2$ in \eqref{hmon} that for a.e.~$(t,x)\in Q$ and for all $p\in\RR^d$, the function $H(t,x,p,\cdot)$ is non-increasing in $\RR^+_0$. 
Similarly, taking $m_1=m_2=m>0$ in~\eqref{hmon} shows that, for a.e.~$(t,x)\in Q$, the map $p\mapsto D_pH(t,x,p,m)$ is monotone; hence $H(t,x,\cdot,m)$ is convex. The case $m=0$ follows by continuity in $m$.
\end{remark}

The convexity in $p$ and monotonicity in $m$ of the above remark are used repeatedly throughout the paper, in particular when enlarging the class of test functions and for the comparison principle. On the other hand, the following lemma shows that any MFG-solution is automatically a VI-solution. The converse direction, which is more subtle, is the content of Theorem~\ref{weak.is.strong}.

\begin{lemma}\label{strong.is.weak}
  Consider the setting of Problem~\ref{prob.mfg} and suppose that $H$ satisfies the monotonicity condition (Assumption~\ref{onH.monotone}). Let $(m,u)$ be a solution to Problem~\ref{prob.mfg} in the sense of Definition~\ref{def.strong.sol}. Then $(m,u)$ is a solution to Problem~\ref{prob.vi} in the sense of Definition~\ref{def.weak.sol}.
\end{lemma}
\begin{proof}
    Using~\eqref{hmon} with $(p_1,m_1) := (Du,m)$ and $(p_2,m_2) := (D\upsilon,\mu)$, and then adding and subtracting like terms, we obtain
  \begin{equation*}
  \begin{aligned}
     & (\mu-m)(\upsilon_t-H(t,x,D\upsilon,\mu)) + \mu D_pH(t,x,D\upsilon,\mu)\cdot(D\upsilon-Du) + \mu(u_t^{\text{ac}}-\upsilon_t) \\
     & \qquad \geq -m(u_t^{\text{ac}}-H(t,x,Du,m)) + \mu(u_t^{\text{ac}}-H(t,x,Du,m)) \\
     & \qquad\qquad + m(-\upsilon_t + D_pH(t,x,Du,m)\cdot (D\upsilon-Du_T)) \\
     & \qquad\qquad  
     + m\bigl(u_t^{\text{ac}} + D_pH(t,x,Du,m)\cdot(Du_T-Du)\bigr)
  \end{aligned}
  \end{equation*}
  a.e.~in~$Q$. 
  The Hamilton--Jacobi conditions in~\eqref{def:hjb-strong} imply that the first two terms on the right-hand side are non-negative, so we may discard them. Next, we note that the rest of the right-hand side is in $L^1(Q)$, thus we integrate both sides, use~\eqref{def:transport-strong} with $\varphi(t,x):=\upsilon(t,x)-u_T(x)$, and use~\eqref{def:singular-and-terminal} to obtain
  \begin{equation*}
  \begin{aligned}
     & \int_Q \bigl((\mu-m)(\upsilon_t-H(t,x,D\upsilon,\mu)) + 
     \mu D_pH(t,x,D\upsilon,\mu)\cdot(D\upsilon-Du)\bigr) \,dxdt \\
     & \qquad + \int_Q \mu(u_t^{\text{ac}}-\upsilon_t) \,dxdt \geq \int_{\Tt^d} m_0(\upsilon(0,x)-u(0,x))\,dx.
  \end{aligned}
  \end{equation*}
Adding the terms
\[
\int_Q\mu\,du_t^{\text{s}}
\qquad\text{and}\qquad
\int_{\Tt^d}\mu(T,x)(u_T(x)-u(T,x))\,dx,
\]
which are nonnegative by $u_t^{\text{s}}\geq0$ and $u(T,\cdot)\leq u_T$, yields~\eqref{weak-sol}.
\end{proof}

We also have a strict inequality version of Assumption~\ref{onH.monotone}, which we only use in the uniqueness proof.

\addtocounter{assumption}{-1}
\renewcommand{\theassumption}{\arabic{assumption}$^+$}
\begin{assumption}[Semi-Strict Monotonicity]\label{onH.strict-monotone}
    In addition to Assumption~\ref{onH.monotone}, the inequality in \eqref{hmon} is strict whenever $m_1 \neq m_2$.
\end{assumption}
\renewcommand{\theassumption}{\arabic{assumption}}

Next, we set forth our growth assumptions. The bounds that we assume indicate a power-like asymptotic behavior of~$H$ as $|p|\to\infty$ and $m\to\infty$, which is uniform for $(t,x)\in Q$. For all of these conditions, let the parameters
\begin{equation}\label{eq:alpha.beta}
  \alpha > 1 \qquad\text{and}\qquad \beta > 1
\end{equation}
be fixed. Our basic assumption, upon which all our further growth assumptions are built, essentially asserts that $p\mapsto H(t,x,p,m)$ grows like $|p|^\alpha$ and $m\mapsto -H(t,x,p,m)$ grows like~$m^{\beta-1}$. 

\begin{assumption}[Basic Power Growth] \label{onH.powergrowth}   
  There exists a positive constant $C$ and some $V\in L^{\beta'}(Q;\RR^+_0)$,  such that the following bounds hold
  for a.e.~$(t,x)\in Q$ and all $(p,m)\in\RR^d\times\RR^+_0$:
     \begin{align}
     H(t,x,0,m) & \leq -\frac{1}{C}m^{{\beta-1}} + CV(t,x), \label{eq:assH.upper.simple} \\
     |D_pH(t,x,p,m)| & \leq C\bigl( |p|^{\alpha-1} + m^{\frac{(\alpha-1)}{\alpha}(\beta-1)}+V(t,x)^{\frac{\alpha-1}{\alpha}}\bigr),\label{eq:assH.DpH.upper.simple}\\
     H(t,x,p,m) & \geq \frac{1}{C}|p|^{\alpha} - C( m^{{\beta-1}}+V(t,x)),\label{eq:assH.lower.simple}
     \end{align}
where $\alpha$, $\beta$ are as in~\eqref{eq:alpha.beta} and $\beta' := \beta/(\beta-1)$. By adding a positive constant to $V$ and increasing $C$ if necessary, we assume throughout that
\begin{equation}\label{eq:assH.potential}
 V(t,x)\geq 1 \quad\text{for a.e. }(t,x)\in Q,
\qquad
\int_Q V(t,x)^{\beta'}\,dxdt \leq C.
\end{equation}
\end{assumption}

The bounds in~\eqref{eq:assH.upper.simple} and~\eqref{eq:assH.DpH.upper.simple} together imply the upper bound
\begin{equation}\label{res:H.upper.simple}
  H(t,x,p,m) \leq C(|p|^{\alpha}+V(t,x)) - \frac{1}{C}m^{{\beta-1}}
\end{equation}
on the Hamiltonian, which complements the lower bound in~\eqref{eq:assH.lower.simple}. Moreover, Assumption~\ref{onH.monotone} and Assumption~\ref{onH.powergrowth} together imply
\begin{equation}\label{dphdotp-minus-h.lower.bound}
     D_p H(t,x,p,m) \cdot p - H(t,x,p,m) \geq \frac{1}{C}(|p|^\alpha + m^{\beta-1}) - CV(t,x),
\end{equation}
which is a lower bound on the corresponding Lagrangian. We refer to~\cite[Lemma~2.11 and Remark~2.12]{ferreiraSolvingMeanFieldGames2025} for the proofs of these facts, which are based on Fundamental Theorem of Calculus identities followed by Young's inequality estimates. The proof of Lemma~\ref{lem:deviation-bound} below follows the same idea.

Assumption~\ref{onH.powergrowth} is not sufficient for all of our analysis. In some estimates, we need more stringent upper bounds on the deviation of~$H$ between various quadruples $(t,x,p,m)$, which we impose under additional assumptions below. 
The first two of these amount to requiring that the part of the Hamiltonian that grows as
 $|p|\to\infty$ with a power greater than $\alpha(1-1/\beta)$ is independent of $m$ and $(t,x)$ respectively, so that it cancels in an expression of the form $H(t,x,p,m)-H(t,x,p,0)$ or $H(t,x,p,0)-H(t',x',p,0)$, respectively. These assumptions are used to obtain the $L^{\beta'}$ bound on $(-u_t+H(t,x,Du,0))^+$ where $u$ stands for the value function or its space-time shifts.

\begin{assumption}[High-Order Momentum Terms Free of Density]\label{onH.high-order-p.has-no-m}
  In addition to Assumption~\ref{onH.powergrowth}, the following bound holds for a.e.~$(t,x)\in Q$ and all $(p,m)\in\RR^d\times\RR^+_0$:
\begin{equation}\label{eq:high-order-p.has-no-m}
     |H(t,x,p,m) - H(t,x,p,0)| \leq C(|p|^{\alpha\frac{(\beta-1)}{\beta}} + m^{{\beta-1}} + V(t,x)),
  \end{equation}
  where $\alpha$, $\beta$, $C$, and $V$ are as in Assumption~\ref{onH.powergrowth}.
\end{assumption}

\begin{assumption}[High-Order Momentum Terms Free of Time/Space]
\label{onH.remedy-bound}
In addition to Assumption~\ref{onH.powergrowth}, the following bound holds for a.e.~$(t,x)\in Q$, a.e.~$(t',x')\in Q$, and all $p\in\RR^d$:
\begin{equation}\label{remedy-bound}
|H(t,x,p,0) - H(t',x',p,0)| \leq C\bigl( |p|^{\alpha\frac{(\beta-1)}{\beta}} + V(t,x) + V(t',x') \bigr),
\end{equation}
where $\alpha$, $\beta$, $C$, and $V$ are as in Assumption~\ref{onH.powergrowth}.
\end{assumption}

Our next assumption on deviations of $H$ means the following: the term in the Hamiltonian that grows proportionally to $|p|^\alpha$ is only a function of $p$, so that it cancels in an expression of the form $H(t,x,p,m)-H(t',x',p,0)$, which is then left with a growth power of $\alpha(1-1/C)$ for some sufficiently large constant $C$.

\begin{assumption}[Highest Order Momentum Term Free of Time/Space/Density]\label{onH.highest-order-p.has-no-txm}
  In addition to Assumption~\ref{onH.powergrowth}, there exists a non-decreasing function $R\colon\RR^+_0\to\RR^+_0$ such that the following bound holds for a.e.~$(t,x)\in Q$,  a.e.~$(t',x')\in Q$, and all $(p,m)\in\RR^d\times\RR^+_0$:
    \begin{equation}\label{eq:devH-at-fixed.p-bound}
        |H(t,x,p,m)-H(t',x',p,0)| \leq CR(m)\bigl(|p|^\alpha + V(t,x) +V(t',x')\bigr)^{1-1/C},
    \end{equation}
 where $\alpha$, $\beta$, $C$, and $V$ are as in Assumption~\ref{onH.powergrowth}.
 By adding a positive constant if necessary, we assume throughout that $R(m)\geq 1$ for all $m\in\RR^+_0$.
\end{assumption}

The bound \eqref{eq:devH-at-fixed.p-bound} can also be combined with~\eqref{eq:assH.DpH.upper.simple} to yield a bound on the deviation of~$H$ between arbitrary quadruples $(t,x,p,m)$ as follows.
\begin{lemma}\label{lem:deviation-bound}
  Suppose that $H$ satisfies Assumption~\ref{onH.highest-order-p.has-no-txm}. Then, the following bound holds for a.e.~$(t,x)\in Q$,  a.e.~$(t',x')\in Q$, and all $(p_1,m_1)\in\RR^d\times\RR^+_0$, $(p_2,m_2)\in\RR^d\times\RR^+_0$:
    \begin{equation}\label{eq:devH-quadruple}
    \begin{aligned}
        & |H(t,x,p_1,m_1) - H(t',x',p_2,m_2)| \\
        & \qquad \leq C \bigl( R(m_1+m_2)+|p_1-p_2|\bigr)\bigl(|p_1|^\alpha + |p_2|^\alpha + V(t,x) +V(t',x')\bigr)^{1-1/C},
    \end{aligned}
    \end{equation}
    where $C$ is possibly redefined.
\end{lemma}
\begin{proof}
    First, by \eqref{eq:devH-at-fixed.p-bound}, we have
    \begin{equation*}
    \begin{aligned}
        |H(t,x,p_1,m_1)-H(\bar{t},\bar{x},p_1,0)| & \leq CR(m_1)\bigl(|p_1|^\alpha + V(t,x) +V(\bar{t},\bar{x})\bigr)^{1-1/C}, \\
        |H(t',x',p_2,m_2)-H(\bar{t},\bar{x},p_2,0)| & \leq CR(m_2)\bigl(|p_2|^\alpha + V(t',x') +V(\bar{t},\bar{x})\bigr)^{1-1/C},
    \end{aligned}
    \end{equation*}
    for a.e.~$(t,x)\in Q$, a.e.~$(t',x')\in Q$, and a.e.~$(\bar{t},\bar{x})\in Q$. Replacing $R(m_1)$, $R(m_2)$ on the right-hand sides by $R(m_1+m_2)$ and adding the two bounds, then combining with the triangle inequality, we get
    \begin{equation}\label{eq:devH-at-fixed.p-assumption-used}
        \begin{aligned}
            & |H(t,x,p_1,m_1) - H(t',x',p_2,m_2)| \\
            & \qquad \leq |H(\bar{t},\bar{x},p_1,0) - H(\bar{t},\bar{x},p_2,0)| \\
            & \qquad\qquad + CR(m_1+m_2)\bigl(|p_1|^\alpha + |p_2|^\alpha + V(t,x) + V(t',x') + V(\bar{t},\bar{x})\bigr)^{1-1/C}.
        \end{aligned}
    \end{equation}
    Then, we use the Fundamental Theorem of Calculus to express
  \begin{equation*}
     H(\bar{t},\bar{x},p_1,0)-H(\bar{t},\bar{x},p_2,0) = \int_0^1 D_pH(\bar{t},\bar{x}, \lambda p_1 +(1-\lambda) p_2, 0)\cdot(p_1-p_2)\,d\lambda,
  \end{equation*}
  and estimate the integrand with~\eqref{eq:assH.DpH.upper.simple} to get
  \begin{equation}\label{devH.DpH.FTC+estimate}
     |H(\bar{t},\bar{x},p_1,0)-H(\bar{t},\bar{x},p_2,0)| \leq C|p_1-p_2|(|p_1|^\alpha+|p_2|^\alpha + V(\bar{t},\bar{x}))^{1-\frac{1}{\alpha}}.
  \end{equation}
    Finally, we replace the terms $R(m_1+m_2)$ and $|p_1-p_2|$ on the right-hand sides of \eqref{eq:devH-at-fixed.p-assumption-used} and~\eqref{devH.DpH.FTC+estimate} by the greater $\bigl(R(m_1+m_2)+|p_1-p_2|\bigr)$, we replace the power of $(1-\frac{1}{\alpha})$ in \eqref{devH.DpH.FTC+estimate} by the greater $(1-1/C)$, we plug the latter bound into the former, and we take the essential infimum over $(\bar{t},\bar{x})\in Q$ to replace $V(\bar{t},\bar{x})$ by $C$, consequently we conclude \eqref{eq:devH-quadruple}.
\end{proof}

The above lemma is used in the proof of Proposition~\ref{lem:hjb-limsup} to control the error when approximating BV test functions by smooth ones.

Finally, our last assumption on deviations of $H$ indicates the following: those parts of $H$ which grow proportionally to $|p|^\alpha$ and $m^{\beta-1}$ are functions of $p$ and $m$ only, so that they cancel in a triple difference of the form $H(t,x,p,m) - H(\bar{t},\bar{x},p,0) - H(\tilde{t},\tilde{x},0,m)$. This assumption controls the error when approximating BV test functions without the shift-integrable condition, and is used in the uniqueness proof for the density $m$ (Theorem~\ref{maximal-u-exists}~\ref{thmitem:unique.m}).

\begin{assumption}[High-Order Momentum and Density Terms Independent of Time and Space]\label{onH.high-order-pm.has-no-tx}
    In addition to Assumption~\ref{onH.powergrowth}, the following bound holds for a.e.~$(t,x)\in Q$, a.e.~$(\bar{t},\bar{x})\in Q$, a.e~$(\tilde{t},\tilde{x})\in Q$, and all $(p,m)\in\RR^d\times\RR^+_0$:
\begin{equation}\label{eq:H.triple-deviation}
    \begin{aligned}
        & |H(t,x,p,m) - H(\bar{t},\bar{x},p,0) - H(\tilde{t},\tilde{x},0,m)| \\
        & \qquad \leq C\bigl(|p|^{\alpha\frac{(\beta-1)}{\beta}}+m^{\beta-1} + V(t,x)+ V(\bar{t},\bar{x}) + V(\tilde{t},\tilde{x})\bigr)^{1-1/C}, 
    \end{aligned}
    \end{equation}
 where $\alpha$, $\beta$, $C$, and $V$ are as in Assumption~\ref{onH.powergrowth}.
\end{assumption}

\begin{lemma}\label{lem:highest-order-m.separated}
    Suppose that $H$ satisfies Assumption~\ref{onH.high-order-pm.has-no-tx}. Then, there exists a continuous function $f\colon \RR^+_0 \to \RR$ satisfying
    \begin{equation}\label{eq:highest-order-m.separated}
        |f(m)|\leq C(m^{\beta-1}+1)
    \end{equation}
    for all $m\in \RR^+_0$, such that the following bound holds for a.e.~$(t,x)\in Q$,  a.e.~$(t',x')\in Q$, and all $(p,m_1,m_2)\in\RR^d\times\RR^+_0\times\RR^+_0$:
    \begin{equation}\label{eq:deviation-at-fixed-pm-improved-to-m1m2}
    \begin{aligned}
        & |H(t,x,p,m_1) + f(m_1) - H(t',x',p,m_2) - f(m_2)|\\
        & \qquad \leq C\bigl(|p|^{\alpha\frac{(\beta-1)}{\beta}}+ m_1^{\beta-1} + m_2^{\beta-1} + V(t,x) + V(t',x')\bigr)^{1-1/C},
    \end{aligned}
    \end{equation}
    for a possibly larger constant $C$.
\end{lemma}

\begin{proof}
    We consider the identity
    \[\begin{aligned}
        & H(t,x,p,m_1) - H(\tilde{t},\tilde{x},0,m_1) - H(t',x',p,m_2) + H(\tilde{t},\tilde{x},0,m_2) \\
        & = \bigl(H(t,x,p,m_1) - H(\bar{t}, \bar{x}, p, 0) - H(\tilde{t},\tilde{x},0,m_1)\bigr) \\
        & \qquad - \bigl(H(t',x',p,m_2) - H(\bar{t}, \bar{x}, p, 0) - H(\tilde{t},\tilde{x},0,m_2)\bigr). 
    \end{aligned}
    \]
   We estimate both of the terms on the right-hand side with \eqref{eq:H.triple-deviation} and take essential infimum over $(\bar{t},\bar{x})\in Q$ to absorb $V(\bar{t},\bar{x})$ into the constant $C$, thus we get
    \begin{equation}\label{eq:deviation-at-fixed-pm-improved-to-m1m2:tilde-variables}
        \begin{aligned}
        & |H(t,x,p,m_1) - H(\tilde{t},\tilde{x},0,m_1) - H(t',x',p,m_2) + H(\tilde{t},\tilde{x},0,m_2)| \\
        & \qquad \leq C\bigl(|p|^{\alpha\frac{(\beta-1)}{\beta}}+m_1^{\beta-1} + m_2^{\beta-1} + V(t,x)+V(t',x')+V(\tilde{t},\tilde{x})\bigr)^{1-1/C}, 
    \end{aligned}
    \end{equation}
    for a.e.~$(t,x)\in Q$,  a.e.~$(t',x')\in Q$, a.e.~$(\tilde{t},\tilde{x})\in Q$, and all $(p,m_1,m_2)\in\RR^d\times\RR^+_0\times\RR^+_0$. Next, we fix $(\tilde{t},\tilde{x})\in Q$ such that $V(\tilde{t},\tilde{x})$ is finite and \eqref{eq:deviation-at-fixed-pm-improved-to-m1m2:tilde-variables} holds for a.e.~$(t,x)\in Q$ and a.e.~$(t',x')\in Q$. Then, we define $f(m) := -H(\tilde{t},\tilde{x},0,m)$. Thus, \eqref{eq:deviation-at-fixed-pm-improved-to-m1m2} immediately follows from \eqref{eq:deviation-at-fixed-pm-improved-to-m1m2:tilde-variables} after absorbing $V(\tilde{t},\tilde{x})$ into the constant $C$. On the other hand, \eqref{eq:highest-order-m.separated} follows from \eqref{eq:assH.upper.simple} and \eqref{eq:assH.lower.simple}.
\end{proof}

\begin{remark}\label{rmk:last-deviation-implies-the-earlier-two}
Under Assumption~\ref{onH.powergrowth},
Assumption~\ref{onH.high-order-pm.has-no-tx} implies
Assumptions~\ref{onH.high-order-p.has-no-m}, \ref{onH.remedy-bound}, and
\ref{onH.highest-order-p.has-no-txm}, with possibly redefined \(C\).

We first show that the deviation estimate \eqref{eq:H.triple-deviation} in
Assumption~\ref{onH.high-order-pm.has-no-tx} directly implies
Assumption~\ref{onH.remedy-bound}, without using the power-growth bounds. Indeed, applying \eqref{eq:H.triple-deviation} with \(m=0\), we write
\[
\begin{aligned}
& H(t,x,p,0)-H(t',x',p,0) \\
&\quad =
\bigl(H(t,x,p,0)-H(\bar t,\bar x,p,0)-H(\tilde t,\tilde x,0,0)\bigr) \\
&\qquad -
\bigl(H(t',x',p,0)-H(\bar t,\bar x,p,0)-H(\tilde t,\tilde x,0,0)\bigr).
\end{aligned}
\]
Therefore,
\[
\begin{aligned}
& |H(t,x,p,0)-H(t',x',p,0)| \\
&\quad \leq
C\bigl(|p|^{\alpha\frac{\beta-1}{\beta}}
      +V(t,x)+V(\bar t,\bar x)+V(\tilde t,\tilde x)\bigr)^{1-1/C} \\
&\qquad +
C\bigl(|p|^{\alpha\frac{\beta-1}{\beta}}
      +V(t',x')+V(\bar t,\bar x)+V(\tilde t,\tilde x)\bigr)^{1-1/C},
\end{aligned}
\]
for a.e.~$(t,x)\in Q$, a.e.~$(t',x')\in Q$, a.e.~$(\bar{t},\bar{x})\in Q$, and a.e~$(\tilde{t},\tilde{x})\in Q$. Since \(V\geq1\), the base expressions are greater than $1$, hence the powers \(1-1/C\) can be dropped. Hence
\[
\begin{aligned}
|H(t,x,p,0)-H(t',x',p,0)|
&\leq
C\bigl(|p|^{\alpha\frac{\beta-1}{\beta}}
      +V(t,x)+V(t',x')  \\
&\qquad\qquad
      +V(\bar t,\bar x)+V(\tilde t,\tilde x)\bigr).
\end{aligned}
\]
Taking the essential infimum over \((\bar t,\bar x)\) and
\((\tilde t,\tilde x)\), and absorbing the resulting finite quantity into
\(C\), gives
\[
|H(t,x,p,0)-H(t',x',p,0)|
\leq
C\bigl(|p|^{\alpha\frac{\beta-1}{\beta}}
      +V(t,x)+V(t',x')\bigr),
\]
which is precisely \eqref{remedy-bound}.

We now prove the remaining two implications, using
Assumption~\ref{onH.powergrowth}.
Considering the identity
    \[\begin{aligned}
        H(t,x,p,m) - H(t,x,p,0) & = \bigl(H(t,x,p,m) - H(\bar{t},\bar{x},p,0) - H(\tilde{t},\tilde{x},0,m)\bigr) \\
        & \qquad - \bigl(H(t,x,p,0) - H(\bar{t},\bar{x},p,0) - H(\tilde{t},\tilde{x},0,0)\bigr) \\
        & \qquad + H(\tilde{t},\tilde{x},0,m) - H(\tilde{t},\tilde{x},0,0)
    \end{aligned}\]
and estimating the terms on the right-hand side with \eqref{eq:H.triple-deviation} (noting that the exponent $1-1/C$ can be dropped due to $V \geq 1$), as well as \eqref{eq:assH.upper.simple} and \eqref{eq:assH.lower.simple}, we get
    \[\begin{aligned}
        |H(t,x,p,m) - H(t,x,p,0)| & \leq C\bigl(|p|^{\alpha\frac{(\beta-1)}{\beta}}+m^{\beta-1} + V(t,x)+ V(\bar{t},\bar{x}) + V(\tilde{t},\tilde{x})\bigr) \\
        & \qquad + C\bigl(|p|^{\alpha\frac{(\beta-1)}{\beta}} + V(t,x)+ V(\bar{t},\bar{x}) + V(\tilde{t},\tilde{x})\bigr) \\
        & \qquad + C(m^{\beta-1}+V(\tilde{t},\tilde{x})) + CV(\tilde{t},\tilde{x}) \\
        & \leq C\bigl(|p|^{\alpha\frac{(\beta-1)}{\beta}}+m^{\beta-1} + V(t,x)+ V(\bar{t},\bar{x}) + V(\tilde{t},\tilde{x})\bigr)
    \end{aligned}\]
    for a.e.~$(t,x)\in Q$, a.e.~$(\bar{t},\bar{x})\in Q$, a.e~$(\tilde{t},\tilde{x})\in Q$. Taking the essential infimum on the right-hand side over $(\bar{t},\bar{x})$ and $(\tilde{t},\tilde{x})$, we conclude \eqref{eq:high-order-p.has-no-m}. 
    
    Similarly, using \eqref{eq:highest-order-m.separated} and \eqref{eq:deviation-at-fixed-pm-improved-to-m1m2} of Lemma~\ref{lem:highest-order-m.separated}, we estimate
    \[\begin{aligned}
        |H(t,x,p,m) - H(t',x',p,0)| & \leq |H(t,x,p,m) + f(m) - H(t',x',p,0) - f(0)| \\
        & \qquad + |f(m)| + |f(0)| \\
        & \leq C\bigl(|p|^{\alpha\frac{(\beta-1)}{\beta}} + V(t,x) + V(t',x')\bigr)^{1-1/C} \\
        & \qquad + C(m^{\beta-1}+1)^{1-1/C} + C(m^{\beta-1}+1) + C \\
        & \leq C\bigl(|p|^\alpha + V(t,x) + V(t',x')\bigr)^{1-1/C}  + C(m^{\beta-1}+1) \\
        & \leq C(m^{\beta-1}+1)\bigl(|p|^{\alpha} + V(t,x) + V(t',x')\bigr)^{1-1/C},
    \end{aligned}\]
    where, in the last step, we replaced the sum of two terms by their product noting that the terms are greater than or equal to $1$ (redefining $C$ and using the inequality $a+b\leq 2 a b$ for $a,b\geq 1$). Thus, we obtain \eqref{eq:devH-at-fixed.p-bound} with $R(m):= m^{\beta-1}+1$.

\end{remark}

The preceding assumptions give integrability bounds for $H(t,x,0,0)$, but the following one-sided uniform bound is needed for a global lower bound on the associated Lagrangian, and is used in Proposition~\ref{prop:uepsi-lower-bound-A} to obtain a pointwise lower bound on $u_\epsi$.

\renewcommand{\theassumption}{\arabic{assumption}A}
\begin{assumption}[Lagrangian Lower Bound]
\label{onH.Lagrangianlowerbound}
The following one-sided bound holds: \[\esssup_{(t,x)\in Q} H(t, x, 0, 0) < \infty.\]
\end{assumption}
The Lagrangian $L$ is defined via the Legendre transform of the Hamiltonian,
\[
L(t,x,v,m) = \sup_{p\in \RR^d} \{-p \cdot v - H(t,x,p,m)\}. 
\]
Therefore, $L(t,x,v,m) \geq -H(t,x,0,m)$ by evaluating the function inside the supremum at $p=0$. If $H$ is non-increasing in $m$ as in Remark~\ref{rmk:cxty}, then $-H(t,x,0,m) \geq -H(t,x,0,0)$, so the condition $\esssup H(t,x,0,0) < \infty$ yields $L(t,x,v,m) \geq -C$ for some constant $C$. This prevents the optimal control problem from becoming ill-posed due to unbounded negative costs.

Alternatively, if the global lower bound on the Lagrangian is not guaranteed, our existence framework can still proceed, provided that the initial population density is strictly positive. Thus, the following assumption serves as a mutually independent alternative to Assumption~\ref{onH.Lagrangianlowerbound}; our main results require only one of these two conditions to hold.

\addtocounter{assumption}{-1}
\renewcommand{\theassumption}{\arabic{assumption}B}
\begin{assumption}[Strict Positivity of Initial Density]
  \label{onm0.lowerbound}
  We have
  $m_0(x)>0$ for every $x\in\Tt^d$.
\end{assumption}
\renewcommand{\theassumption}{\arabic{assumption}}

The condition $m_0 > 0$ in Assumption~\ref{onm0.lowerbound} has both modeling and analytical significance. From a modeling perspective, it excludes initial vacuum regions: agents are present throughout the state space $\Tt^d$. Mathematically, since the torus is compact, continuous strict positivity implies
\[\inf_{x\in\Tt^d} m_0(x) > 0.\] 
This uniform lower bound is essential for obtaining global a priori estimates. In energy and monotonicity arguments, integrating the transport equation against the value function produces terms of the form $\int_{\Tt^d} u(0,x)\, m_0(x) \,dx$. Bounding $m_0$ away from zero ensures that this integral controls $\int_{\Tt^d} u(0,x) \,dx$.

\subsection{Examples of Admissible Hamiltonians}\label{sec:example}
In this subsection, we present examples of Hamiltonians~$H$ that simultaneously satisfy Assumptions \ref{onH.strict-monotone}, \ref{onH.powergrowth}, \ref{onH.high-order-p.has-no-m}, \ref{onH.remedy-bound}, \ref{onH.highest-order-p.has-no-txm}, \ref{onH.high-order-pm.has-no-tx}, and \ref{onH.Lagrangianlowerbound}, thereby showing that Corollary~\ref{strong.exists} is not vacuous, for arbitrary $m_0\in C(\Tt^d;\RR^+_0)$ and $u_T\in C^1(\Tt^d)$. Throughout, we consider $H$ of the form
\begin{equation}\label{Hamiltonian-split-form}
    H(t,x,p,m) = H_0(p) - f(m) + g(t,x,p,m),
\end{equation}
and observe that the following conditions~\ref{exH0}, \ref{exf}, and \ref{exg}, on $H_0\colon\RR^d\to \RR$, $f\colon\RR^+_0\to \RR$, and $g\colon Q\times \RR^d\times \RR^+_0 \to \RR$, are sufficient to ensure that $H$ in~\eqref{Hamiltonian-split-form} satisfies the aforementioned assumptions.\smallskip
\begin{itemize}
        \item[\myulabel{$\mathbf{(H_0)}$}{exH0}] The principal Hamiltonian $H_0\colon\RR^d\to\RR$ is continuously differentiable and convex.
       Moreover, it satisfies
        \begin{equation}\label{onH0.growth}
            \begin{aligned}
                |D_pH_0(p)| & \leq C(|p|^{\alpha-1}+1),\\
            H_0(p) & \geq \frac{1}{C}|p|^\alpha - C,
            \end{aligned}
        \end{equation}
        for some positive constant $C$.\smallskip
        
        \item[\myulabel{$\mathbf{(f)}$}{exf}] The principal coupling  $f\colon \RR^+_0\to \RR$ is continuous and strictly increasing. Moreover, it satisfies
        \begin{equation}\label{onf.growth}
            \frac{1}{C}m^{\beta-1} - C \leq f(m) \leq C(m^{\beta-1}+1),
        \end{equation}
        for some positive constant $C$. \smallskip
        
        \item[\myulabel{$\mathbf{(g)}$}{exg}] The lower-order term $g\colon Q\times\RR^d\times\RR^+_0\to \RR$ is measurable and the maps $(p,m)\mapsto g(t,x,p,m)$ and $(p,m)\mapsto D_p g(t,x,p,m)$ are continuous for a.e.~$(t,x)\in Q$. Moreover, the monotonicity condition
\begin{equation}
  \label{gmon}
   \begin{aligned}
        &\bigl(- g(t,x,p_1,m_1) + g(t,x,p_2,m_2)\bigr) (m_1 - m_2)\\ &\quad +
        \bigl( m_1 D_p g(t,x,p_1,m_1) - m_2 D_p g(t,x,p_2,m_2)\bigr) \cdot (p_1-p_2) \geq 0
   \end{aligned}
     \end{equation}
holds for a.e.~$(t,x)\in Q$ and all $(p_1,m_1),\, (p_2,m_2)\in\RR^d\times\RR^+_0$, and the upper bounds
        \begin{align}
            |g(t,x,p,m)| & \leq C\bigl(|p|^{\alpha\frac{(\beta-1)}{\beta}} + m^{\beta-1}+V(t,x)\bigr)^{1-1/C}, \label{ong-upper-growth}\\
            |D_p g(t,x,p,m)| & \leq C\bigl(|p|^{\alpha-1} + m^{\frac{(\alpha-1)}{\alpha}(\beta-1)}+V(t,x)^{\frac{\alpha-1}{\alpha}}\bigr), \label{onDpg-upper-growth}
        \end{align}
hold for a.e~$(t,x)\in Q$ and all $(p,m)\in\RR^d\times\RR^+_0$, for a positive constant $C$ and some $V\in L^{\beta'}(Q)$ with $V\geq 1$ a.e. Finally, we have the one-sided bound
\begin{equation}\label{gLagrange-bound}
    \esssup_{(t,x)\in Q} g(t,x,0,0) < \infty.
\end{equation} \smallskip
\end{itemize}
We state and justify this observation in a lemma below.
\begin{lemma}\label{lem:H0fg-implies}
    Consider the setting above and suppose that the conditions in~\ref{exH0}, \ref{exf}, and \ref{exg} hold. Then the Hamiltonian $H$ defined in \eqref{Hamiltonian-split-form} satisfies Assumptions \ref{onH.strict-monotone}, \ref{onH.powergrowth}, \ref{onH.high-order-p.has-no-m}, \ref{onH.remedy-bound}, \ref{onH.highest-order-p.has-no-txm}, \ref{onH.high-order-pm.has-no-tx}, and \ref{onH.Lagrangianlowerbound}, possibly with a redefined value of $C$.
\end{lemma}
\begin{proof} We verify the assumptions in order.

 \medskip\noindent\textbf{Assumption~\ref{onH.strict-monotone}:} The monotonicity condition~\eqref{hmon} is additive in the Hamiltonian, hence we may verify it for $H_0$, $-f$, and $g$ separately. The term $H_0(p)$ satisfies~\eqref{hmon} by convexity in $p$, the term $-f(m)$ satisfies~\eqref{hmon} because $f$ is non-decreasing, and the term $g$ satisfies~\eqref{hmon} by \eqref{gmon}. Finally, strictness in the case $m_1\neq m_2$ follows because $f$ is strictly increasing.

 \medskip\noindent\textbf{Assumption~\ref{onH.powergrowth}:} First, we use the lower bound in \eqref{onf.growth} and \eqref{ong-upper-growth} to get
 \[\begin{aligned}
     H(t,x,0,m) & = H_0(0) - f(m) + g(t,x,0,m) \\
     & \leq -\frac{1}{C}m^{\beta-1} + C + C(m^{\beta-1})^{1-1/C} + CV(t,x).
 \end{aligned}\]
 Now we absorb the term $(m^{\beta-1})^{1-1/C}$ into the first term using Young's inequality, thus conclude \eqref{eq:assH.upper.simple} by redefining $C$.
 
 Secondly, we use the gradient bound in \eqref{onH0.growth} and \eqref{onDpg-upper-growth} to get
 \[\begin{aligned}
     |D_pH(t,x,p,m)| & \leq |D_pH_0(p)| + |D_pg(t,x,p,m)| \\
     & \leq C|p|^{\alpha-1} + C + C\bigl(|p|^{\alpha-1} + m^{\frac{(\alpha-1)}{\alpha}(\beta-1)}+V(t,x)^{\frac{\alpha-1}{\alpha}}\bigr),
 \end{aligned}\]
    from which we immediately conclude \eqref{eq:assH.DpH.upper.simple} by redefining $C$.

Finally, we use the lower bound in \eqref{onH0.growth}, the upper bound in \eqref{onf.growth}, and \eqref{ong-upper-growth} to get
\[\begin{aligned}
     H(t,x,p,m) & = H_0(p) - f(m) + g(t,x,p,m) \\
     & \geq \frac{1}{C}|p|^{\alpha}-C - Cm^{\beta-1} - C - C\bigl(|p|^{\alpha\frac{(\beta-1)}{\beta}} + m^{\beta-1}+V(t,x)\bigr), 
 \end{aligned}\]
 where the last two terms follow from Young's inequality applied to \eqref{ong-upper-growth}. Finally,  
 we absorb the $|p|^{\alpha\frac{(\beta-1)}{\beta}}$ term into the first term using again Young's inequality, thus conclude \eqref{eq:assH.lower.simple} by redefining $C$.

  \medskip\noindent\textbf{Assumptions~\ref{onH.high-order-p.has-no-m}, \ref{onH.remedy-bound}, \ref{onH.highest-order-p.has-no-txm}, \ref{onH.high-order-pm.has-no-tx}:} 
In view of Remark~\ref{rmk:last-deviation-implies-the-earlier-two}, it is enough to obtain \eqref{eq:H.triple-deviation}. To this end, we use \eqref{ong-upper-growth} to get
\[\begin{aligned}
     & |H(t,x,p,m) - H(\bar{t},\bar{x},p,0) - H(\tilde{t},\tilde{x},0,m)| \\
     & \qquad = |-H_0(0)+f(0)+g(t,x,p,m)-g(\bar{t},\bar{x},p,0) - g(\tilde{t},\tilde{x},0,m)| \\
     & \qquad \leq |g(t,x,p,m)| + |g(\bar{t},\bar{x},p,0)| + |g(\tilde{t},\tilde{x},0,m)| + C \\
     & \qquad \leq C\bigl(|p|^{\alpha\frac{(\beta-1)}{\beta}} + m^{\beta-1} + V(t,x)\bigr)^{1-1/C} \\
     & \qquad\qquad + C\bigl(|p|^{\alpha\frac{(\beta-1)}{\beta}} + V(\bar{t},\bar{x})\bigr)^{1-1/C} + C\bigl(m^{\beta-1} + V(\tilde{t},\tilde{x})\bigr)^{1-1/C} + C, \\
     & \qquad \leq C\bigl(|p|^{\alpha\frac{(\beta-1)}{\beta}}+m^{\beta-1} + V(t,x)+ V(\bar{t},\bar{x}) + V(\tilde{t},\tilde{x})\bigr)^{1-1/C},
 \end{aligned}\]
 yielding the desired result.
 
 \medskip\noindent\textbf{Assumption~\ref{onH.Lagrangianlowerbound}:} It immediately follows from \eqref{gLagrange-bound}.
\end{proof}

In view of Lemma~\ref{lem:H0fg-implies}, we now provide examples of $g\colon Q\times\RR^d\times\RR^+_0\to \RR$ that satisfy \eqref{gmon}, \eqref{ong-upper-growth}, \eqref{onDpg-upper-growth}, and \eqref{gLagrange-bound}, 
because standard examples for $H_0$ and $f$ are immediate.
Note that all four of these conditions are additive. Hence, the examples can also be additively combined.

\begin{example}[Separable Hamiltonian]\label{exm:sep}
    Choose $\sigma$ such that 
    \[0 < \sigma < \beta - 1\] 
    and choose $w_1\in L^\infty(Q;\RR^+_0)$ and $w_2\in L^{\beta'+\epsilon}(Q;\RR^+_0)$ for some $\epsilon > 0$. Then
    \[g(t,x,p,m) = -w_1(t,x)m^{\sigma} - w_2(t,x)\]
satisfies \eqref{gmon}, \eqref{ong-upper-growth}, \eqref{onDpg-upper-growth}, and \eqref{gLagrange-bound}. Indeed, \eqref{gmon} follows from the monotonicity of $m\mapsto m^\sigma$, \eqref{onDpg-upper-growth} is immediate since $D_pg=0$, and \eqref{gLagrange-bound} follows from $g(t,x,0,0)=-w_2(t,x)\leq 0$. 
For \eqref{ong-upper-growth}, take
\[
V(t,x)=(w_2(t,x)+1)^{(\beta'+\epsilon)/\beta'},
\]
so that $V\geq 1$ and $V\in L^{\beta'}(Q)$, the latter because $w_2\in L^{\beta'+\epsilon}(Q)$. 
Writing \( |g|=w_1m^\sigma+w_2 \), choose \(C\) sufficiently large so that
\(1-\frac1C\geq \frac{\sigma}{\beta-1}\) and
\(1-\frac1C\geq \frac{\beta'}{\beta'+\epsilon}\). Then
\[
w_1m^\sigma
\leq C(m^{\beta-1}+V)^{1-1/C},
\qquad
w_2\leq V^{1-1/C},
\]
where we used \(w_1\in L^\infty(Q;\RR^+_0)\), \(V\ge1\), and
\(\sigma<\beta-1\). Hence \eqref{ong-upper-growth} follows.
\end{example}

\begin{example}[Non-separable Hamiltonian]\label{exm:nonsep}
Assume that
\[\frac{1}{\alpha}+\frac{1}{\beta}<1. 
\]
Choose $\sigma$ and $\tau$ such that 
\[1 < \sigma < \alpha\frac{(\beta-1)}{\beta}, \qquad\qquad 0<\tau\leq \frac{4(\sigma-1)}{\sigma}.\]
Moreover, let
\[\rho := \frac{\alpha\beta}{\alpha(\beta-1)-\sigma\beta}\]
and choose $w\in L^{\rho+\epsilon}(Q;\RR^+_0)$ for some $\epsilon > 0$. Then, as we check below, 
\[g(t,x,p,m) = w(t,x)\frac{|p|^\sigma}{(1+m)^\tau}\]
satisfies \eqref{gmon}, \eqref{ong-upper-growth}, \eqref{onDpg-upper-growth}, and \eqref{gLagrange-bound}.
We note that \eqref{gLagrange-bound} is immediate from $g(t,x,0,0)=0$.
Thus, it remains to 
check \eqref{gmon},  \eqref{ong-upper-growth} and \eqref{onDpg-upper-growth}. 

\medskip\noindent\textbf{Proof of \eqref{gmon}:} We want to prove
\begin{equation*}
   \begin{aligned}
        & w(t,x)\biggl(- \frac{|p_1|^\sigma}{(1+m_1)^\tau} + \frac{|p_2|^\sigma}{(1+m_2)^\tau}\biggr) (m_1 - m_2)\\ &\quad +
        w(t,x)\sigma\biggl( m_1\frac{|p_1|^{\sigma-2}p_1}{(1+m_1)^\tau} - m_2\frac{|p_2|^{\sigma-2}p_2}{(1+m_2)^\tau}\biggr) \cdot (p_1-p_2) \geq 0,
   \end{aligned}
     \end{equation*}
for a.e.~$(t,x)\in Q$ and all $(m_1, p_1),\, (m_2, p_2) \in \RR^+_0\times \RR^d$. 
Since \(w(t,x)\ge0\), it suffices to prove the corresponding inequality for the factor independent of \(w\). Define \(A\colon \RR^+_0\times \RR^d \to \RR^{d+1}\) by
\[
A(m,p) =
\begin{cases}
\displaystyle
\frac{|p|^{\sigma-1}}{(1+m)^\tau}\biggl(-|p|,\, \sigma \frac{m}{|p|}p \biggr), & p\neq0,\\[2mm]
0, & p=0.
\end{cases}
\]
We want to prove
\begin{equation}\label{gmon-verify}
    \bigl\langle A(m_1,p_1) - A(m_2,p_2),\ (m_1,p_1) - (m_2,p_2) \bigr\rangle \geq 0
\end{equation}
for all $(m_1, p_1),\, (m_2, p_2) \in \RR^+_0\times \RR^d$, with respect to the standard bilinear pairing of $\RR^{d+1}$ with itself. For $p_1 = p_2 = 0$, \eqref{gmon-verify} is trivial. Otherwise, we use the Fundamental Theorem of Calculus to express
\begin{equation}\label{gmon-ftc}
    \begin{aligned}
    & \left\langle A(m_1,p_1) - A(m_2,p_2),\ (m_1,p_1) - (m_2,p_2) \right\rangle \\
    & \quad = \int_0^1 (m_1-m_2,p_1-p_2)^T DA\bigl(\lambda m_1 + (1-\lambda)m_2, \lambda p_1 + (1-\lambda)p_2\bigr) (m_1-m_2,p_1-p_2)\,d\lambda.
\end{aligned}
\end{equation}
Indeed, the map $A\colon \RR^+_0\times \RR^d \to \RR^{d+1}$ is continuous everywhere and it is $C^1$ away from the ray $p = 0$, hence \eqref{gmon-ftc} holds if $\lambda p_1 + (1-\lambda) p_2 \neq 0$ for all $0\leq \lambda \leq 1$. Moreover, it is easy to see that \eqref{gmon-ftc} holds even if $\lambda p_1 + (1-\lambda) p_2 = 0$ for exactly one value of $\lambda$, because then the integrand is still an integrable function even as it may diverge around the point of discontinuity. Consequently, it remains to show that the symmetrized Jacobian of $A$ is positive semi-definite, that is
\begin{equation}\label{sym-jac-positive}
    \frac{1}{2}\bigl(DA(m,p) + DA^T(m,p)\bigr) \geq 0
\end{equation}
for all $(m,p)$ with $p\neq 0$.

For \eqref{sym-jac-positive}, we directly compute
\[\frac{1}{2}\bigl(DA(m,p) + DA^T(m,p)\bigr) = \frac{|p|^{\sigma-2}}{(1+m)^{\tau+1}}\begin{pmatrix}
    \tau|p|^2 & -\frac{\sigma\tau}{2} mp^T \\[2mm]
    -\frac{\sigma\tau}{2} mp & \sigma m(1+m)\bigl(I + (\sigma-2)\frac{p\otimes p}{|p|^2}\bigr)
\end{pmatrix},\]
hence we want to show 
\[\begin{pmatrix}
    \tau|p|^2 & -\frac{\sigma\tau}{2} mp^T \\[2mm]
    -\frac{\sigma\tau}{2} mp & \sigma m(1+m)\bigl(I + (\sigma-2)\frac{p\otimes p}{|p|^2}\bigr)
\end{pmatrix} \geq 0.\]
The top left term is positive, hence we need to check the positive semi-definiteness of its Schur complement:
\[\begin{aligned}
    & \sigma m(1+m)\biggl(I + (\sigma-2)\frac{p\otimes p}{|p|^2}\biggr) - \bigl(-\frac{\sigma\tau}{2} mp\bigr)\biggl(\frac{1}{\tau|p|^2}\biggr)\bigl(-\frac{\sigma\tau}{2} mp^T\bigr) \\
    & = \sigma m(1+m)\biggl(I + (\sigma-2)\frac{p\otimes p}{|p|^2}\biggr) - \frac{\sigma^2\tau}{4}m^2\frac{p\otimes p}{|p|^2} \\
    & = \sigma m\biggl((1+m)I + \bigl((\sigma-2)(1+m)-\frac{\sigma\tau}{4}m\bigr)\frac{p\otimes p}{|p|^2}\biggr).
\end{aligned}\]
If $m=0$, the Schur complement is identically zero. If $m>0$, we cancel the factor of $\sigma m$ and see that the $d$-dimensional matrix inside the parenthesis has eigenvalues $(1+m) > 0$ in the directions orthogonal to $p$ and
\begin{equation}\label{radial-eigenvalue}
    (\sigma-1)(1+m) - \frac{\sigma\tau}{4}m = \sigma - 1 + \Bigl(\sigma - 1 - \frac{\sigma\tau}{4}\Bigr)m > 0
\end{equation}
in the direction of $p$, hence the proof is complete. Note that the inequality in \eqref{radial-eigenvalue} follows from the choice of $\sigma$ and $\tau$.

\medskip\noindent\textbf{Proof of \eqref{ong-upper-growth}:} We first observe
\[\begin{aligned}
    |g(t,x,p,m)| \leq w(t,x)|p|^\sigma & \leq w(t,x)^{\frac{\rho+(1/2)\epsilon}{\beta'}} + |p|^{\sigma\frac{\rho+(1/2)\epsilon}{\rho+(1/2)\epsilon - \beta'}} \\
    & \leq C\bigl((w(t,x)+1)^{\frac{\rho+\epsilon}{\beta'}} + |p|^{\frac{\sigma\rho}{\rho - \beta'}}\bigr)^{1-1/C}.
\end{aligned}
\]
Then, we compute
\[\frac{\sigma\rho}{\rho - \beta'} = \alpha\frac{(\beta-1)}{\beta}\]
and take $V(t,x) := (w(t,x)+1)^{\frac{\rho+\epsilon}{\beta'}} \in L^{\beta'}(Q)$, with $V\ge1$. This proves \eqref{ong-upper-growth}.

\medskip\noindent\textbf{Proof of 
\eqref{onDpg-upper-growth}:} Using \((1+m)^{-\tau}\leq 1\), we have
\[
|D_pg(t,x,p,m)|
\leq Cw(t,x)|p|^{\sigma-1}.
\]
Since
\[
\sigma<\alpha\frac{\beta-1}{\beta}<\alpha,
\]
Young's inequality with exponents
\[
\frac{\alpha-1}{\sigma-1}
\qquad\text{and}\qquad
\frac{\alpha-1}{\alpha-\sigma}
\]
gives
\[
w(t,x)|p|^{\sigma-1}
\leq C|p|^{\alpha-1}
   + Cw(t,x)^{\frac{\alpha-1}{\alpha-\sigma}}.
\]
We now compare the last exponent with the potential \(V\) chosen above. Since
\[
\frac{\alpha\beta'}{\alpha-\sigma}
=
\frac{\alpha\beta}{(\beta-1)(\alpha-\sigma)}
<
\frac{\alpha\beta}{\alpha(\beta-1)-\sigma\beta}
=\rho,
\]
we have
\[
\frac{\alpha-1}{\alpha-\sigma}
<
\frac{\alpha-1}{\alpha}\frac{\rho+\epsilon}{\beta'}.
\]
Thus, because \(V=(w+1)^{(\rho+\epsilon)/\beta'}\) and \(V\ge1\),
\[
w(t,x)^{\frac{\alpha-1}{\alpha-\sigma}}
\leq C(w(t,x)+1)^{\frac{\alpha-1}{\alpha}\frac{\rho+\epsilon}{\beta'}}
= C V(t,x)^{\frac{\alpha-1}{\alpha}}.
\]
Consequently,
\[
|D_pg(t,x,p,m)|
\leq C\left(|p|^{\alpha-1}+V(t,x)^{\frac{\alpha-1}{\alpha}}\right),
\]
which is stronger than \eqref{onDpg-upper-growth}.

\end{example}

\subsection{Monotone Operator Theory}\label{sec:monotone}

Because first-order MFG systems lack the regularizing diffusion that provides classical compactness, standard techniques and classical PDE theory fall short. Instead, we cast our problem into the framework of monotone operators. This theory handles non-separable Hamiltonians natively. Consequently, the proofs of existence of solutions that we present in the subsequent sections rely on an abstract existence theorem for monotone operators on reflexive Banach spaces, Theorem~\ref{thm:monotone.abstract-weak} below. Before stating this theorem, we recall the notions of monotonicity, coercivity, and hemicontinuity used in our analysis. These properties are standard in the theory of variational inequalities; see for example \cite{KiSt00}. A prototypical example of an operator satisfying these properties is the $p$-Laplacian, which we will directly employ as our core regularization tool in Section~\ref{sec:proof-weak-exists}.

\begin{definition}\label{def:prelim}
  Let \(X\) be a reflexive Banach space, and let \(X'\) denote its dual. Fix an element \(z_0\) in a vector space containing \(X\), and let \(\mathcal{K}\subset z_0 + X\) be a nonempty convex set. Let \(A\colon \mathcal{K} \to X' \) be an operator. We say that:
  \begin{enumerate}[label=(\alph*)]
\item $A$ is monotone if for all $v_1, v_2\in \mathcal{K}$, we have \[\langle A[v_1] - A[v_2], v_1 - v_2 \rangle_{X', X} \geq 0.\]
  \item\label{defitem:coercive} $A$ is coercive if there exists $\bar v \in \mathcal{K}$ such that
\[\lim_{\substack{v\in \mathcal{K} \\ \lVert v - \bar v\rVert_X \to \infty}} \frac{\langle A[v] - A[\bar v], v - \bar v \rangle_{X', X}}{\lVert v - \bar v \rVert_X} = +\infty.\]
   \item\label{defitem:hemicont} $A$ is hemicontinuous if for all $u, v\in \mathcal{K}$ and $w\in X$, the real-valued map
   \[\lambda\mapsto \langle A[(1-\lambda)u + \lambda v], w \rangle_{X', X}\]
   is continuous in the interval $[0,1]$.
  \end{enumerate}
\end{definition}

With these notions in place, we introduce the weak and strong formulations of the
variational inequality associated with $(A,\mathcal{K})$.

\begin{definition}\label{def:abstract-vi}
  In the setting of Definition~\ref{def:prelim}, an element $u\in\overline{\mathcal{K}}$ is
  a solution of the \emph{weak variational inequality} for $(A,\mathcal{K})$ if
  \begin{equation}\label{eq:weakIneqA}
      \langle A[v],\, v-u\rangle_{X',X}\geq 0 \qquad\text{for all } v\in\mathcal{K},
  \end{equation}
  and an element $u\in\mathcal{K}$ is a solution of the \emph{strong variational inequality}
  for $(A,\mathcal{K})$ if
  \begin{equation}\label{eq:strongVI.IneqA}
      \langle A[u],\, v-u\rangle_{X',X}\geq 0 \qquad\text{for all } v\in\mathcal{K}.
  \end{equation}
  Here \(\overline{\mathcal{K}}\) denotes the closure of \(\mathcal{K}\) in \(z_0 + X\).
\end{definition}

\begin{remark}\label{rmk:hemicontinuous-upgrade}
By monotonicity, every strong solution in Definition~\ref{def:abstract-vi} is a weak solution. For hemicontinuous $A$ and weak solutions in $\mathcal{K}$, we also have the converse. Indeed, using \eqref{eq:weakIneqA} with $v$ replaced by $(1-\lambda)u+\lambda w$ for any $w \in \mathcal{K}$ and $0<\lambda\leq 1$, subsequently dividing the resulting inequality by $\lambda$, and letting $\lambda \to 0$, we obtain $\langle A[u], w-u\rangle_{X',X} \geq 0$. Relabeling $w$ as $v$, we conclude that $u$ satisfies 
  the strong variational inequality \eqref{eq:strongVI.IneqA} of Definition~\ref{def:abstract-vi}.
\end{remark}

The cornerstone of our analysis is the following existence theorem. It guarantees solutions to weak variational inequalities under minimal assumptions by extending the classical Debrunner--Flor theorem to unbounded domains.
\begin{theorem}[{cf.~\cite[Theorem~2.16]{ferreiraSolvingMeanFieldGames2025}}]
\label{thm:monotone.abstract-weak}
Let \(X\) be a reflexive Banach space, and let \(X'\) denote its dual. Fix an element \(z_0\) in a vector space containing \(X\), and let \(\mathcal{K} \subset z_0 + X\) be a nonempty convex subset. Assume that the operator \(A\colon \mathcal{K} \to X' \) is monotone and coercive. Then, there exists a solution $u\in\overline{\mathcal{K}}$ of the weak variational
inequality, namely
\begin{equation*}
\begin{aligned}\forall\, v\in \mathcal{K}, \enspace \langle A[v], v-u\rangle_{X',X} \geq 0,
\end{aligned}
\end{equation*}
where \(\overline{\mathcal{K}}\) denotes the closure of \(\mathcal{K}\) in \(z_0 + X\). Assume further that $A$ is hemicontinuous and $\mathcal{K}$ is closed. Then, there exists a solution $u\in\mathcal{K}$ of the strong variational
inequality, namely
\begin{equation*}
\begin{aligned}
\forall\, v\in \mathcal{K}, \enspace \langle A[u], v-u\rangle_{X',X} \geq 0.
\end{aligned}
\end{equation*}
\end{theorem}
\begin{remark}
    The affine case of the preceding theorem reduces to the linear setting of \cite[Theorem~2.16]{ferreiraSolvingMeanFieldGames2025} by the substitution \(\widetilde{v} = v - z_0\), \(\widetilde{\mathcal{K}} = \mathcal{K} - z_0 \subset X\), \(\widetilde{A}[\widetilde{v}] = A[\widetilde{v} + z_0]\). The second statement in the theorem follows from the first statement and Remark~\ref{rmk:hemicontinuous-upgrade}.
\end{remark}
  
  For comparison of the above notions to the literature, recall that a mapping $A\colon \mathcal{K} \to X'$ is  \emph{continuous on finite-dimensional subspaces} if, for every finite-dimensional subspace $\mathcal{M} \subset X$, the restriction $A|_{\mathcal{K} \cap (z_0+\mathcal{M})}$ is continuous from $\mathcal{K} \cap (z_0+\mathcal{M})$ into $X'$ endowed with the weak topology. Clearly, this property, and even the weaker property of continuity on one-dimensional subspaces, implies hemicontinuity. In particular, the existence of solutions to the strong variational inequality established in \cite[Thm.~1.7]{KiSt00}, under the assumption of continuity on finite-dimensional subspaces, follows as a consequence of Theorem~\ref{thm:monotone.abstract-weak}. Moreover, for monotone mappings, hemicontinuity is known to imply continuity on finite-dimensional subspaces so that the two properties are, in fact, equivalent; see, e.g., \cite{MR163198}.

\subsection{Measures and Functions of Bounded Variation}\label{sec:meas-bv}

This section collects some technical results of analysis that we use in our proofs of the main theorems in the subsequent sections. These results would be conventionally formulated for an open set $U$ in some Euclidean space $\RR^m$, but they are valid for any manifold with sufficient structure, in particular the flat torus $\Tt^d$ and the space-time cylinder $Q = (0,T)\times\Tt^d$, which is the setting relevant to our applications. Thus, throughout this
section, we use \emph{reference domain} as a shorthand for either an open subset of a Euclidean space or an open subset of $\Tt^d$ or $Q$.

The first two results we provide concern certain special approximations in a.e.~convergence; we are not aware of references for these exact formulations, thus we give the full proofs. The first of these is about the approximation of Borel measurable functions by continuous functions with prescribed values on a given compact set.

\begin{lemma}\label{lem:dense.restr-cont}
  Let $B_1, B_2, \ldots, B_k\subset U$ be Borel sets in a reference domain $U$, let $\rho_i$ be a finite, non-negative Borel measure on $B_i$, and let $f$ be a non-negative Borel measurable function on $B := B_1\,\cup \ldots \cup\,B_k$. Given a compact set $K \subset U \setminus B$ and a non-negative continuous function $g$ on $K$, there exists a sequence $\{f_n\}$ of non-negative continuous functions on $U$ such that
  \begin{itemize}
     \item $f_n = g$ on $K$,
     \item $f_n$ converges to $f$ on $B_i$ a.e.~with respect to $\rho_i$.
  \end{itemize}
  Moreover, if $f$ is bounded, we may take the sequence $\{f_n\}$ to be uniformly bounded.
\end{lemma}
\begin{proof}
  Define the unified non-negative finite Radon measure $\rho$ on $U$ by \[\rho(A) = \sum_{i=1}^k \rho_i(A\cap B_i) \quad\text{for any Borel set } A\subset U,\]
  and extend $f$ to a non-negative Borel measurable function on all of $U$ by setting $f = 0$ outside $B$. Since $K \subset U \setminus B$, we clearly have $\rho(K) = 0$. We will construct the desired sequence as $f_n := \Tilde{f}_n + g_n$, where $\{\Tilde{f}_n\}$ and $\{g_n\}$ are sequences of non-negative continuous functions satisfying
  \begin{align*}
     \Tilde{f}_n &= 0 \text{ on } K \qquad \text{and} \qquad \Tilde{f}_n \to f \text{ a.e.~with respect to } \rho, \\
     g_n &= g \text{ on } K \qquad \text{and} \qquad g_n \to 0 \text{ a.e.~with respect to } \rho.
  \end{align*}

  \medskip\noindent\textbf{Construction of $\{g_n\}$:}
  For each $n \geq 1$, define the open neighborhood
  \[U_n := \{x \in U \colon \operatorname{dist}(x, K) < 1/n\}.\]
  Because $K$ and $U \setminus U_n$ are disjoint closed sets, we can define a continuous function on their union that equals $g$ on $K$ and $0$ on $U \setminus U_n$. By the Tietze Extension Theorem, this can be extended to a non-negative continuous function $g_n$ on $U$. For any $x \notin K$, we have $x \notin U_n$ for all sufficiently large $n$, meaning $g_n(x) \to 0$ everywhere on $U \setminus K$. Since $\rho(K)=0$, it immediately follows that $g_n \to 0$ a.e.~with respect to $\rho$.

\medskip\noindent\textbf{Construction of $\{\Tilde{f}_n\}$:}
We first establish global approximation in $\rho$-measure. Suppose initially that $f = \chi_A$ for a Borel set $A \subset U$. Given $\epsi > 0$, outer regularity provides an open neighborhood $V \supset A$ such that $\rho(V \setminus A) < \epsi/2$. We then isolate our construction from the set $K$ by defining the open set $W := V \setminus K$. Inner regularity guarantees a compact core $\Tilde{K} \subset W$ such that $\rho(W \setminus \Tilde{K}) < \epsi/2$, which implies $\rho(V \setminus \Tilde{K}) < \epsi/2$ since $\rho(K) = 0$.
  
By Urysohn's Lemma, there exists a continuous function $\Tilde{f}\colon U \to [0,1]$ such that $\Tilde{f} = 1$ on $\Tilde{K}$ and $\supp \Tilde{f} \subset W = V \setminus K$. In particular, this ensures $\Tilde{f} = 0$ on $K$. Furthermore, we observe that
\[
    \{x\in U\colon |f(x)-\Tilde{f}(x)| > \epsi\} \subset \{x\in U\colon \chi_A(x) \neq \Tilde{f}(x)\} \subset (V\setminus A) \cup (V\setminus \Tilde{K}).
\]
The sum of the $\rho$-measures of these two sets is strictly less than $\epsi$, which proves that $\Tilde{f}$ approximates $\chi_A$ globally in $\rho$-measure.

By linearity, this global approximation in $\rho$-measure extends from indicator functions to non-negative simple functions. Since any non-negative Borel measurable function is the pointwise limit of an increasing sequence of simple functions, the approximation in $\rho$-measure extends to $f$ as well. Finally, because global convergence in measure guarantees the existence of a subsequence that converges almost everywhere, we can directly extract the required sequence $\{\Tilde{f}_n\}$ that vanishes on $K$ and converges to $f$ a.e.~with respect to $\rho$.

  \medskip\noindent\textbf{Uniform Boundedness:}
  If the original function $f$ is bounded, the generated sequence $\{f_n\}$ can simply be replaced by $\min(f_n,\, \max_{K} g + \sup f + 1)$. This truncation ensures uniform boundedness while preserving continuity, the exact boundary values on $K$, and the a.e.\ convergence to $f$.
\end{proof}

The second approximation result concerns certain Carath\'eodory-type compositions and space-time convolutions. To formulate this result, we fix non-negative, compactly supported smooth functions $\overline{\zeta}$ and $\hat{\zeta}$ on $\RR^d$ and $\RR^+$, respectively, such that
\begin{equation*}
    \int_{\RR^d} \overline{\zeta} \,dx = 1 \qquad\text{and}\qquad \int_0^\infty \hat{\zeta}\,dt = 1.
\end{equation*}
For definiteness, we take $\supp\overline{\zeta}$ to be inside the unit ball around the origin and $\supp\hat{\zeta}$ to be inside the interval $(0,1)$. Then, for small $\xi, \tau > 0$, we define the mollification kernel
\begin{equation}\label{kernel-mollification}
    K_{\xi, \tau}(t',x'; t,x) := \frac{1}{\tau} \hat{\zeta}\left(\frac{t'-t}{\tau}\right)\cdot \frac{1}{\xi^d} \overline{\zeta}\left(\frac{x'-x}{\xi}\right),
\end{equation}
for $0 \leq t < t'$ and $x, x' \in \Tt^d$. Considering $K_{\xi, \tau}(t',x'; t,x)$ as a function of $(t',x')$ for a fixed $(t,x)$, its support is contained in the cylinder
\begin{equation}\label{eq:cylinders}
    C_{\xi,\tau}(t,x) := [t, t+\tau)\times B_\xi(x).
\end{equation}
Finally, we note that
\begin{equation}\label{K-int-C-vol:formula}
    \int_t^\infty\int_{\Tt^d} K_{\xi, \tau}(t',x';t,x)\,dx'dt' = 1 \qquad\text{and}\qquad |C_{\xi,\tau}(t,x)| = \tau\xi^d|B_1(0)|.
\end{equation}
Before stating the approximation result in Lemma~\ref{lem:mollification-convergence-with-adapted.Lebesgue} below, we give an auxiliary estimate to be subsequently used in its proof.

\begin{lemma}[Weak Type $(1,1)$ Estimate for Adapted Maximal Function]\label{lem:HL-max-analogue-weak-1-1}
    Recall $C_{\xi,\tau}(t,x)$ defined in \eqref{eq:cylinders}.
     Let $\{\xi_n\}$ and $\{\tau_n\}$ be decreasing, positive sequences converging to $0$. For any $v\in L^1((0,T+\tau_1)\times \Tt^d; \RR^+_0)$, let $M[v]$ denote the non-negative, extended real-valued, measurable function on $Q$ defined by
    \begin{equation*}
        M[v](t,x) := \sup_n \frac{1}{|C_{{\xi_n}, {\tau_n}}(t,x)|}  \int_{C_{{\xi_n}, {\tau_n}}(t,x)} v(t',x') \,dx'dt'.
    \end{equation*}
    Then the following weak type $(1,1)$ estimate holds:
    \begin{equation}\label{HL-max-analogue-weak-1-1}
        \bigl|\bigl\{(t,x)\in Q\colon M[v](t,x) > \delta \bigr\}\bigr| \leq 3^{d+1}\cdot \frac{\|v\|_{L^1}}{\delta} \qquad \text{for any }\delta>0.
    \end{equation}
\end{lemma}
\begin{proof}
    Let $\mathcal{C}$ be the collection of all cylinders $C_{{\xi_n}, {\tau_n}}(t,x)$ such that 
    \[\frac{1}{|C_{{\xi_n}, {\tau_n}}(t,x)|}  \int_{C_{{\xi_n}, {\tau_n}}(t,x)} v(t',x') \,dx'dt' > \delta.\]
    Clearly we have
    \[\bigl\{(t,x)\in Q\colon M[v](t,x) > \delta \bigr\} \subset \bigcup_{C\in \mathcal{C}} C.\]
    Hence, it suffices to show that the measure of the latter union is bounded by $3^{d+1}\bigl(\|v\|_{L^1}/\delta\bigr)$.

    We follow the standard Vitali argument and 
    construct a subcollection $\tilde{\mathcal{C}} \subset \mathcal{C}$ of pairwise disjoint cylinders via a greedy extraction algorithm. 
    Because the sequences $\{\xi_n\}$ and $\{\tau_n\}$ are decreasing, the cylinders in $\mathcal{C}$ take sizes from a discrete, ordered set indexed by $n$.
    We proceed iteratively: for $n=1, 2, \dots$, we select a maximal pairwise disjoint subcollection of cylinders of size index $n$ from those remaining in $\mathcal{C}$, add them to $\tilde{\mathcal{C}}$, 
    and discard from the remaining collection any unselected cylinder, of any size index, that intersects a newly added cylinder.
    (Note that each selected subcollection is finite because the total volume of $(0,T+\tau_1)\times\Tt^d$ is bounded, making this maximal selection straightforward.)
    
    This procedure guarantees that $\tilde{\mathcal{C}}$ is a pairwise disjoint collection. Furthermore, it ensures a size-dominated covering property: if a cylinder $C \in \mathcal{C}$ is not in $\tilde{\mathcal{C}}$, it was discarded because it intersected some cylinder $\tilde{C} := C_{{\xi_m},{\tau_m}}(s,y) \in \tilde{\mathcal{C}}$ that was processed at the same or an earlier step. Thus, $\tilde{C}$ is at least as large as $C$. Since $C\cap \tilde{C} \neq \emptyset$ and $|C| \leq |\tilde{C}|$, the spatial projection of $C$ is contained in $B_{3\xi_m}(y)$ and its temporal projection is contained in $(s-\tau_m, s+2\tau_m)$. 
Thus we have $C\subset (s-\tau_m, s+2\tau_m)\times B_{3\xi_m}(y)$, and by the doubling property of balls on the flat torus, the measure of the latter cylinder is at most $3^{d+1}|\tilde{C}|$.
    
    Because every cylinder in $\mathcal{C}$ is contained within the $3$-times dilated version of some cylinder in $\tilde{\mathcal{C}}$, we conclude
    \[\biggl|\bigcup_{C\in \mathcal{C}} C\biggr| \leq \sum_{\tilde{C}\in\tilde{\mathcal{C}}} 3^{d+1}|\tilde{C}| \leq 3^{d+1} \sum_{\tilde{C}\in\tilde{\mathcal{C}}} \frac{1}{\delta}\int_{\tilde{C}} v(t',x')\,dx'dt' \leq 3^{d+1}\cdot \frac{\|v\|_{L^1}}{\delta},\]
    where the last inequality holds because $\tilde{\mathcal{C}}$ is a disjoint collection.
\end{proof}

We now use this weak-type estimate to obtain the corresponding one-sided Lebesgue differentiation result for the space-time mollifiers used below.

\begin{lemma}\label{lem:mollification-convergence-with-adapted.Lebesgue}
    Let $\sigma \geq 1$, let $\theta\in L^\sigma((0,\infty)\times\Tt^d; \RR^d)$, and let $F\colon Q\times\RR^d\to \RR$ be a measurable function such that
\begin{equation}\label{adapted.lebesgue-caratheodory}
        \text{the map}\quad s\mapsto \frac{F(t,x,s)}{|s|^\sigma+1} \quad \text{is continuous and bounded on } \RR^d,
    \end{equation}
     for a.e.~$(t,x)\in Q$. 
     Consider decreasing sequences $\xi_n\downarrow 0$ and $\tau_n\downarrow 0$ as in Lemma \ref{lem:HL-max-analogue-weak-1-1}. Then,
     \begin{equation}\label{half-convolution-lebesgue-point-approximation}
         \lim_{n\to\infty} \int_t^\infty\int_{\Tt^d} F(t,x,\theta(t',x'))K_{{\xi_n}, {\tau_n}}(t',x';t,x)\,dx'dt' = F(t,x,\theta(t,x))
     \end{equation}
    for a.e.~$(t,x)\in Q$.
\end{lemma}
\begin{proof}
    We first prove that a.e.~$(t,x)\in Q$ is an $L^\sigma$-Lebesgue point of $\theta$ with respect to the family of cylinders $C_{{\xi_n}, {\tau_n}}(t,x)$. More precisely, we have
    \begin{equation}\label{xi-tau-n.adapted.Lebesgue}
        \lim_{n\to\infty} \frac{1}{|C_{{\xi_n}, {\tau_n}}(t,x)|} \int_{C_{{\xi_n}, {\tau_n}}(t,x)} |\theta(t',x')-\theta(t,x)|^\sigma \,dx'dt' = 0
    \end{equation}
    for a.e.~$(t,x)\in Q$. 
The argument follows the steps of the classical proof of the Lebesgue differentiation theorem and relies on the variant of the Hardy--Littlewood maximal operator from Lemma~\ref{lem:HL-max-analogue-weak-1-1}.
Accordingly, to prove \eqref{xi-tau-n.adapted.Lebesgue}, it is enough to show
    \begin{equation}\label{lim-neq-zero-means-limsup-geq-delta}
        \Bigl|\Bigl\{(t,x)\in Q\colon \limsup  \frac{1}{|C_{{\xi_n}, {\tau_n}}(t,x)|} \int_{C_{{\xi_n}, {\tau_n}}(t,x)} |\theta(t',x')-\theta(t,x)|^\sigma \,dx'dt' > 2^\sigma\delta \Bigr\}\Bigr| = 0
    \end{equation}
    for any $\delta > 0$. For any continuous $f\colon (0,T+\tau_1)\times\Tt^d \to \RR^d$, we have 
    \[\lim_{n\to\infty} \frac{1}{|C_{{\xi_n}, {\tau_n}}(t,x)|} \int_{C_{{\xi_n}, {\tau_n}}(t,x)} |f(t',x')-f(t,x)|^\sigma \,dx'dt' = 0.\]
    By Minkowski's inequality, outside the union
    \[\bigl\{(t,x)\in Q\colon |\theta-f|^\sigma > \delta \bigr\} \cup \bigl\{(t,x)\in Q\colon M[|\theta-f|^\sigma](t,x) > \delta \bigr\},\]
    the limsup in \eqref{lim-neq-zero-means-limsup-geq-delta} is at most $2^\sigma\delta$. Therefore, the set in \eqref{lim-neq-zero-means-limsup-geq-delta} is contained in this union.
    Using \eqref{HL-max-analogue-weak-1-1}, we bound the measure of this union by
    \[\frac{3^{d+1}+1}{\delta}\int_0^{T+\tau_1}\int_{\Tt^d} |\theta -f|^\sigma \,dxdt.\]
The latter quantity can be made arbitrarily small by choosing continuous $f$ to approximate $\theta$ in $L^\sigma((0,T+\tau_1)\times\Tt^d)$.
     Hence, we obtain \eqref{lim-neq-zero-means-limsup-geq-delta}.
    
   We finish the proof by showing that any $(t,x)\in Q$ which satisfies \eqref{adapted.lebesgue-caratheodory} and \eqref{xi-tau-n.adapted.Lebesgue} also satisfies \eqref{half-convolution-lebesgue-point-approximation}.
   Indeed, by \eqref{adapted.lebesgue-caratheodory}, the map \(s\mapsto F(t,x,s)\) is continuous and has at most \(\sigma\)-growth. Hence, by a standard truncation argument combined with \eqref{xi-tau-n.adapted.Lebesgue}, we obtain
   \begin{equation}\label{lebesgue-point-at-composition}
       \lim_{n\to\infty} \frac{1}{|C_{{\xi_n}, {\tau_n}}(t,x)|} \int_{C_{{\xi_n}, {\tau_n}}(t,x)} |F(t,x,\theta(t',x'))-F(t,x,\theta(t,x))| \,dx'dt' = 0.
   \end{equation}
   On the other hand, using \eqref{K-int-C-vol:formula}, we get
   \[\begin{aligned}
       & \left|\int_t^\infty\int_{\Tt^d} F(t,x,\theta(t',x'))K_{{\xi_n}, {\tau_n}}(t',x';t,x)\,dx'dt' - F(t,x,\theta(t,x))\right| \\
       & \quad = \left|\int_t^\infty\int_{\Tt^d} \bigl(F(t,x,\theta(t',x')) - F(t,x,\theta(t,x))\bigr)K_{{\xi_n}, {\tau_n}}(t',x';t,x)\,dx'dt'\right| \\
       & \qquad \leq \frac{1}{\tau_n\xi_n^d} \|\overline{\zeta}\|_{L^\infty}\|\hat{\zeta}\|_{L^\infty} \int_{C_{{\xi_n}, {\tau_n}}(t,x)} \bigl|F(t,x,\theta(t',x')) - F(t,x,\theta(t,x))\bigr|\,dx'dt' \\
       & \qquad\quad =  \frac{\|\overline{\zeta}\|_{L^\infty}\|\hat{\zeta}\|_{L^\infty}|B_1(0)|}{|C_{{\xi_n}, {\tau_n}}(t,x)|} \int_{C_{{\xi_n}, {\tau_n}}(t,x)} |F(t,x,\theta(t',x'))-F(t,x,\theta(t,x))| \,dx'dt'.
   \end{aligned}\]
    Using \eqref{lebesgue-point-at-composition} in the last estimate, we conclude \eqref{half-convolution-lebesgue-point-approximation}.
\end{proof}

Next, we recall a result on the least upper bound property of the
lattice of measurable functions. Our statement of this classic result
is adapted from the proof of \cite[Theorem~4.7.1]{Bog07}.

\begin{lemma}\label{lub-lemma}
     Let $\mathcal{F}$ be a non-empty family of Lebesgue measurable functions on a reference domain $U$. Then there exists a sequence $f_1, f_2, \ldots \in \mathcal{F}$ and a Lebesgue measurable function $f^*$ such that
    \begin{itemize}
        \item $\max(f_1, f_2, \ldots, f_n)$ converges a.e.~to $f^*$,
        \item $f\leq f^*$ a.e.~for all $f\in \mathcal{F}$.
    \end{itemize}
\end{lemma}

Finally, we provide a variant of the chain rule for functions of bounded
variation, adapted from \cite[Section 3]{AFP2000} and suited to the present
setting. Recall that, for \(u\in BV(U)\), the approximate limit \(u^*(x)\) is
defined at points \(x\in U\) for which there exists \(a\in\RR\) such that,
for every \(\epsi>0\),
\[
\lim_{r\downarrow0}
\frac{\left|\{y\in B_r(x): |u(y)-a|>\epsi\}\right|}
{|B_r(x)|}=0.
\]
A point \(x\in U\) is called a jump point of \(u\) if there exist distinct values
\(u^\sharp(x),u^\flat(x)\in\RR\) and a unit vector \(\nu_u(x)\) such that
\(u\) has approximate limits \(u^\sharp(x)\) and \(u^\flat(x)\) on the two
half-balls
\[
B_r^+(x,\nu_u(x)):=\{y\in B_r(x):(y-x)\cdot\nu_u(x)>0\}
\]
and
\[
B_r^-(x,\nu_u(x)):=\{y\in B_r(x):(y-x)\cdot\nu_u(x)<0\},
\]
respectively, as \(r\downarrow0\).

Departing from the standing convention of the paper, in this section we write $Du$
 for the full space-time distributional derivative of $u\in BV(U)$, viewed as a Radon measure, rather than for the spatial derivative.

\begin{lemma}[A Chain Rule for BV]\label{lem:bv-chain}
Let \(u\in BV(U)\), where \(U\) is a reference domain, and let
\(F\colon\RR\to\RR\) be globally Lipschitz with \(F(0)=0\). Define a bounded Borel
function \(\overline{F'u}\colon U\to\RR\) by
\[
\overline{F'u}(x)=
\begin{cases}
F'(u^*(x)), & \text{whenever this is well-defined},\\[1mm]
\dfrac{F(u^\sharp(x))-F(u^\flat(x))}
      {u^\sharp(x)-u^\flat(x)}, & \text{if }x\text{ is a jump point},\\[2mm]
\text{an arbitrary value}, & \text{otherwise}.
\end{cases}
\]
The first two cases determine \(\overline{F'u}\) up to a
\(|Du|\)-negligible set. Hence the measure
\(\overline{F'u}\,Du\) does not depend on the arbitrary values assigned
on the remaining set.
Then,
\[
F(u)\in BV(U)
\qquad\text{and}\qquad
DF(u)=\overline{F'u}\,Du .
\]
\end{lemma}

We use Lemma~\ref{lem:bv-chain} only through the following corollary. Notice
that the conclusion does not imply that \(h\) can be chosen with values in
\(\{0,1\}\), because the set \(\{0<h<1\}\) may fail to be null with respect
to \(|Du_1|\) or \(|Du_2|\).

\begin{corollary}\label{cor:bv-chain}
Let \(u_1,u_2\in BV(U)\), where \(U\) is a reference domain, and let
\(\tilde u:=\max(u_1,u_2)\). Then there exists a Borel measurable function
\(h\colon U\to[0,1]\) such that
\[
D\tilde u
=
h\,Du_1+(1-h)\,Du_2,
\]
and the set \(\{0<h<1\}\) is Lebesgue null.
\end{corollary}

\begin{proof}
We have
\[
\tilde u=\frac{u_1+u_2+|u_1-u_2|}{2}
\]
and apply Lemma~\ref{lem:bv-chain} with \(F(s)=|s|\) to
\(v:=u_1-u_2\). On the set where the value of \(\overline{F'v}\) is
arbitrary, we choose it to belong to \(\{-1,1\}\). Then
\(\overline{F'v}\) takes values in \([-1,1]\), and
\[
D\tilde u
=
\frac{1+\overline{F'v}}{2}\,Du_1
+
\frac{1-\overline{F'v}}{2}\,Du_2 .
\]
Thus the first conclusion follows with
\[
h:=\frac{1+\overline{F'v}}{2}.
\]
Moreover, outside the jump set of \(v\), up to a Lebesgue null set, either
\(v^*\neq0\), in which case \(\overline{F'v}=\operatorname{sgn}(v^*)\in
\{-1,1\}\), or \(\overline{F'v}\) has been chosen in \(\{-1,1\}\). Hence
\(\{0<h<1\}\) is contained, up to a Lebesgue null set, in the jump set of
\(v\). Since the jump set is Lebesgue null, \(\{0<h<1\}\) is Lebesgue null.
\end{proof}

\section{Proof of Theorem~\ref{weak.exists}}\label{sec:proof-weak-exists}

In this section, we prove Theorem~\ref{weak.exists} by first considering suitable regularizations and establishing the existence of solutions via the abstract theorem, then proving uniform bounds on the regularized solutions, and finally passing to the limit by Minty's method.

\subsection{Finding Regularized Solutions}

Let $\gamma := \alpha\beta'$ and define the space
\[
W^{1,\gamma}_{T}(Q) := \{\nu\in W^{1,\gamma}(Q) \colon \nu(T,x) = 0\},
\]
where the identity $\nu(T,x) = 0$ is understood in the trace sense. 

 In the notation of Theorem~\ref{thm:monotone.abstract-weak}, set $X := L^{\beta}(Q)\times W^{1,\gamma}_{T}(Q)$ and $z_0 := (0, u_T)$, where $u_T(x)$ is viewed as a function of $(t,x)$ independent of $t$. Consider the convex, closed set $\mathcal{K}\subset z_0 + X$ given by
\begin{equation*}
  \mathcal{K} := \{(\mu, \upsilon) \in z_0 + X \colon \mu \geq 0\} =  L^{\beta}(Q;\RR^+_0)\times (u_T+W^{1,\gamma}_{T}(Q)),
\end{equation*}
and define an operator $A\colon \mathcal{K} \to X'$ as follows: for all $(\mu,\upsilon)\in\mathcal{K}$ and $(\eta,\nu)\in X$, let
\begin{equation}\label{defA}
\begin{aligned}
\bigl\langle A[\mu,\upsilon],(\eta,\nu)\bigr\rangle & = \int_Q\bigl(\upsilon_t -H(t,x,D\upsilon,\mu)\bigr)\,\eta\,dxdt\\
&\quad +\int_Q\bigl(-\mu\nu_t + \mu D_pH(t,x,D\upsilon,\mu)\cdot D\nu \bigr)\,dxdt-\int_{\Tt^d}m_0(x)\nu(0,x)\,dx.
\end{aligned}
\end{equation}

\begin{proposition}
 Consider the setting of Problem~\ref{prob.mfg}. Suppose Assumptions~\ref{onH.monotone} and~\ref{onH.powergrowth} hold. Then the operator $A$ defined in~\eqref{defA} is well-defined, monotone, and hemicontinuous.
\end{proposition}
\begin{proof}
We first establish that the operator is well-defined. Let $(\mu,\upsilon)\in\mathcal{K}$ and $(\eta,\nu)\in L^{\beta}(Q)\times W^{1,\gamma}_{T}(Q)$. We verify that each integral in~\eqref{defA} is finite. Since $\upsilon\in W^{1,\gamma}(Q)$ with $\gamma = \alpha\beta' > \beta'$, we have $\upsilon_t\in L^{\gamma}(Q)\subset L^{\beta'}(Q)$ (as the domain $Q$ has finite measure) and $D\upsilon\in L^{\gamma}(Q)$. By~\eqref{res:H.upper.simple} and~\eqref{eq:assH.lower.simple},
\[|H(t,x,D\upsilon,\mu)| \leq C(|D\upsilon|^{\alpha} + \mu^{\beta-1} + V(t,x)).\]
Since $|D\upsilon|^{\alpha}\in L^{\gamma/\alpha}(Q) = L^{\beta'}(Q)$, $\mu^{\beta-1}\in L^{\beta/(\beta-1)}(Q) = L^{\beta'}(Q)$, and $V\in L^{\beta'}(Q)$ 
by Assumption \ref{onH.powergrowth}, we conclude that $\upsilon_t - H(t,x,D\upsilon,\mu)\in L^{\beta'}(Q)$. Thus, the first integral in~\eqref{defA} is finite by H\"older's inequality.

For the second integral, $\mu\in L^{\beta}(Q)$ and $\nu_t\in L^{\gamma}(Q)\subset L^{\beta'}(Q)$, so $\mu\nu_t\in L^1(Q)$. Moreover, by~\eqref{eq:assH.DpH.upper.simple},
\[|\mu\, D_pH(t,x,D\upsilon,\mu)\cdot D\nu| \leq C\mu\bigl(|D\upsilon|^{\alpha-1} + \mu^{\frac{(\alpha-1)(\beta-1)}{\alpha}} + V^{\frac{\alpha-1}{\alpha}}\bigr)|D\nu|.\]
Each of the three resulting terms is integrable by H\"older's inequality with exponents summing to one; for instance, for the first term, $\frac{1}{\beta}+\frac{\alpha-1}{\gamma}+\frac{1}{\gamma} = \frac{1}{\beta}+\frac{1}{\beta'} = 1$.
Finally, the boundary integral $\int_{\Tt^d}m_0(x)\nu(0,x)\,dx$ is finite since $m_0\in C(\Tt^d)$ is bounded and $\nu\in W^{1,\gamma}(Q)$ has a well-defined trace at $t=0$.

The estimates above are linear in \((\eta,\nu)\) and bounded by
\(C_{\mu,\upsilon}\|(\eta,\nu)\|_X\), hence \(A[\mu,\upsilon]\in X'\).

Now, we address the monotonicity.
Let $(\mu_1,\upsilon_1),\,(\mu_2,\upsilon_2)\in\mathcal{K}$. Since both $\upsilon_1(T,x)=\upsilon_2(T,x)=u_T(x)$, we have $\upsilon_1-\upsilon_2\in W^{1,\gamma}_{T}(Q)$. Therefore,  $(\mu_1-\mu_2,\,\upsilon_1-\upsilon_2)\in L^{\beta}\times W^{1,\gamma}_{T}(Q)$. A direct computation using~\eqref{defA} shows that the terms involving $(\upsilon_1-\upsilon_2)_t\,(\mu_1-\mu_2)$ from the first and second lines cancel, and the boundary terms involving $m_0$ also cancel, leaving
\begin{equation*}
    \begin{aligned}
        &\bigl\langle A[\mu_1,\upsilon_1] - A[\mu_2,\upsilon_2],\,(\mu_1-\mu_2,\,\upsilon_1-\upsilon_2)\bigr\rangle\\
&\quad = \int_Q\Big[\bigl(- H(t,x,D\upsilon_1,\mu_1) + H(t,x,D\upsilon_2,\mu_2)\bigr)(\mu_1-\mu_2)\\
&\qquad\qquad + \bigl(\mu_1\, D_pH(t,x,D\upsilon_1,\mu_1) - \mu_2\, D_pH(t,x,D\upsilon_2,\mu_2)\bigr)\cdot(D\upsilon_1-D\upsilon_2)\Big]\,dxdt\\
&\quad\geq 0,
    \end{aligned}
\end{equation*}
where the last inequality follows from the monotonicity condition~\eqref{hmon} applied pointwise a.e.\ with $p_i=D\upsilon_i$ and $m_i = \mu_i$.

Finally, we prove hemicontinuity. Let $(\mu,\upsilon),\,(\bar\mu,\bar\upsilon)\in\mathcal{K}$ and $(\eta,\nu)\in L^{\beta}\times W^{1,\gamma}_{T}(Q)$. For $\lambda\in[0,1]$, set $(\mu_\lambda,\upsilon_\lambda) := (1-\lambda)(\bar\mu,\bar\upsilon) + \lambda(\mu,\upsilon)$. Since the maps $(p,m)\mapsto H(t,x,p,m)$ and $(p,m)\mapsto D_pH(t,x,p,m)$ are continuous by the formulation of Problem~\ref{prob.mfg}, each integrand in $\langle A[\mu_\lambda,\upsilon_\lambda],(\eta,\nu)\rangle$ is continuous in $\lambda$ for a.e.~$(t,x)$. Moreover, the growth bounds~\eqref{res:H.upper.simple} and~\eqref{eq:assH.DpH.upper.simple}, applied with $|D\upsilon_\lambda|\leq |D\bar\upsilon|+|D\upsilon|$ and $\mu_\lambda \leq \bar\mu+\mu$, provide a fixed integrable majorant. The conclusion follows from Lebesgue's dominated convergence theorem.
\end{proof}

Now, for $\epsi\geq 0$, define a regularized operator $A_\epsi\colon \mathcal{K} \to X'$ by adding to $A$ a space-time $\gamma$-Laplacian term in $\upsilon$:
\begin{equation}\label{defAepsi}
\begin{aligned}
\bigl\langle A_\epsi[\mu,\upsilon],(\eta,\nu)\bigr\rangle & = \bigl\langle A[\mu,\upsilon],(\eta,\nu)\bigr\rangle \\
&\quad + \epsi\int_Q \bigl(|D\upsilon|^{\gamma-2}D\upsilon\cdot D\nu + |\upsilon_t|^{\gamma-2}\upsilon_t\nu_t\bigr)\, dxdt.
\end{aligned}
\end{equation}
In all uniform estimates below, we take \(0<\epsi\leq 1\). This is sufficient for the passage to the limit \(\epsi\downarrow0\).

\begin{proposition}\label{prop:regularized-solutions}
Consider the setting of Problem~\ref{prob.mfg}. Suppose Assumptions~\ref{onH.monotone} and~\ref{onH.powergrowth} hold.
  For $\epsi>0$, the operator $A_\epsi$ defined in~\eqref{defAepsi} is well-defined, monotone, coercive, and hemicontinuous. Hence, there exists $(m_\epsi, u_\epsi)\in\mathcal{K}$ such that
  \[\bigl\langle A_\epsi[m_\epsi, u_\epsi], (\mu-m_\epsi, \upsilon-u_\epsi) \bigr\rangle \geq 0\]
  for all $(\mu, \upsilon)\in \mathcal{K}$.
\end{proposition}
\begin{proof}
By the preceding proposition, $A$ is well-defined and monotone on $\mathcal{K}$. The additional regularization term in~\eqref{defAepsi} involves
$|D\upsilon|^{\gamma-2}D\upsilon\cdot D\nu$ and $|\upsilon_t|^{\gamma-2}\upsilon_t\,\nu_t$, each of which is integrable by H\"older's inequality with conjugate exponents $\gamma'$ and $\gamma$ for the terms involving $\upsilon$ and $\nu$, respectively. Furthermore, this regularization term is the G\^ateaux derivative of the convex functional $\upsilon\mapsto \frac{1}{\gamma}\int_Q(|D\upsilon|^{\gamma}+|\upsilon_t|^{\gamma})\,dxdt$. 
Therefore, the regularization term defines a monotone operator in the \(\upsilon\)-component.
Hence, $A_\epsi$ is well-defined and monotone as the sum of well-defined and monotone operators.

To establish coercivity, let $(\bar\mu,\bar\upsilon) := (0,u_T)\in\mathcal{K}$, where we view $u_T(x)$ as a function of $(t,x)$ independent of $t$; note that $\bar\upsilon\in W^{1,\gamma}(Q)$ since $u_T\in C^1(\Tt^d)$. For any $(\mu,\upsilon)\in\mathcal{K}$, we have $\upsilon-u_T\in W^{1,\gamma}_{T}(Q)$. 
A direct computation, using~\eqref{defA}--\eqref{defAepsi} to show that the terms involving $\upsilon_t\,\mu$ from the first and second lines of~\eqref{defA} cancel, and the boundary terms involving $m_0$ also cancel, yields
\begin{equation}\label{eq:coercive-test.Aepsi}
\begin{aligned}
  &\bigl\langle A_\epsi[\mu,\upsilon] - A_\epsi[0,u_T],\,(\mu,\,\upsilon-u_T)\bigr\rangle\\
  &\quad = \int_Q\mu\bigl[D_pH(t,x,D\upsilon,\mu)\cdot D\upsilon - H(t,x,D\upsilon,\mu) + H(t,x,Du_T,0)\\
  &\qquad\qquad\qquad - D_pH(t,x,D\upsilon,\mu)\cdot Du_T\bigr]\,dxdt\\
  &\qquad + \epsi\int_Q\bigl[(|D\upsilon|^{\gamma-2}D\upsilon - |Du_T|^{\gamma-2}Du_T)\cdot(D\upsilon-Du_T) + |\upsilon_t|^{\gamma}\bigr]\,dxdt.
\end{aligned}
\end{equation}

We estimate the first integral on the right-hand side of~\eqref{eq:coercive-test.Aepsi} using~\eqref{dphdotp-minus-h.lower.bound} for the Lagrangian lower bound, \eqref{eq:assH.lower.simple} for $H(t,x,Du_T,0)$, and~\eqref{eq:assH.DpH.upper.simple} for the $D_pH\cdot Du_T$ term (noting that $\|Du_T\|_{L^\infty}\leq C$ since $u_T \in C^1(\Tt^d)$). Combining these estimates with Young's inequality to absorb the lower-order terms into $\frac{1}{C}(\mu|D\upsilon|^\alpha+\mu^\beta)$, and using~\eqref{eq:assH.potential}, we obtain
\begin{equation}\label{eq:coercive.Aepsi}
\begin{aligned}
  &\bigl\langle A_\epsi[\mu,\upsilon] - A_\epsi[0,u_T],\,(\mu,\,\upsilon-u_T)\bigr\rangle\\
  &\quad \geq \frac{1}{C}\int_Q\mu^\beta\,dxdt \\
  &\qquad + \epsi\int_Q\bigl[(|D\upsilon|^{\gamma-2}D\upsilon - |Du_T|^{\gamma-2}Du_T)\cdot(D\upsilon-Du_T) + |\upsilon_t|^{\gamma}\bigr]\,dxdt - C.
\end{aligned}
\end{equation}

For the $\epsi$-regularization term, since $|Du_T|\leq C$, Young's inequality gives
\[(|D\upsilon|^{\gamma-2}D\upsilon - |Du_T|^{\gamma-2}Du_T)\cdot(D\upsilon-Du_T) \geq |D\upsilon|^{\gamma} - C|D\upsilon|^{\gamma-1} - C|D\upsilon| \geq \frac{1}{C}|D\upsilon|^{\gamma} - C.\]
Substituting into~\eqref{eq:coercive.Aepsi} and noting that $\upsilon-u_T\in W^{1,\gamma}_{T}(Q)$ (so $\|(\mu,\upsilon-u_T)\|_X$ is controlled by $\|\mu\|_{L^\beta}$, $\|D\upsilon\|_{L^\gamma}$, and $\|\upsilon_t\|_{L^\gamma}$ up to a constant depending on $u_T$ and the Poincar\'e inequality), we conclude that
\[\lim_{\substack{\Vert (\mu,\upsilon-u_T)\Vert_X \to+\infty \\ (\mu,\upsilon) \in \mathcal{K}}}\frac{\bigl\langle A_\epsi[\mu,\upsilon] - A_\epsi[0, u_T] ,\, (\mu,\,\upsilon-u_T) \bigr\rangle}{\Vert (\mu,\upsilon-u_T)\Vert_X} = +\infty,\]
which shows that $A_\epsi$ is coercive.

The hemicontinuity of $A$ was shown in the preceding proposition. The additional term $\epsi\int_Q(|D\upsilon_\lambda|^{\gamma-2}D\upsilon_\lambda\cdot D\nu + |(\upsilon_\lambda)_t|^{\gamma-2}(\upsilon_\lambda)_t\,\nu_t)\,dxdt$ is continuous in $\lambda$ by the continuity of the maps $p\mapsto |p|^{\gamma-2}p$ and $a\mapsto |a|^{\gamma-2}a$, and Lebesgue's dominated convergence theorem. Hence $A_\epsi$ is hemicontinuous.

Since $\mathcal{K}$ is convex and closed, and $A_\epsi$ is monotone, coercive, and hemicontinuous, Theorem~\ref{thm:monotone.abstract-weak} yields $(m_\epsi, u_\epsi)\in\mathcal{K}$ satisfying the strong variational inequality.
\end{proof}

\begin{lemma}\label{lem:strong-epsi}
Consider the setting of Problem~\ref{prob.mfg}. Suppose Assumptions \ref{onH.monotone} and \ref{onH.powergrowth} hold. Let $\epsi>0$ and let $(m_\epsi, u_\epsi)\in \mathcal{K}$ be given by Proposition~\ref{prop:regularized-solutions}. Then $(m_\epsi, u_\epsi)$ solves a regularization of Problem~\ref{prob.mfg} in the following sense:
  \begin{enumerate}
  \item The Hamilton--Jacobi equation holds in the sense that
\begin{equation}\label{hjb-strong-epsi}
\begin{alignedat}{2}
& -(u_\epsi)_t + H(t,x,Du_\epsi,m_\epsi) \leq 0 \qquad && \text{a.e.~in } Q = (0,T)\times\Tt^d, \\
 & -(u_\epsi)_t + H(t,x,Du_\epsi,m_\epsi) = 0 \qquad && \text{a.e.~in the set } \{m_\epsi>0\};
\end{alignedat}
  \end{equation}\smallskip

  \item The regularized transport equation with the initial condition holds in the sense that
  \begin{equation}\label{transport-strong-epsi}
     \begin{aligned}
       & \int_Q m_\epsi\bigl(-\varphi_t + D_pH(t,x,Du_\epsi,m_\epsi)\cdot D\varphi\bigr) \,dxdt \\ & \qquad + \epsi\int_Q \bigl(|Du_\epsi|^{\gamma-2}Du_\epsi\cdot D\varphi + |(u_\epsi)_t|^{\gamma-2}(u_\epsi)_t\varphi_t\bigr) \,dxdt \\
       & \qquad \qquad = \int_{\Tt^d} m_0(x)\varphi(0,x)\,dx
     \end{aligned}
  \end{equation}
  for all $\varphi\in W^{1,\gamma}_{T}(Q)$.
  \end{enumerate}
\end{lemma}

\begin{proof}
  By Proposition~\ref{prop:regularized-solutions}, $(m_\epsi, u_\epsi)\in\mathcal{K}$ satisfies $\langle A_\epsi[m_\epsi, u_\epsi], (\mu-m_\epsi, \upsilon-u_\epsi) \rangle \geq 0$ for all $(\mu, \upsilon)\in \mathcal{K}$. We derive each claim by choosing appropriate variations in $(\mu, \upsilon)$.

  \medskip\noindent\textbf{Hamilton--Jacobi equation.}
  Taking $(\mu, \upsilon) = (m_\epsi + \phi,\, u_\epsi)$ for an arbitrary non-negative $\phi \in L^\beta(Q)$, the terms involving $\upsilon - u_\epsi$ vanish in~\eqref{defAepsi}, leaving
  \[
     \int_Q \bigl((u_\epsi)_t - H(t,x,Du_\epsi,m_\epsi)\bigr)\,\phi \,dxdt \geq 0.
  \]
  Since $\phi \geq 0$ is arbitrary, this yields $(u_\epsi)_t - H(t,x,Du_\epsi,m_\epsi) \geq 0$ a.e., establishing the first line of~\eqref{hjb-strong-epsi}. Next, taking $(\mu, \upsilon) = \bigl(m_\epsi(1 \pm \phi),\, u_\epsi\bigr)$ for an arbitrary $\phi \in L^\infty(Q)$ satisfying $\|\phi\|_{L^\infty} \leq 1$, the constraint $\mu \geq 0$ is satisfied, and we obtain
  \[
     \pm\int_Q \bigl((u_\epsi)_t - H(t,x,Du_\epsi,m_\epsi)\bigr)\,m_\epsi\,\phi \,dxdt \geq 0.
  \]
  Since $\phi$ is arbitrary, this forces $m_\epsi\bigl((u_\epsi)_t - H(t,x,Du_\epsi,m_\epsi)\bigr) = 0$ a.e., which establishes the second line of~\eqref{hjb-strong-epsi}.

  \medskip\noindent\textbf{Regularized transport equation.}
  Taking $(\mu, \upsilon) = (m_\epsi,\, u_\epsi \pm \varphi)$ for an arbitrary $\varphi \in W^{1,\gamma}_{T}(Q)$, we note that $(u_\epsi \pm \varphi)(T,x) = u_T(x)$ since $\varphi(T,x) = 0$, so $(\mu, \upsilon) \in \mathcal{K}$. Since $\mu - m_\epsi = 0$, the first integral in~\eqref{defA} vanishes, and~\eqref{defAepsi} gives
  \begin{align*}
     & \pm\biggl[\int_Q m_\epsi \bigl(-\varphi_t + D_pH(t,x,Du_\epsi,m_\epsi)\cdot D\varphi\bigr) \,dxdt \\
     & \qquad + \epsi\!\int_Q \bigl(|Du_\epsi|^{\gamma-2}Du_\epsi\cdot D\varphi + |(u_\epsi)_t|^{\gamma-2}(u_\epsi)_t\varphi_t\bigr) \,dxdt \\
     & \qquad\qquad - \int_{\Tt^d} m_0(x)\varphi(0,x)\,dx\biggr] \geq 0.
  \end{align*}
  The $\pm$ condition forces the bracketed expression to vanish, which after rearrangement is precisely~\eqref{transport-strong-epsi}.
\end{proof}

\subsection{Uniform Estimates on Regularized Solutions}\label{sec:estimates}
In this section, we derive uniform-in-$\epsi$ estimates on the regularized solutions $(m_\epsi, u_\epsi)$ provided by Proposition~\ref{prop:regularized-solutions}. The first proposition below contains the estimates that do not rely on Assumptions~\ref{onH.Lagrangianlowerbound} or~\ref{onm0.lowerbound}.

\begin{proposition}\label{prop:uinitial.and.m.bound}
Consider the setting of Problem~\ref{prob.mfg}.
Suppose Assumptions~\ref{onH.monotone} and~\ref{onH.powergrowth} hold.
Let $\epsi>0$ and let $(m_\epsi, u_\epsi)\in \mathcal{K}$ be given by Proposition~\ref{prop:regularized-solutions}.
Then we have
  \begin{align}
     & \int_{\Tt^d} m_0(x)u_\epsi^{-}(0,x)\,dx \leq C, \label{u.initial.m0} \\
     & \int_Q m_\epsi|Du_\epsi|^\alpha\,dxdt \leq C, \label{mDu-basic} \\
     & \int_Q m_\epsi^{\beta}\,dxdt \leq C, \label{m.bound} \\
     & \epsi\int_Q \bigl(|Du_\epsi|^{\gamma} + |(u_\epsi)_t|^{\gamma}\bigr) \,dxdt \leq C, \label{epsi-dependent-bound} \\
     & \int_Q \bigl(((u_\epsi)_t)^{-}\bigr)^{\beta'}\,dxdt \leq C, \label{ut-minus-bound} \\
     & \sup_{s\in [0,T]} \int_{\Tt^d} u_\epsi^{+}(s,x)^{\beta'}\,dx \leq C, \label{u.upper.beta'-uniform-in-time}
  \end{align}
  for a positive constant $C$ independent of~$\epsi$. Here, $u_\epsi(s,\cdot)$ refers to the trace of $u_\epsi$ on the hypersurface $\{s\}\times \Tt^d$.
\end{proposition}

\begin{proof} The proof proceeds in three steps: we first derive an energy-type inequality from testing the transport equation against the value function and combining with the Hamilton--Jacobi equation, then derive a pointwise-in-time estimate from the Hamilton--Jacobi inequality alone, and finally combine the two.

  \medskip\noindent\textbf{Step 1:} From~\eqref{hjb-strong-epsi}, we get
  \begin{equation}\label{hjb-times-m}
     -m_\epsi(u_\epsi)_t + m_\epsi H(t,x,Du_\epsi,m_\epsi) = 0 \qquad\text{a.e. in } Q,
  \end{equation}
  while from~\eqref{transport-strong-epsi} with $\varphi(t,x) = u_\epsi(t,x)-u_T(x)$, we get
  \begin{equation}\label{transport-against-u}
     \begin{aligned}
       & \int_Q m_\epsi\bigl(-(u_\epsi)_t + D_pH(t,x,Du_\epsi,m_\epsi)\cdot Du_\epsi - D_pH(t,x,Du_\epsi,m_\epsi)\cdot Du_T\bigr) \,dxdt \\ & \qquad + \epsi\int_Q \bigl(|Du_\epsi|^{\gamma}-|Du_\epsi|^{\gamma-2}Du_\epsi\cdot Du_T + |(u_\epsi)_t|^{\gamma}\bigr) \,dxdt \\
       & \qquad \qquad = \int_{\Tt^d} m_0(x)u_\epsi(0,x)\,dx - \int_{\Tt^d} m_0(x)u_T(x)\,dx.
     \end{aligned}
  \end{equation}
  Plugging~\eqref{hjb-times-m} into~\eqref{transport-against-u} and rearranging, we get
  \begin{equation}\label{basic-equality}
     \begin{aligned}
       & \int_Q m_\epsi\bigl(D_pH(t,x,Du_\epsi,m_\epsi)\cdot Du_\epsi-H(t,x,Du_\epsi,m_\epsi)\bigr)\,dxdt \\
       & \qquad - \int_Q m_\epsi D_pH(t,x,Du_\epsi,m_\epsi)\cdot Du_T \,dxdt \\
       & \qquad + \epsi\int_Q \bigl(|Du_\epsi|^{\gamma}-|Du_\epsi|^{\gamma-2}Du_\epsi\cdot Du_T + |(u_\epsi)_t|^{\gamma}\bigr)\,dxdt \\
       & \qquad \qquad = \int_{\Tt^d} m_0(x)u_\epsi(0,x)\,dx - \int_{\Tt^d} m_0(x)u_T(x)\,dx.
     \end{aligned}
  \end{equation}
  Using~\eqref{dphdotp-minus-h.lower.bound} and~\eqref{eq:assH.DpH.upper.simple} in~\eqref{basic-equality}, we obtain
  \begin{equation}\label{first-estimate-crude}
     \begin{aligned}
       \int_{\Tt^d} m_0(x)u_\epsi(0,x)\,dx & \geq \frac{1}{C}\int_Q m_\epsi(|Du_\epsi|^\alpha+m_\epsi^{\beta-1})\,dxdt \\
       & \quad - C\int_Q m_\epsi V(t,x)\,dxdt \\
       & \quad - C\int_Q m_\epsi\bigl(|Du_\epsi|^{\alpha-1}+m_\epsi^{\frac{(\alpha-1)}{\alpha}(\beta-1)} + V(t,x)^{\frac{\alpha-1}{\alpha}}\bigr)|Du_T| \,dxdt \\
       & \quad + \epsi\int_Q \bigl(|Du_\epsi|^{\gamma}-|Du_\epsi|^{\gamma-1}|Du_T| + |(u_\epsi)_t|^{\gamma}\bigr) \,dxdt \\
       & \qquad + \int_{\Tt^d} m_0(x)u_T(x)\,dx.
     \end{aligned}
  \end{equation}  
Now, we apply Young's inequality to handle the terms on the second and the third lines of~\eqref{first-estimate-crude}.
 For the potential term in the second line, Young's inequality with conjugate exponents $\beta$ and $\beta'$ gives $m_\epsi V(t,x) \leq \delta m_\epsi^\beta + C_\delta V(t,x)^{\beta'}$. On the other hand, regarding the third line, we begin by observing that 
$|Du_T|$ is uniformly bounded. Accordingly, 
 the first term of the third line is bounded by $\delta m_\epsi|Du_\epsi|^\alpha + C_\delta m_\epsi \leq \delta (m_\epsi|Du_\epsi|^\alpha + m_\epsi^\beta) + C_\delta$. Next, the second term involves $m_\epsi^{1 + \frac{(\alpha-1)(\beta-1)}{\alpha}}$ where the exponent satisfies $1 + \frac{(\alpha-1)(\beta-1)}{\alpha} = \beta - \frac{\beta-1}{\alpha} < \beta$, thus Young's inequality yields $m_\epsi^{\beta - \frac{\beta-1}{\alpha}} \leq \delta m_\epsi^\beta + C_\delta$. Finally, the third term is bounded by $\delta m_\epsi V(t,x) + C_\delta m_\epsi \leq \delta m_\epsi^\beta + C_\delta V(t,x)^{\beta'} + C_\delta$ for any $\delta > 0$, and $\int_Q V(t,x)^{\beta'}\leq C$, by Assumption \ref{onH.powergrowth}.
  
  Choosing $\delta$ sufficiently small allows us to absorb all $\delta m_\epsi^\beta$ and $\delta m_\epsi|Du_\epsi|^\alpha$ terms into the corresponding positive terms on the first line. Similarly, we absorb the second term of the fourth line into its first term. Bounding the remaining integrable terms and the initial/terminal data by a constant $C$, we obtain
  \[\begin{aligned}
     &\int_{\Tt^d} m_0(x)u_\epsi(0,x)\,dx \\
     &\qquad \geq \frac{1}{C}\left(\int_Q (m_\epsi|Du_\epsi|^\alpha + m_\epsi^{\beta} )\,dxdt + \epsi\int_Q (|Du_\epsi|^{\gamma}+|(u_\epsi)_t|^\gamma)\,dxdt \right) - C,
  \end{aligned} \]
  which we rewrite as
  \begin{equation}\label{first-estimate}
  \begin{aligned}
     & \int_{\Tt^d} m_0(x)u_\epsi^{-}(0,x)\,dx  + \frac{1}{C}\Biggl(\int_Q (m_\epsi|Du_\epsi|^\alpha + m_\epsi^{\beta} )\,dxdt + \epsi\int_Q (|Du_\epsi|^{\gamma}+|(u_\epsi)_t|^\gamma)\,dxdt \Biggr) \\
     & \qquad \leq \int_{\Tt^d} m_0(x)u_\epsi^{+}(0,x)\,dx + C \\
     & \qquad\qquad \leq C\int_{\Tt^d} u_\epsi^{+}(0,x)\,dx + C.
  \end{aligned}      
  \end{equation}

  \medskip\noindent\textbf{Step 2:} Using~\eqref{eq:assH.lower.simple} in~\eqref{hjb-strong-epsi}, we get
  \begin{equation}\label{hjb-estimated-wo-Du}
    -(u_\epsi)_t\leq C(m_\epsi^{\beta-1}+V(t,x)) \qquad\text{a.e. in } Q.  
  \end{equation}
    We use \eqref{hjb-estimated-wo-Du} in two ways. 
    
    First, we take the positive part of the left-hand side, raise to the power of $\beta'$, integrate over $Q$, and use \eqref{eq:assH.potential} to get
    \begin{equation}\label{ut-minus-bound.in-terms-of-m}
        \int_Q \bigl(((u_\epsi)_t)^{-}\bigr)^{\beta'}\,dxdt \leq C\int_Q m_\epsi^\beta\,dxdt + C.
    \end{equation}
  
  Second, for any $s\in [0,T]$, we integrate both sides of \eqref{hjb-estimated-wo-Du} from $s$ to $T$ for a.e.~$x\in\Tt^d$ to get
  \begin{equation}\label{pre-sesqui-estimate.for-remark-reference}
      u_\epsi(s,x)-u_T(x)\leq C\int_s^T (m_\epsi^{\beta-1}+V(t,x))\,dt \qquad\text{a.e. on } \Tt^d.
  \end{equation}
    We move the second term of the left-hand side to the right, take the positive part and expand the bounds of integration to $[0,T]$, and take supremum over $s\in [0,T]$ to obtain
  \[\sup_{s\in [0,T]} u_\epsi^+(s,x) \leq |u_T(x)| + C\int_0^T (m_\epsi^{\beta-1}+V(t,x))\,dt  \qquad\text{a.e. on } \Tt^d.\]
  We raise both sides to the power of $\beta'$, while absorbing the data into the constant $C$, and get
  \begin{equation}\label{sesqui-estimate}
    \begin{aligned}
        \sup_{s\in [0,T]} u_\epsi^+(s,x)^{\beta'} & \leq C + C\biggl(\int_0^T (m_\epsi^{\beta-1}+V(t,x))\,dt\biggr)^{\beta'} \\
        & \leq C + C\int_0^T (m_\epsi^{\beta}+V(t,x)^{\beta'})\,dt \qquad\qquad \text{for a.e. } x\in\Tt^d.
    \end{aligned}
  \end{equation}
   We integrate \eqref{sesqui-estimate} over $\Tt^d$ and use \eqref{eq:assH.potential} to conclude
  \begin{equation}\label{second-estimate}
    \sup_{s\in [0,T]} \int_{\Tt^d} u_\epsi^{+}(s,x)^{\beta'}\,dx \leq
    \int_{\Tt^d} \sup_{s\in [0,T]} u_\epsi^+(s,x)^{\beta'}\,dx \leq
    C\int_Q m_\epsi^{\beta}\,dxdt + C.
  \end{equation}

  \medskip\noindent\textbf{Step 3:}
    Taking \eqref{second-estimate} for $s= 0$ and combining with~\eqref{first-estimate}, we get
  \[\int_{\Tt^d} u_\epsi^{+}(0,x)^{\beta'}\,dx \leq C\int_{\Tt^d} u_\epsi^{+}(0,x)\,dx + C,\]
  which yields, by Young's inequality, that the right-hand side of \eqref{first-estimate} is bounded by $C$. Hence, we conclude~\eqref{u.initial.m0},~\eqref{mDu-basic},~\eqref{m.bound}, and~\eqref{epsi-dependent-bound} at once. Then, using \eqref{m.bound} in \eqref{ut-minus-bound.in-terms-of-m} and \eqref{second-estimate}, we conclude \eqref{ut-minus-bound} and \eqref{u.upper.beta'-uniform-in-time}.
\end{proof}

The preceding proposition estimates the positive part of $u_\epsi$ in an effective manner by \eqref{u.upper.beta'-uniform-in-time}, but provides only \eqref{u.initial.m0} for the negative part. Therefore, we use either Assumption~\ref{onH.Lagrangianlowerbound} to estimate $u_\epsi^-$ in a different way or use Assumption~\ref{onm0.lowerbound} to upgrade \eqref{u.initial.m0} to a stronger estimate.

\begin{proposition}
\label{prop:uepsi-lower-bound-A}
Consider the setting of Problem~\ref{prob.mfg}. Suppose Assumptions~\ref{onH.monotone}, \ref{onH.powergrowth}, and~\ref{onH.Lagrangianlowerbound} hold. Let $\epsi>0$ and let $(m_\epsi, u_\epsi)\in \mathcal{K}$ be given by Proposition~\ref{prop:regularized-solutions}.
Then, we have
	\begin{equation}\label{eq:lower-bound-w}
		\sup_{s\in [0,T]} \esssup_{x\in\Tt^d} u_\epsi^-(s,x) \leq C,
	\end{equation}
    for a positive constant $C$ independent of $\epsi$. Here, $u_\epsi(s,\cdot)$ refers to the trace of $u_\epsi$ on the hypersurface $\{s\}\times \Tt^d$.
\end{proposition}

\begin{proof}
    We first note a consequence of \eqref{transport-strong-epsi}. Given $\varphi\in W^{1,\gamma}(Q)$ satisfying $\varphi(T,x)\geq 0$ in the sense of trace, we have
    \begin{equation}\label{transport-epsi-on.one.side}
        \begin{aligned}
       & \int_{\{\varphi < 0\}} m_\epsi\bigl(\varphi_t - D_pH(t,x,Du_\epsi,m_\epsi)\cdot D\varphi\bigr) \,dxdt - \epsi\int_{\{\varphi < 0\}} |Du_\epsi|^{\gamma-2}Du_\epsi\cdot D\varphi \,dxdt \\
       & \qquad = \epsi\int_Q \bigl(-|(u_\epsi)_t|^{\gamma-2}(u_\epsi)_t\bigr)(\varphi^-)_t \,dxdt + \int_{\Tt^d} m_0(x)\varphi^-(0,x)\,dx.
     \end{aligned}
    \end{equation}
    This follows from \eqref{transport-strong-epsi} with $\varphi(t,x)$ replaced by $\varphi^-(t,x)$, after a rearrangement, because 
    \[((\varphi^-)_t,\, D\varphi^-) = \begin{cases}
        (-\varphi_t,\, -D\varphi) \qquad & \text{a.e.~in the set } \{\varphi < 0\}, \\
        (0,\, 0) & \text{a.e.~in the set } \{\varphi \geq 0\}.
    \end{cases}\]
    
    Next, by Assumption~\ref{onH.Lagrangianlowerbound}, we take a constant $c > 0$ satisfying
    \begin{equation}\label{lagrangian-lower-bound-taken.as.c}
        H(t,x,0,0)\leq c \qquad \text{for a.e. } (t,x)\in Q,
    \end{equation}
    and use \eqref{transport-epsi-on.one.side} with $\varphi(t,x) := u_\epsi(t,x) + c(T-t) + \|u_T\|_{L^\infty(\Tt^d)}$. Thus we have
    \begin{equation}\label{transport-epsi-one.side-applied.at.u}
        \begin{aligned}
       & \int_{\{\varphi < 0\}} \Bigl(m_\epsi\bigl(-c+ (u_\epsi)_t - D_pH(t,x,Du_\epsi,m_\epsi)\cdot Du_\epsi\bigr) -\epsi|Du_\epsi|^\gamma\Bigr)\,dxdt \\
       & \qquad = \epsi\int_Q \bigl(-|(u_\epsi)_t|^{\gamma-2}(u_\epsi)_t\bigr)(\varphi^-)_t \,dxdt + \int_{\Tt^d} m_0(x)\varphi^-(0,x)\,dx.
     \end{aligned}
    \end{equation}
    Moreover, we have
    \begin{equation}\label{remove-u-epsi-by-increasing-function-trick-blue}
        |(u_\epsi)_t|^{\gamma-2}(u_\epsi)_t(\varphi^-)_t \leq c^{\gamma-1}(\varphi^-)_t \qquad \text{a.e. in } Q,
    \end{equation}
    because $(\varphi^-)_t \in \{0,\, c-(u_\epsi)_t\}$ a.e.~and \eqref{remove-u-epsi-by-increasing-function-trick-blue} holds trivially in the case $(\varphi^-)_t = 0$ while it reduces to $(|s|^{\gamma-2}s - c^{\gamma-1})(s - c) \geq 0$ with $s := (u_\epsi)_t$ in the case $(\varphi^-)_t = c-(u_\epsi)_t$. Combining \eqref{transport-epsi-one.side-applied.at.u} and \eqref{remove-u-epsi-by-increasing-function-trick-blue}, we get
    \begin{equation}\label{transport-one-side-applied-and-cleaned}
        \begin{aligned}
       & \int_{\{\varphi < 0\}} \Bigl(m_\epsi\bigl(-c+ (u_\epsi)_t - D_pH(t,x,Du_\epsi,m_\epsi)\cdot Du_\epsi\bigr) -\epsi|Du_\epsi|^\gamma\Bigr)\,dxdt \\
       & \qquad \geq -\epsi c^{\gamma-1}\int_Q (\varphi^-)_t \,dxdt + \int_{\Tt^d} m_0(x)\varphi^-(0,x)\,dx = \int_{\Tt^d} (\epsi c^{\gamma-1} + m_0(x))\varphi^-(0,x)\,dx.
     \end{aligned}
    \end{equation}
    Now we estimate the left-hand side using \eqref{hjb-strong-epsi}, the convexity of $H(t,x,\cdot,m)$ as in Remark~\ref{rmk:cxty}, the monotonicity of $H(t,x,p,\cdot)$ as in Remark~\ref{rmk:cxty}, and \eqref{lagrangian-lower-bound-taken.as.c} as follows:
     \begin{equation}\label{integrand-on-the-negative-part-of-test-function-also-negative-blue}
         \begin{aligned}
        & m_\epsi\bigl(-c + (u_\epsi)_t - D_pH(t,x,Du_\epsi,m_\epsi)\cdot Du_\epsi\bigr) -\epsi|Du_\epsi|^\gamma \\
        & \quad = m_\epsi\bigl(-c + H(t,x,Du_\epsi,m_\epsi) - D_pH(t,x,Du_\epsi,m_\epsi)\cdot Du_\epsi\bigr) -\epsi|Du_\epsi|^\gamma \\
        & \qquad \leq m_\epsi(-c + H(t,x,0,m_\epsi)) \leq m_\epsi(-c + H(t,x,0,0)) \leq 0 \qquad \text{a.e. in } Q.
    \end{aligned}
     \end{equation}  
        Consequently, if \(L\) and \(B\) denote the left- and right-hand sides of
    \eqref{transport-one-side-applied-and-cleaned}, respectively, then
    \(L\leq0\) by \eqref{integrand-on-the-negative-part-of-test-function-also-negative-blue},
    while \(B\geq0\). Since \eqref{transport-one-side-applied-and-cleaned}
    gives \(L\geq B\), we have \(L=B=0\). Let \(R\) be the right-hand side of the original identity
    \eqref{transport-epsi-one.side-applied.at.u} and note that \eqref{transport-one-side-applied-and-cleaned} was established from the inequality $R\geq B$. Accordingly, we obtain \(R=L=0\), and
    \[
        R-B
        =
        \epsi\int_Q
        \bigl(c^{\gamma-1}
        - |(u_\epsi)_t|^{\gamma-2}(u_\epsi)_t\bigr)
        (\varphi^-)_t\,dxdt
        =0.
    \]
    The integrand above is non-negative by
    \eqref{remove-u-epsi-by-increasing-function-trick-blue}; therefore
    \eqref{remove-u-epsi-by-increasing-function-trick-blue} holds as equality
    a.e.~in \(Q\). Also, since \(L=0\) and its integrand is non-positive, that
    integrand vanishes a.e.~on \(\{\varphi<0\}\). Thus equality holds in the
    first inequality of
    \eqref{integrand-on-the-negative-part-of-test-function-also-negative-blue}
    a.e.~on \(\{\varphi<0\}\).

    The vanishing of the right-hand side of \eqref{transport-one-side-applied-and-cleaned} implies $\varphi^-(0,x) = 0$ for a.e.~$x\in\Tt^d$. Equality in \eqref{remove-u-epsi-by-increasing-function-trick-blue} implies $(\varphi^-)_t = 0$ a.e.~in $Q$ because $(|s|^{\gamma-2}s - c^{\gamma-1})(s - c) > 0$ for $s \neq c$. Finally, equality in the first inequality of \eqref{integrand-on-the-negative-part-of-test-function-also-negative-blue} implies $D\varphi = Du_\epsi = 0$ a.e.~in the set $\{\varphi < 0\}$, which yields $D\varphi^- = 0$ a.e.~in $Q$.
   
    Collecting everything, we have
    \[\varphi^-(0,\cdot) = 0, \qquad ((\varphi^-)_t,\, D\varphi^-) = (0,0) \qquad \text{a.e.}\]
    Therefore, $\varphi^-=0$ in $W^{1,\gamma}(Q)$, and hence its trace on
    every slice $\{s\}\times\Tt^d$ vanishes. Thus
    \[
        0\leq \varphi(s,x)
        =
        u_\epsi(s,x)+c(T-s)+\|u_T\|_{L^\infty(\Tt^d)}
        \qquad\text{for a.e. }x\in\Tt^d
    \]
    for every $s\in[0,T]$. Rearranging, we obtain
    \[
        u_\epsi^-(s,x)
        \leq cT+\|u_T\|_{L^\infty(\Tt^d)}
        \qquad\text{for a.e. }x\in\Tt^d,
    \]
    for all $s\in[0,T]$, concluding \eqref{eq:lower-bound-w}.
\end{proof}

\begin{remark}[Optimal control interpretation of the upper and the lower bounds]
The estimates derived in Propositions~\ref{prop:uinitial.and.m.bound} and~\ref{prop:uepsi-lower-bound-A} have a natural interpretation via the optimal control problem underlying the MFG. The value function formally satisfies
\[
 u(s,x) = \inf_{\gamma} \left\{ \int_s^T L(t,\gamma(t),\dot\gamma(t),
 m(t,\gamma(t)))\,dt + u_T(\gamma(T)) \right\},
\]
where $L(t,x,v,m) = \sup_{p} \{-p\cdot v - H(t,x,p,m)\}$ is the Lagrangian.

An upper bound can be obtained by evaluating the cost of a specific suboptimal trajectory. The simplest choice is the stationary strategy $\gamma(t)\equiv x$, which gives a running cost $L(t,x,0,m)=\sup_{p} \{ - H(t,x,p,m)\}$. By the basic power growth assumption~\eqref{eq:assH.lower.simple}
we obtain $L(t,x,0,m)\leq C(m^{\beta-1}+V(t,x))$. This is exactly the optimal control meaning of Step~2 in Proposition~\ref{prop:uinitial.and.m.bound}: integrating the Hamilton--Jacobi inequality in time at a fixed $x$ corresponds to evaluating the cost of the stationary strategy, which leads directly to~\eqref{pre-sesqui-estimate.for-remark-reference}. However, for local couplings satisfying Assumption~\ref{onH.powergrowth}---such as $H=\tfrac1\alpha|p|^\alpha - m^{\beta-1}$---the density $m$ can spike, driving $H(t,x,0,m)$ to $-\infty$ and the stationary running cost to $+\infty$. As a result, this stationary strategy does not yield a uniform pointwise constant upper bound. Controlling $u_\epsi$ from above therefore requires combining this stationary bound with the coupled integral estimates developed in Step~1, which exploit the duality between the Hamilton--Jacobi and transport equations to successfully control the resulting $m_\epsi^\beta$ integrals.

For the lower bound in Proposition~\ref{prop:uepsi-lower-bound-A}, evaluating the Lagrangian at $p=0$ gives $L(t,x,v,m) \geq -H(t,x,0,m)$. Since $H(t,x,0,\cdot)$ is non-increasing (Remark~\ref{rmk:cxty}), the uniform condition $H(t,x,0,0)\leq c$ yields $L \geq -c$ globally for all velocities~$v$. Thus, the total cost of \emph{any} trajectory is bounded from below by $-c(T-s)+\min_{\Tt^d} u_T$, which explains the absolute floor established in~\eqref{eq:lower-bound-w}.
\end{remark}

\begin{proposition}\label{prop:uepsi-lower-bound-B}
  Consider the setting of Problem~\ref{prob.mfg}. Suppose Assumptions~\ref{onH.monotone},~\ref{onH.powergrowth}, and \ref{onm0.lowerbound} hold. Let $\epsi>0$ and let $(m_\epsi, u_\epsi)\in \mathcal{K}$ be given by Proposition~\ref{prop:regularized-solutions}. Then we have
  \begin{equation} \label{u-lower-L1-uniform-in-time}
      \sup_{s\in [0,T]} \int_{\Tt^d} u_\epsi^-(s,x)\,dx \leq C, 
  \end{equation}
  for a positive constant $C$ independent of~$\epsi$. Here, $u_\epsi(s,\cdot)$ refers to the trace of $u_\epsi$ on the hypersurface $\{s\}\times \Tt^d$.
\end{proposition}

\begin{proof}
     Because $m_0\in C(\Tt^d)$ and $\Tt^d$ is compact, Assumption~\ref{onm0.lowerbound} implies $\inf_{x\in\Tt^d} m_0(x) > 0$. Thus,  \eqref{u.initial.m0} implies
     \begin{equation}\label{u-lower-L1-uniform-in-time:initial-time}
         \int_{\Tt^d} u_\epsi^-(0,x)\,dx \leq C;
     \end{equation}
     that is, the $s = 0$ case of \eqref{u-lower-L1-uniform-in-time}. Next, we integrate both sides of \eqref{hjb-estimated-wo-Du} from $0$ to $s$ for a.e.~$x\in\Tt^d$ to get
     \[u_\epsi(0,x)-u_\epsi(s,x)\leq C\int_0^s (m_\epsi^{\beta-1}+V(t,x))\,dt \qquad\text{a.e. on } \Tt^d.\]
     Then we move the first term of the left-hand side to the right, take the positive part, and expand the bounds of integration to $[0,T]$ to obtain
     \[u_\epsi^-(s,x) \leq u_\epsi^-(0,x) + C\int_0^T (m_\epsi^{\beta-1}+V(t,x))\,dt  \qquad\text{a.e. on } \Tt^d.\]
     Now we integrate both sides over $\Tt^d$ and conclude by \eqref{u-lower-L1-uniform-in-time:initial-time}, \eqref{m.bound}, and \eqref{eq:assH.potential}.
\end{proof}

We note that \eqref{eq:lower-bound-w} implies \eqref{u-lower-L1-uniform-in-time} and express this in a corollary.

\begin{corollary}\label{prop:uinitial.lower}
    Consider the setting of Problem~\ref{prob.mfg}. Suppose Assumptions~\ref{onH.monotone},~\ref{onH.powergrowth} and one of \ref{onH.Lagrangianlowerbound}/\ref{onm0.lowerbound} hold. Let $\epsi>0$ and let $(m_\epsi, u_\epsi)\in \mathcal{K}$ be given by Proposition~\ref{prop:regularized-solutions}.
	 Then \eqref{u-lower-L1-uniform-in-time} holds for a positive constant $C$ independent of~$\epsi$.
\end{corollary}
\begin{proof}
If Assumption~\ref{onH.Lagrangianlowerbound} holds, then
\eqref{eq:lower-bound-w} immediately implies \eqref{u-lower-L1-uniform-in-time}
because \(\Tt^d\) has finite measure. If Assumption~\ref{onm0.lowerbound}
holds, the conclusion is precisely Proposition~\ref{prop:uepsi-lower-bound-B}.
\end{proof}

Using these initial bounds, we now obtain uniform $L^\alpha(Q)$ estimates on $Du_\epsi$ and control on the positive part of $(u_\epsi)_t$.

\begin{proposition}\label{prop:Du.bound}
Consider the setting of Problem~\ref{prob.mfg}. Suppose Assumptions~\ref{onH.monotone},~\ref{onH.powergrowth} and one of \ref{onH.Lagrangianlowerbound}/\ref{onm0.lowerbound} hold. Let $\epsi>0$ and let $(m_\epsi, u_\epsi)\in \mathcal{K}$ be given by Proposition~\ref{prop:regularized-solutions}.
  Then we have
  \begin{align}
     & \int_Q ((u_\epsi)_t)^{+}\,dxdt \leq C, \label{ut-plus-bound} \\
     & \int_Q |Du_\epsi|^\alpha\,dxdt \leq C, \label{Du.bound}
  \end{align}
  for a positive constant $C$ independent of~$\epsi$.
\end{proposition}

\begin{proof}
    First we observe
  \[\begin{aligned}
      \int_Q ((u_\epsi)_t)^{+} \,dxdt & = \int_Q ((u_\epsi)_t)^{-}\,dxdt + \int_{\Tt^d} u_T(x)\,dx - \int_{\Tt^d} u_\epsi(0,x)\,dx, \\
      & \leq \int_Q ((u_\epsi)_t)^{-}\,dxdt + \int_{\Tt^d} u_\epsi^-(0,x)\,dx + C.
  \end{aligned}\]
  Hence, we conclude~\eqref{ut-plus-bound} by~\eqref{ut-minus-bound}~and the $s=0$ case of \eqref{u-lower-L1-uniform-in-time}, which holds by Corollary~\ref{prop:uinitial.lower}. On the other hand, \eqref{hjb-strong-epsi} and~\eqref{eq:assH.lower.simple} imply
  \begin{equation}\label{Du-alpha-upper-bound:hjb+Hlower}
      \frac{1}{C}|Du_\epsi|^\alpha \leq C(m_\epsi^{\beta-1}+V(t,x)) + (u_\epsi)_t \qquad\text{a.e.~in } Q,
  \end{equation}
  hence
  \begin{equation*}
     |Du_\epsi|^\alpha \leq C\bigl(m_\epsi^{\beta-1} + V(t,x) + ((u_\epsi)_t)^{+}\bigr) \qquad\text{a.e.~in } Q.  
  \end{equation*}
  We integrate the last expression over $Q$ and conclude~\eqref{Du.bound} by~\eqref{m.bound},~\eqref{eq:assH.potential}, and~\eqref{ut-plus-bound}.
\end{proof}

Now we provide an estimate with cross terms, which essentially states that~\eqref{mDu-basic} holds even when $m$ and $u$ come from different values of the regularization parameter and even when $u$ is shifted in the space variable and shifted backward in the time variable. For convenience of notation regarding this variable shifting, we introduce, for $s\in [0,T]$ and $y\in\Tt^d$, the functions $u_\epsi^{[s,y]}\in W^{1,\gamma}(Q)$, $m_\epsi^{[s,y]}\in L^{\beta}((0,T-s)\times\Tt^d)$, and $u_T^{[y]}\in C^1(\Tt^d)$ as follows:
\[\begin{aligned}
  u_\epsi^{[s,y]}(t,x) & = \begin{cases}
     u_\epsi(t+s,x+y),\enskip &\text{for } t<T-s,\\
     u_T(x+y), &\text{for } t>T-s,
  \end{cases} \\ 
  m_\epsi^{[s,y]}(t,x) & = m_\epsi(t+s,x+y), \\
  u_T^{[y]}(x) & = u_T(x+y).
\end{aligned} 
\]

\begin{proposition}\label{cross-estimate}
Consider the setting of Problem~\ref{prob.mfg}. Suppose Assumptions~\ref{onH.monotone},~\ref{onH.powergrowth} and one of \ref{onH.Lagrangianlowerbound}/\ref{onm0.lowerbound} hold. Let $\epsi\geq \epsi' > 0$, let $(m_{\epsi'}, u_{\epsi'})\in\mathcal{K}$ and $(m_\epsi, u_\epsi)\in\mathcal{K}$ be given by Proposition~\ref{prop:regularized-solutions}, and let $(s,y)\in \bar{Q} = [0,T]\times\Tt^d$. Then we have
\begin{equation}\label{mDu-advanced}
  \int_0^{T-s}\int_{\Tt^d} m_{\epsi'} |Du_\epsi^{[s,y]}|^\alpha\,dxdt \leq C,
\end{equation}
for a positive constant $C$ independent of $\epsi$, $\epsi'$, and $(s,y)$.
\end{proposition}

\begin{proof}
Evaluating~\eqref{hjb-strong-epsi} for $(m_\epsi,u_\epsi)$ at the shifted coordinates $(t+s,x+y)$ and multiplying by $m_{\epsi'}$, we obtain
 \begin{equation}\label{hjb-times-m.epsi-delta.shifted}
-m_{\epsi'}(u_\epsi)_t^{[s,y]} + m_{\epsi'}H(t+s,x+y,Du_\epsi^{[s,y]},m_\epsi^{[s,y]}) \leq 0 \qquad\text{a.e.~in } (0,T-s)\times\Tt^d.
 \end{equation}
On the other hand, using $\varphi(t,x) = u_\epsi^{[s,y]}(t,x) -u_T^{[y]}(x)$ as a test function in~\eqref{transport-strong-epsi} for $(m_{\epsi'},u_{\epsi'})$, we obtain
  \begin{equation}\label{transport-against-u.epsi-delta-shifted}
     \begin{aligned}
       & \int_0^{T-s}\int_{\Tt^d} m_{\epsi'}\bigl(-(u_\epsi)_t^{[s,y]} + D_pH(t,x,Du_{\epsi'},m_{\epsi'})\cdot (Du_\epsi^{[s,y]}-Du_T^{[y]})\bigr) \,dxdt \\ & \qquad + {\epsi'}\int_0^{T-s}\int_{\Tt^d} \bigl(|Du_{\epsi'}|^{\gamma-2}Du_{\epsi'}\cdot (Du_\epsi^{[s,y]}- Du_T^{[y]}) + |(u_{\epsi'})_t|^{\gamma-2}(u_{\epsi'})_t (u_\epsi)_t^{[s,y]}\bigr) \,dxdt \\
       & \qquad \qquad = \int_{\Tt^d} m_0(x)u_\epsi^{[s,y]}(0,x)\,dx - \int_{\Tt^d} m_0(x)u_T^{[y]}(x)\,dx.
     \end{aligned}
  \end{equation}
 Note that we changed the upper bound of the time integrals from $T$ to $T-s$ because $\varphi = 0$ on $(T-s,T)\times\Tt^d$.
 Plugging~\eqref{hjb-times-m.epsi-delta.shifted} into~\eqref{transport-against-u.epsi-delta-shifted} and rearranging, we obtain
 \begin{equation}\label{basic-equality.epsi-delta-shifted}
 \begin{aligned}
 & \int_0^{T-s}\int_{\Tt^d} m_{\epsi'} H(t+s,x+y,Du_\epsi^{[s,y]},m_\epsi^{[s,y]})\,dxdt \\
 & \qquad \leq \int_0^{T-s}\int_{\Tt^d} m_{\epsi'} D_pH(t,x,Du_{\epsi'},m_{\epsi'})\cdot (Du_\epsi^{[s,y]}-Du_T^{[y]}) \,dxdt \\
& \qquad\quad + {\epsi'}\int_0^{T-s}\int_{\Tt^d} \bigl(|Du_{\epsi'}|^{\gamma-2}Du_{\epsi'} \cdot (Du_\epsi^{[s,y]}-Du_T^{[y]}) + |(u_{\epsi'})_t|^{\gamma-2}(u_{\epsi'})_t (u_\epsi)_t^{[s,y]}\bigr) \,dxdt 
\\ & \qquad\qquad + \int_{\Tt^d} m_0(x)(u_\epsi^{[s,y]})^-(0,x)\,dx + \int_{\Tt^d} m_0(x)u_T^{[y]}(x)\,dx.
 \end{aligned}
\end{equation}
We bound the left-hand side of \eqref{basic-equality.epsi-delta-shifted} with~\eqref{eq:assH.lower.simple}, the first term on the right-hand side with~\eqref{eq:assH.DpH.upper.simple}, and the last line of the right-hand side by using Corollary~\ref{prop:uinitial.lower} and absorbing the data into the constant $C$. Hence, we obtain
\begin{equation}\label{crude-estimate.epsi-delta-shifted}
\begin{aligned}
& \frac{1}{C}\int_0^{T-s}\int_{\Tt^d} m_{\epsi'} |Du_\epsi^{[s,y]}|^\alpha\,dxdt \\
& \qquad \leq C\int_0^{T-s}\int_{\Tt^d} m_{\epsi'} V(t+s,x+y)\,dxdt + C\int_0^{T-s}\int_{\Tt^d} m_{\epsi'} (m_\epsi^{[s,y]})^{\beta-1} \,dxdt \\
 & \qquad\quad + C\int_0^{T-s}\int_{\Tt^d} m_{\epsi'} \bigl(|Du_{\epsi'}|^{\alpha-1}+m_{\epsi'}^{\frac{(\alpha-1)}{\alpha}(\beta-1)}+V(t,x)^{\frac{\alpha-1}{\alpha}}\bigr)(|Du_\epsi^{[s,y]}|+|Du_T^{[y]}|) \,dxdt \\
& \qquad\quad + {\epsi'}\int_0^{T-s}\int_{\Tt^d} \bigl(|Du_{\epsi'}|^{\gamma-1}(|Du_\epsi^{[s,y]}|+|Du_T^{[y]}|) + |(u_{\epsi'})_t|^{\gamma-1}|(u_\epsi)_t^{[s,y]}|\bigr) \,dxdt + C.
\end{aligned}
\end{equation}
Now,~\eqref{eq:assH.potential} and \eqref{m.bound} allow us to bound the first line of the right-hand side of~\eqref{crude-estimate.epsi-delta-shifted}, i.e.~the terms involving $V$ and $(m_\epsi^{[s,y]})^{\beta-1}$, by a constant~$C$. Moreover, we use Young's inequality to absorb the terms involving $|Du_\epsi^{[s,y]}|$ multiplied by $C m_{\epsi'}$ into the left-hand side of~\eqref{crude-estimate.epsi-delta-shifted}. Applying Young's inequality again to the remaining terms, we obtain
\begin{equation}\label{estimate.epsi-delta-shifted}
     \begin{aligned}
       \int_0^{T-s}\int_{\Tt^d} m_{\epsi'} |Du_\epsi^{[s,y]}|^\alpha\,dxdt & \leq C\int_0^{T-s}\int_{\Tt^d} m_{\epsi'} \bigl(|Du_{\epsi'}|^\alpha+m_{\epsi'}^{{\beta-1}}+V(t,x)+|Du_T^{[y]}|^\alpha\bigr)\,dxdt \\
       & + {\epsi'}\int_0^{T-s}\int_{\Tt^d} (|Du_{\epsi'}|^{\gamma} + |(u_{\epsi'})_t|^{\gamma} + |Du_T^{[y]}|^\gamma) \,dxdt \\
       & + {\epsi'}\int_0^{T-s}\int_{\Tt^d} (|Du_\epsi^{[s,y]}|^{\gamma} + |(u_\epsi)_t^{[s,y]}|^{\gamma}) \,dxdt.
     \end{aligned}
   \end{equation}
We finish the proof by bounding the first integral on the right-hand side of~\eqref{estimate.epsi-delta-shifted} using~\eqref{mDu-basic} and~\eqref{m.bound}, and bounding the $\epsi'$-integrals using~\eqref{epsi-dependent-bound}. Note that we can use~\eqref{epsi-dependent-bound} on both $\epsi'$-integrals since $\epsi'\leq \epsi$.
\end{proof}

\subsection{Additional Estimates}

In this section, we give two additional estimates which are not used in the proof of Theorem~\ref{weak.exists}, but rather may provide additional regularity on the solutions $(m,u)$ obtained as weak limits of $(m_\epsi, u_\epsi)$. We include them because they are direct consequences of the preceding bounds and may be useful for future regularity or compactness refinements.

\begin{proposition}
Consider the setting of Problem~\ref{prob.mfg}. Suppose Assumptions~\ref{onH.monotone} and \ref{onH.powergrowth} hold. Let $\epsi>0$ and let $(m_\epsi, u_\epsi)\in \mathcal{K}$ be given by Proposition~\ref{prop:regularized-solutions}. Then we have
\begin{equation}\label{uplus-grad-bound-v2}
    \int_Q (u_\epsi^+)^{\beta'-1}|Du_\epsi|^\alpha \,dxdt \leq C,
\end{equation}
for a positive constant $C$ independent of $\epsi$.
\end{proposition}

\begin{proof}
    For fixed \(k>0\), define
\[
\psi_k(s) :=\min\{(s^+)^{\beta'-1},k\},
\qquad
\Psi_k(s) :=\int_0^s\psi_k(r)\,dr.
\]
for all $s\in\RR$. Then $\psi_k$ is bounded, hence $\Psi_k$ is globally Lipschitz so that the Sobolev chain rule
\[
\bigl(\Psi_k(u_\epsi)\bigr)_t = \psi_k(u_\epsi)(u_\epsi)_t
\]
holds. Moreover, note that
\[0\leq \psi_k(s)\leq (s^+)^{\beta'-1}, \qquad 0\leq \Psi_k(s) \leq \frac{1}{\beta'}(s^+)^{\beta'}.\]

Now we have
\[\frac{1}{C}|Du_\epsi|^\alpha \leq C(m_\epsi^{\beta-1}+V(t,x)) + (u_\epsi)_t \qquad\text{a.e.~in } Q,\]
as in \eqref{Du-alpha-upper-bound:hjb+Hlower}. Multiplying this by \(\psi_k(u_\epsi)\geq 0\) and integrating over $Q$, while using the Sobolev chain rule, we get
\[
\begin{aligned}
& \frac{1}{C}\int_Q|Du_\epsi|^\alpha\psi_k(u_\epsi)\,dxdt \\
&\qquad\leq
C\int_Q
(m_\epsi^{\beta-1}+V(t,x))\psi_k(u_\epsi)\,dxdt + \int_{\Tt^d}\Psi_k(u_T(x))\,dx - \int_{\Tt^d}\Psi_k(u_\epsi(0,x))\,dx \\
&\qquad\leq
C\int_Q
(m_\epsi^{\beta-1}+V(t,x))(u_\epsi^+)^{\beta'-1}\,dxdt + C\int_{\Tt^d} |u_T(x)|^{\beta'}\,dx.
\end{aligned}
\]
The right-hand side is bounded by \eqref{m.bound}, \eqref{eq:assH.potential}, and \eqref{u.upper.beta'-uniform-in-time}, hence we find
\[
\int_Q|Du_\epsi|^\alpha\psi_k(u_\epsi)\,dxdt
\leq C,
\]
with \(C\) independent of \(k\) and \(\epsi\). Letting \(k\to\infty\) and
using the Monotone Convergence Theorem yields
\eqref{uplus-grad-bound-v2}.
\end{proof}

\begin{proposition}
    Consider the setting of Problem~\ref{prob.mfg}. Suppose Assumptions~\ref{onH.monotone},~\ref{onH.powergrowth} and one of \ref{onH.Lagrangianlowerbound}/\ref{onm0.lowerbound} hold. Let $\epsi>0$ and let $(m_\epsi, u_\epsi)\in \mathcal{K}$ be given by Proposition~\ref{prop:regularized-solutions}. Let
\[\alpha^* = \begin{cases}
        \frac{d\alpha}{d-\alpha},\quad & \text{if } \alpha < d, \\
        \infty & \text{if } \alpha > d,
    \end{cases}\]
    with the convention that $\alpha^*$ is any finite exponent in the borderline case $\alpha = d$.
    Then we have
  \begin{equation} \label{uminus-integral-bound-v2}
      \int_0^T \bigl( \|u_\epsi(s,\cdot)\|_{L^{\alpha^*}(\Tt^d)} \bigr)^\alpha\,ds \leq C, 
  \end{equation}
  for a positive constant $C$ independent of~$\epsi$. Moreover, let $q = \min(\alpha^*, \beta')$. Then, for any $\tau \in (0,T)$, we have
  \begin{equation}\label{uminus-forward-bound}
     \sup_{s\in [\tau,T]} \int_{\Tt^d} (u_\epsi^-(s,x))^q\,dx \leq C(\tau^{-\frac{q}{\alpha}}+1),
  \end{equation}
  for a positive constant $C$ independent of~$\epsi$ and $\tau$. Here, $u_\epsi(s,\cdot)$ refers to the trace of $u_\epsi$ on the hypersurface $\{s\}\times \Tt^d$.
\end{proposition}

\begin{proof}
    By \eqref{u.upper.beta'-uniform-in-time} and \eqref{u-lower-L1-uniform-in-time}, which holds due to Corollary~\ref{prop:uinitial.lower}, the spatial average $\bar{u}_\epsi(s) = \frac{1}{|\Tt^d|}\int_{\Tt^d} u_\epsi(s,x)\,dx$ is bounded uniformly for all $s \in [0,T]$. Thus, the Poincar\'e--Sobolev inequality yields
  \[ \|u_\epsi(s,\cdot)\|_{L^{\alpha^*}(\Tt^d)} \leq \|u_\epsi(s,\cdot) - \bar{u}_\epsi(s)\|_{L^{\alpha^*}(\Tt^d)} + \|\bar{u}_\epsi(s)\|_{L^{\alpha^*}(\Tt^d)} \leq C \|Du_\epsi(s,\cdot)\|_{L^\alpha(\Tt^d)} + C, \]
  for a.e.~time slice $s \in (0,T)$. We raise both sides to the power of $\alpha$, integrate over $s \in [0,T]$, and use \eqref{Du.bound} to deduce \eqref{uminus-integral-bound-v2}. 
  
  Now we note a particular consequence of \eqref{uminus-integral-bound-v2}. Let $\tau \in (0,T)$. Then
  \begin{equation}\label{alpha-star-space-minimize-over-time}
      \inf_{r\in (0,\tau)} \|u_\epsi(r,\cdot)\|_{L^{\alpha^*}(\Tt^d)} \leq \tau^{-\frac{1}{\alpha}}\left(\int_0^\tau \bigl( \|u_\epsi(r,\cdot)\|_{L^{\alpha^*}(\Tt^d)} \bigr)^\alpha\,dr\right)^{\frac{1}{\alpha}} \leq C\tau^{-\frac{1}{\alpha}}.
  \end{equation}

  Next, for any $s \in [\tau,T]$ and $r \in (0,\tau)$, we integrate both sides of \eqref{hjb-estimated-wo-Du} from $r$ to $s$ for a.e.~$x\in\Tt^d$ to get
     \[u_\epsi(r,x)-u_\epsi(s,x)\leq C\int_r^s (m_\epsi^{\beta-1}+V(t,x))\,dt \qquad\text{a.e. on } \Tt^d.\]
     Then we move the first term of the left-hand side to the right, take the positive part, and expand the bounds of integration to $[0,T]$ to obtain
     \[u_\epsi^-(s,x) \leq u_\epsi^-(r,x) + C\int_0^T (m_\epsi^{\beta-1}+V(t,x))\,dt  \qquad\text{a.e. on } \Tt^d.\]
    The integral term on the right-hand side is independent of $s$ and $r$ and it is uniformly bounded in $L^{\beta'}(\Tt^d)$, hence also in $L^q(\Tt^d)$ since $q\leq \beta'$. Thus, taking the $L^q(\Tt^d)$ norm of both sides, we find
  \[ \|u_\epsi^-(s,\cdot)\|_{L^q(\Tt^d)} \leq \|u_\epsi(r,\cdot)\|_{L^q(\Tt^d)} + C \leq C\|u_\epsi(r,\cdot)\|_{L^{\alpha^*}(\Tt^d)} + C, \]
  where we also used $q\leq \alpha^*$. Since the left-hand side is independent of $r$, we can take the infimum over $r\in (0,\tau)$ on the right-hand side and use \eqref{alpha-star-space-minimize-over-time} to get
  \[\|u_\epsi^-(s,\cdot)\|_{L^q(\Tt^d)} \leq C(\tau^{-\frac{1}{\alpha}} + 1).\]  
  Now we conclude \eqref{uminus-forward-bound} by raising both sides to the power of $q$.
\end{proof}

\subsection{Passing Regularized Solutions to the Limit}\label{sec:minty}

Here we show that the family $(m_\epsi,u_\epsi)$ provided by Proposition~\ref{prop:regularized-solutions} converges, up to a subsequence, to a pair $(m,u)$ that is a solution to Problem~\ref{prob.vi} in the sense of Definition~\ref{def.weak.sol}. Throughout the proof, $\mathcal{M}(\Tt^d)$ and $\mathcal{M}(Q)$ denote the spaces of Radon measures on $\Tt^d$ and $Q= (0,T)\times\Tt^d$, respectively.

\begin{proof}[Proof of Theorem~\ref{weak.exists}]
Let $(m_\epsi,u_\epsi)\in\mathcal{K}$ be given by Proposition~\ref{prop:regularized-solutions}. By~\eqref{m.bound}, the family $\{m_\epsi\}$ is uniformly bounded in $L^\beta(Q)$, while by \eqref{ut-minus-bound}, \eqref{ut-plus-bound}, \eqref{Du.bound}, the terminal condition $u_\epsi(T,\cdot)=u_T$, and the Poincar\'e--Wirtinger trace inequality for $BV(Q)$, the family $\{u_\epsi\}$ is uniformly bounded in $BV(Q)$. Therefore, after extracting a subsequence,
\[
m_\epsi\rightharpoonup m \quad\text{in }L^\beta(Q), 
\qquad
u_\epsi\stackrel{*}{\rightharpoonup}u \quad\text{in }BV(Q).
\]
Moreover, again up to a subsequence,
\[
Du_\epsi\rightharpoonup Du \quad\text{in }L^\alpha(Q;\RR^d),
\qquad
((u_\epsi)_t)^-\rightharpoonup g \quad\text{in }L^{\beta'}(Q)
\]
for some $g\geq 0$. Then
\[((u_\epsi)_t)^+ = (u_\epsi)_t + ((u_\epsi)_t)^- \stackrel{*}{\rightharpoonup} du_t + g\,dxdt \quad\text{in } \mathcal{M}(Q).\]
 Therefore,  $du_t+g\,dxdt$ is nonnegative, thus $(u_t)^-\leq g$, implying $(u_t)^-\in L^{\beta'}(Q)$. Consequently, $(m,u)$ satisfies~\eqref{integrability.further.m}, \eqref{integrability.further.du}, and~\eqref{integrability.further.ut}. It remains to prove that $(m,u)$ satisfies~\eqref{weak-sol} and~\eqref{integrability.further.mdu}. Under Assumption~\ref{onH.high-order-p.has-no-m}, we also prove~\eqref{integrability.further.H}.

Before that, we record the possible loss of trace at the endpoints. By the $s=0$ cases of~\eqref{u.upper.beta'-uniform-in-time} and~\eqref{u-lower-L1-uniform-in-time}, the traces $u_\epsi(0,\cdot)$ are uniformly bounded in $\mathcal{M}(\Tt^d)$. Hence, after passing to a further subsequence,
\[
u_\epsi(0,\cdot)\stackrel{*}{\rightharpoonup}u_0
\quad\text{in }\mathcal{M}(\Tt^d),
\]
where the positive part $u_0^+$ is absolutely continuous with density in $L^{\beta'}(\Tt^d)$. However, $u_0$ need not coincide with the trace $u(0,\cdot)$. Similarly, we do not necessarily have $u(T,\cdot)=u_T$, even though $u_\epsi(T,\cdot)=u_T$ for every $\epsi>0$.

\medskip\noindent\textbf{Proof of \eqref{weak-sol}:} Combining Proposition~\ref{prop:regularized-solutions} with the monotonicity of $A_\epsi$, we find that for all $(\mu,\upsilon)\in \mathcal{K}$, i.e.~for all $\mu\in L^{\beta}(Q;\RR^+_0)$ and all $\upsilon\in W^{1,\gamma}(Q)$ satisfying $\upsilon(T,x) = u_T(x)$, we have
\begin{equation}\label{variational-ineq-epsi}
  \begin{aligned}
  0 & \leq \bigl\langle A_\epsi[m_\epsi, u_\epsi], (\mu-m_\epsi, \upsilon-u_\epsi) \bigr\rangle \leq \bigl\langle A_\epsi[\mu, \upsilon], (\mu-m_\epsi, \upsilon-u_\epsi) \bigr\rangle \\
  & = \int_Q(\mu-m_\epsi)\bigl(\upsilon_t - H(t,x,D\upsilon,\mu)\bigr)\,dxdt \\
  & +\int_Q\bigl(-\mu (\upsilon_t-(u_\epsi)_t) + \mu D_pH(t,x,D\upsilon,\mu)\cdot (D\upsilon-Du_\epsi) \bigr)\,dxdt\\
  &\qquad -\int_{\Tt^d}m_0(x)(\upsilon(0,x)-u_\epsi(0,x))\,dx \\
  & \qquad + \epsi\int_Q \bigl(|D\upsilon|^{\gamma} - |D\upsilon|^{\gamma-2}D\upsilon\cdot Du_\epsi + |\upsilon_t|^{\gamma} - |\upsilon_t|^{\gamma-2}\upsilon_t(u_\epsi)_t\bigr)\, dxdt.
\end{aligned}
\end{equation}
For $\mu\in C(\bar{Q};\RR^+_0)$ and $\upsilon\in C^1(\bar{Q})$ satisfying $\upsilon(T,x)=u_T(x)$, we pass to the limit $\epsi\to 0$ in \eqref{variational-ineq-epsi} as follows. The first integral passes to the limit by the weak convergence $m_\epsi\rightharpoonup m$ in $L^\beta(Q)$, since $\upsilon_t - H(t,x,D\upsilon,\mu) \in L^{\beta'}(Q)$ by \eqref{eq:assH.lower.simple} and \eqref{res:H.upper.simple}. The second term in the second integral passes to the limit by the weak convergence $Du_\epsi\rightharpoonup Du$ in $L^\alpha(Q;\RR^d)$, because $\mu D_pH(t,x,D\upsilon,\mu)\in L^{\alpha'\beta'}(Q;\RR^d) \subset L^{\alpha'}(Q;\RR^d)$ by~\eqref{eq:assH.DpH.upper.simple}. The third integral passes to the limit by $u_\epsi(0,\cdot)\stackrel{*}{\rightharpoonup}u_0$ since $m_0\in C(\Tt^d)$. Finally, the $\epsi$-regularization terms vanish by~\eqref{ut-minus-bound}, \eqref{ut-plus-bound}, and \eqref{Du.bound}. It remains only to identify the limit of the first term in the second integral, which involves $(u_\epsi)_t$. 

To identify this limit, we first assume $\mu\in C^1(\bar{Q};\RR^+_0)$. Then, the trace Green formula gives
\begin{equation}\label{green-limit-argument}
    \begin{aligned}
\int_Q \mu (u_\epsi)_t\,dxdt
&= -\int_Q u_\epsi\mu_t\,dxdt + \int_{\Tt^d}\mu(T,x)u_T(x)\,dx
-\int_{\Tt^d}\mu(0,x)u_\epsi(0,x)\,dx
 \\
&\to -\int_Q u\mu_t\,dxdt + \int_{\Tt^d}\mu(T,x)u_T(x)\,dx
-\int_{\Tt^d}\mu(0,x)\,du_0
 \\
& \qquad = \int_Q \mu\,du_t
+\int_{\Tt^d}\mu(T,x)(u_T(x)-u(T,x))\,dx \\
& \qquad\qquad +\int_{\Tt^d}\mu(0,x)(u(0,x)\,dx -du_0).
\end{aligned}
\end{equation}
Even though $\mu_t$ appears in the derivation in \eqref{green-limit-argument}, it does not appear in the extreme terms. Accordingly, the result is valid for all $ \mu\in C(\bar{Q};\RR^+_0) $ by an approximation argument.

Consequently, we pass to the limit in \eqref{variational-ineq-epsi} to get
\begin{equation*}
  \begin{aligned}
  0 \leq & \int_Q(\mu-m)\bigl(\upsilon_t - H(t,x,D\upsilon,\mu)\bigr)\,dxdt \\
  & + \int_Q\bigl(-\mu \upsilon_t + \mu D_pH(t,x,D\upsilon,\mu)\cdot (D\upsilon-Du) \bigr)\,dxdt\\
  &\qquad +\int_Q\mu\,du_t + \int_{\Tt^d} \mu(T,x)(u_T(x)-u(T,x))\,dx \\
  & \qquad + \int_{\Tt^d} (\mu(0,x)-m_0(x))(u(0,x)\,dx-du_0) -\int_{\Tt^d}m_0(x)(\upsilon(0,x)-u(0,x))\,dx 
\end{aligned}
\end{equation*}
for all $(\mu, \upsilon) \in C(\bar{Q};\RR^+_0) \times C^1(\bar{Q})$ satisfying $\upsilon(T,x) = u_T(x)$. Therefore,~\eqref{weak-sol} holds for all $(\mu, \upsilon) \in C(\bar{Q};\RR^+_0) \times C^1(\bar{Q})$ satisfying $\mu(0,x) = m_0(x)$ and $\upsilon(T,x) = u_T(x)$.

\medskip\noindent\textbf{Proof of \eqref{integrability.further.mdu}:} Fix $(s,y)\in \bar{Q} = [0,T]\times\Tt^d$. For fixed $\epsi>0$, we let $\epsi'\to0$ in~\eqref{mDu-advanced}. Since $|Du_\epsi^{[s,y]}|^\alpha\in L^{\beta'}((0,T-s)\times\Tt^d)$, the weak convergence $m_{\epsi'}\rightharpoonup m$ in $L^\beta(Q)$ gives
\[
\int_0^{T-s}\int_{\Tt^d}m|Du_\epsi^{[s,y]}|^\alpha\,dxdt\leq C.
\]
Now let $\epsi\to 0$. For each positive integer $n$, set $m_n:=\min\{m,n\}$. Since $Du_\epsi^{[s,y]}\rightharpoonup Du^{[s,y]}$ in $L^\alpha((0,T-s)\times\Tt^d;\RR^d)$, weak lower semicontinuity yields
\[
\int_0^{T-s}\int_{\Tt^d}m_n|Du^{[s,y]}|^\alpha\,dxdt
\leq
\liminf_{\epsi\to0}\int_0^{T-s}\int_{\Tt^d}m_n|Du_\epsi^{[s,y]}|^\alpha\,dxdt
\leq C.
\]
Letting $n\to\infty$ and using monotone convergence gives
\[
\int_0^{T-s}\int_{\Tt^d}m|Du^{[s,y]}|^\alpha\,dxdt\leq C.
\]
After the change of variables $(t,x)\mapsto(t+s,x+y)$, this is exactly~\eqref{integrability.further.mdu}; the constant is independent of $(s,y)$.

\medskip\noindent\textbf{Proof of \eqref{integrability.further.H}:} 
We have
\[\begin{aligned}
 & h_\epsi:=-(u_\epsi)_t+H(t,x,Du,0) + D_pH(t,x,Du,0)\cdot (Du_\epsi-Du) \\
 & \qquad \leq -(u_\epsi)_t+H(t,x,Du_\epsi,0) \\
& \qquad\qquad \leq H(t,x,Du_\epsi,0) - H(t,x,Du_\epsi,m_\epsi) \\
& \qquad\qquad\qquad \leq C(|Du_\epsi|^{\alpha\frac{(\beta-1)}{\beta}} + m_\epsi^{{\beta-1}}+V(t,x)) \qquad \text{a.e.~in } Q,
\end{aligned}
\]
where the first inequality uses the convexity of $H(t,x,\cdot,0)$ from Remark~\ref{rmk:cxty}, the second uses the regularized Hamilton--Jacobi equation~\eqref{hjb-strong-epsi}, and the third uses the $H$-deviation bound \eqref{eq:high-order-p.has-no-m} from Assumption~\ref{onH.high-order-p.has-no-m}.

Thus, \eqref{Du.bound} and~\eqref{m.bound} imply $\|h_\epsi^+\|_{L^{\beta'}(Q)}\leq C$. Hence, we have $h_\epsi^+\rightharpoonup g$ in $L^{\beta'}(Q)$ for some $g\geq 0$, after passing to a subsequence.
On the other hand,
\[h_\epsi\stackrel{*}{\rightharpoonup} h:=-u_t+H(t,x,Du,0) \quad\text{in }\mathcal{M}(Q),\]
because $Du_\epsi\rightharpoonup Du$ in $L^\alpha(Q;\RR^d)$ and $D_pH(t,x,Du,0) \in L^{\alpha'}(Q;\RR^d)$ by \eqref{eq:assH.DpH.upper.simple}. Now, $h_\epsi\leq h_\epsi^+$. Therefore, $dh\leq g\,dxdt$ as a measure. Thus, $h^+\leq g$, implying $h^+\in L^{\beta'}(Q)$. This is exactly~\eqref{integrability.further.H}.
\end{proof}

\section{Proof of Theorem~\ref{weak.is.strong}}\label{sec:proof-weak-is-strong}

In this section, we present the proof of Theorem~\ref{weak.is.strong}, demonstrating that the VI-solutions obtained via Theorem~\ref{weak.exists} are genuine MFG-solutions. Our strategy is to systematically enlarge the class of admissible test functions in the variational inequality~\eqref{weak-sol}. In Lemma~\ref{lem:mu-improved}, we extend the class of test densities $\mu$ to bounded Borel measurable functions, which allows us to extract the necessary boundary conditions and the non-negativity of the singular part of the time derivative. Subsequently, in Lemma~\ref{lem:upsilon-improved}, we enlarge the class of test value functions $\upsilon$ to functions of bounded variation. Finally, in Lemmas~\ref{lem:vi-gives-hjb} through \ref{lem:transport-tested-upsilon.minus.u}, we further enlarge the class of test densities to $L^\beta(Q)$. This three-step bootstrapping process ultimately enables us to substitute the true population density $m$ and the value function $u$ into the variational framework as test functions, rigorously recovering both the Hamilton--Jacobi and transport equations. We remark that throughout this section every integral that appears is finite; while this is not always immediate, it follows in each case from the bounds established in the preceding steps of the argument.

\subsection{Enlarging Test Densities}

We begin our bootstrapping procedure by expanding the admissible test densities $\mu$ from continuous functions to bounded Borel measurable functions. This enlargement is a critical first step, as it allows us to evaluate the variational inequality using indicator functions, thereby localizing the inequality to temporal boundaries and the concentration set of the singular measure $u_t^{\text{s}}$. Doing so isolates the non-negativity of the singular part of the time derivative and the terminal inequality, establishing essential constraints for accommodating $BV$ test functions in the subsequent steps.

\begin{lemma}\label{lem:mu-improved}
Consider the setting of Problem~\ref{prob.mfg}. Suppose Assumptions \ref{onH.monotone} and \ref{onH.powergrowth} hold. Let $(m,u)$ be a solution to Problem~\ref{prob.vi} in the sense of Definition~\ref{def.weak.sol}, such that $m\in L^{\beta}(Q;\RR^+_0)$ and $Du\in L^\alpha(Q; \RR^d)$. Then,
\begin{enumerate}
  \item We have
  \begin{equation}\label{u.jumps.forward}
     0\leq u_t^{\text{s}} \qquad\text{and}\qquad u(T,x)\leq u_T(x);
  \end{equation}
\smallskip
  
  \item The inequality
\begin{equation}\label{weak-sol-mu-improved-Linfty}
\begin{aligned}
0 \leq & \int_Q(\mu-m)\bigl(\upsilon_t - H(t,x,D\upsilon,\mu)\bigr)\,dxdt\\
&+\int_Q\bigl(\mu (u_t^{\text{ac}}-\upsilon_t) + \mu D_pH(t,x,D\upsilon,\mu)\cdot (D\upsilon-Du) \bigr)\,dxdt\\
&\qquad -\int_{\Tt^d}m_0(x)(\upsilon(0,x)-u(0,x))\,dx
\end{aligned}
\end{equation}
holds for all $(\mu, \upsilon) \in L^{\infty}(Q; \RR^+_0)\times C^1(\bar{Q})$ satisfying $\upsilon(T,x) = u_T(x)$.
\end{enumerate}
\end{lemma}
\begin{proof}
Consider an arbitrary bounded, Borel measurable function $\Tilde{\mu}\colon (0,T]\times\Tt^d \to \RR^+_0$. We apply Lemma~\ref{lem:dense.restr-cont} for some space-time cylinder $U$ containing $\bar{Q}$, e.g.~$U = (-\delta, T+\delta)\times \Tt^d$. We take the Borel sets $B_1 := \{T\}\times \Tt^d\subset U$ and $B_2 = B_3 := Q\subset U$ with $\rho_1$ as the $d$-dimensional Lebesgue measure, $\rho_2$ as the $(d+1)$-dimensional Lebesgue measure, and $\rho_3$ as the total variation measure $|u_t|$. Then we take $K = \{0\}\times \Tt^d$ with the continuous boundary condition $g(0,x) = m_0(x)$. Then, Lemma~\ref{lem:dense.restr-cont} implies that there exists a uniformly bounded sequence of $\mu_n \in C(\bar{Q}; \RR^+_0)$ satisfying $\mu_n(0,x) = m_0(x)$ and converging to $\Tilde{\mu}$ as follows:
\begin{itemize}
  \item $\mu_n(T,\cdot)\to\Tilde{\mu}(T,\cdot)$ a.e.~in~$\Tt^d$ with respect to the $d$-dimensional Lebesgue measure,
  \item $\mu_n \to \Tilde{\mu}$ a.e.~in~$Q$ with respect to the $(d+1)$-dimensional Lebesgue measure,
  \item $\mu_n \to \Tilde{\mu}$ a.e.~in~$Q$ with respect to the Radon measure $|u_t|$.
\end{itemize}
Because $(m,u)$ is a 
VI-solution to Problem~\ref{prob.vi} in the sense of Definition~\ref{def.weak.sol}, we have
\begin{equation}\label{weak-sol-mu.n.seq}
\begin{aligned}
0 \leq & \int_Q(\mu_n-m)\bigl(\upsilon_t - H(t,x,D\upsilon,\mu_n)\bigr)\,dxdt\\
&+\int_Q\bigl(-\mu_n \upsilon_t + \mu_n\,D_pH(t,x,D\upsilon,\mu_n)\cdot (D\upsilon-Du) \bigr)\,dxdt + \int_Q \mu_n \,du_t\\
&\qquad -\int_{\Tt^d}m_0(x)(\upsilon(0,x)-u(0,x))\,dx + \int_{\Tt^d} \mu_n(T,x)(u_T(x)-u(T,x))\,dx,
\end{aligned}
\end{equation}
for every $\upsilon\in C^1(\bar{Q})$ satisfying $\upsilon(T,x) = u_T(x)$. We pass to the limit in~\eqref{weak-sol-mu.n.seq} using the Dominated Convergence Theorem to obtain
\begin{equation}\label{weak-sol-mu-borel}
\begin{aligned}
0 \leq & \int_Q(\Tilde{\mu}-m)\bigl(\upsilon_t - H(t,x,D\upsilon,\Tilde{\mu})\bigr)\,dxdt\\
& + \int_Q\bigl(\Tilde{\mu} (u_t^{\text{ac}}-\upsilon_t) + \Tilde{\mu}\,D_pH(t,x,D\upsilon,\Tilde{\mu})\cdot (D\upsilon-Du) \bigr)\,dxdt\\
&\qquad -\int_{\Tt^d}m_0(x)(\upsilon(0,x)-u(0,x))\,dx \\
&\qquad\qquad +\int_Q \Tilde{\mu} \,du_t^{\text{s}} + \int_{\Tt^d} \Tilde{\mu}(T,x)(u_T(x)-u(T,x))\,dx, 
\end{aligned}
\end{equation}
where we substituted the Lebesgue decomposition $du_t = u_t^{\text{ac}}\,dxdt + du_t^{\text{s}}$.

Note that~\eqref{weak-sol-mu-borel} holds for all bounded, Borel measurable $\Tilde{\mu}\colon (0,T]\times\Tt^d \to \RR^+_0$ and $\upsilon\in C^1(\bar{Q})$ satisfying $\upsilon(T,x) = u_T(x)$. Let $S \subset Q$ be a Borel set of zero Lebesgue measure on which
$u_t^{\text{s}}$ is concentrated. The three sets
\[
Q\setminus S, \qquad S, \qquad \{T\}\times\Tt^d
\]
are pairwise disjoint and partition the domain $(0,T]\times\Tt^d$ on
which $\Tilde{\mu}$ is defined. Consider arbitrary bounded, Borel measurable functions $f_1\colon S\to \RR^+_0$ and $f_2\colon\Tt^d\to \RR^+_0$. Then, for a given $\mu\in L^\infty(Q;\RR^+_0)$, choose any bounded Borel representative and construct $\tilde{\mu}$ such that
\[\tilde{\mu} = \mu \quad\text{a.e.~on } Q\setminus S, \qquad \tilde{\mu} = f_1 \quad\text{on } S, \qquad \tilde{\mu}(T,\cdot) = f_2 \quad\text{on } \Tt^d.\]
Hence, applying \eqref{weak-sol-mu-borel} with this $\tilde{\mu}$ and rearranging, we get
\begin{equation}\label{weak-sol-mu-borel-ac}
\begin{aligned}
& \int_Q(\mu-m)\bigl(\upsilon_t - H(t,x,D\upsilon,\mu)\bigr)\,dxdt\\
& \qquad +\int_Q\bigl(\mu (u_t^{\text{ac}}-\upsilon_t) + \mu\,D_pH(t,x,D\upsilon,\mu)\cdot (D\upsilon-Du) \bigr)\,dxdt\\
&\qquad\qquad -\int_{\Tt^d}m_0(x)(\upsilon(0,x)-u(0,x))\,dx \\
& \qquad\qquad\qquad \geq - \left(\int_S f_1 \,du_t^{\text{s}} + \int_{\Tt^d} f_2(x)(u_T(x)-u(T,x))\,dx\right)
\end{aligned}
\end{equation}
for any $(\mu, \upsilon) \in L^{\infty}(Q; \RR^+_0)\times C^1(\bar{Q})$ satisfying $\upsilon(T,x) = u_T(x)$ and any bounded, Borel measurable $f_1\colon S\to \RR^+_0$ and $f_2\colon\Tt^d\to \RR^+_0$.

Taking $f_1 = f_2 = 0$ in \eqref{weak-sol-mu-borel-ac}, we conclude \eqref{weak-sol-mu-improved-Linfty}. On the other hand, fixing $\mu$, $\upsilon$, and $f_2$ while multiplying $f_1$ by an arbitrarily large positive constant, we find $u_t^{\text{s}}\geq 0$. Similarly, fixing $\mu$, $\upsilon$, and $f_1$ while multiplying $f_2$ by an arbitrarily large positive constant, we find $u(T,x)\leq u_T(x)$.
\end{proof}

\subsection{Enlarging Test Value Functions}

Having secured the required boundary conditions and the sign of the singular measure $u_t^{\text{s}}$ in the previous subsection, we now turn our attention to the test value functions. To eventually test the variational inequality against the value function $u$ itself, we must enlarge the space of admissible test functions $\upsilon$ from $C^1(\bar{Q})$ to $BV(Q)$. Because functions in $BV(Q)$ lack forward-in-time continuity, standard symmetric mollification would disrupt the initial and terminal boundary traces. To safely navigate this, we introduce an asymmetric space-time mollification scheme, verifying its convergence properties in two auxiliary propositions before passing to the limit in the variational inequality.

\begin{lemma}\label{lem:upsilon-improved}
 Consider the setting of Problem~\ref{prob.mfg}. Suppose Assumptions \ref{onH.monotone}, \ref{onH.powergrowth} and \ref{onH.highest-order-p.has-no-txm} hold. Let $(m,u)$ be a solution to Problem~\ref{prob.vi} in the sense of Definition~\ref{def.weak.sol}, such that $m\in L^{\beta}(Q;\RR^+_0)$ and $Du\in L^\alpha(Q;\RR^d)$. Then,
\begin{equation}\label{weak-sol-upsilon-improved}
\begin{aligned}
0 \leq & \int_Q(\mu-m)\bigl(\upsilon_t^{\text{ac}} - H(t,x,D\upsilon,\mu)\bigr)\,dxdt\\
&+\int_Q\bigl(\mu (u_t^{\text{ac}}-\upsilon_t^{\text{ac}}) + \mu D_pH(t,x,D\upsilon,\mu)\cdot (D\upsilon-Du) \bigr)\,dxdt\\
&\qquad -\int_{\Tt^d}m_0(x)(\upsilon(0,x)-u(0,x))\,dx
\end{aligned}
\end{equation}
for all $(\mu, \upsilon) \in L^{\infty}(Q; \RR^+_0)\times BV(Q)$ satisfying
\begin{align} 
  & D\upsilon\in L^\alpha(Q;\RR^d), \label{Dupsilon.Lalpha} \\
  & 0\leq \upsilon_t^{\text{s}} \quad \text{and}\quad \upsilon(T,x) \leq u_T(x), \label{upsilon.jumps.forward}\\
  & \bigl(-\upsilon_t^{\text{ac}}+H(t,x,D\upsilon + \ell,0)\bigr)^+ \in L^{\beta'}(Q) \enskip\text{for some}\enskip \ell\in L^\infty(Q; \RR^d), \label{upsilon-HJB-upper} \\
  & \sup_{(s,y)\in \bar{Q}} \left(\int_0^{T-s}\int_{\Tt^d} m(t,x)|D\upsilon(t+s,x+y)|^\alpha\,dxdt\right) < \infty. \label{mDupsilon-shift-L1}
\end{align}
\end{lemma}
\begin{remark}
    The spatial gradient perturbation $\ell$ in \eqref{upsilon-HJB-upper} acts as a slack variable. It ensures that additively perturbed test functions of the form $\upsilon = u \pm \varphi$ satisfy the integrability bound by choosing $\ell = \mp D\varphi$, which is a crucial flexibility for extracting the transport equation later.
\end{remark}
\begin{remark}
\label{rmk:m-upsilon-t-integrable}
As a consequence of Lemma~\ref{lem:upsilon-improved}, every admissible
\(\upsilon\) in that lemma satisfies \(m|\upsilon_t^{\mathrm{ac}}|\in L^1(Q)\). Indeed, let \(\ell\in
L^\infty(Q;\RR^d)\) be as in \eqref{upsilon-HJB-upper} and set
\[
A:=-\upsilon_t^{\mathrm{ac}}+H(t,x,D\upsilon+\ell,0).
\]
Then \(A^+\in L^{\beta'}(Q)\). Moreover, by \eqref{eq:assH.lower.simple},
\[
H(t,x,D\upsilon+\ell,0)^-\leq C V(t,x).
\]
Hence
\[
(-\upsilon_t^{\mathrm{ac}})^+
\leq A^+ + H(t,x,D\upsilon+\ell,0)^-
\]
and therefore \(m(-\upsilon_t^{\mathrm{ac}})^+\in L^1(Q)\), by
\eqref{integrability.further.m}, \eqref{eq:assH.potential}, and Hölder's
inequality.

It remains to control \(m(\upsilon_t^{\mathrm{ac}})^+\). Taking
\(\mu=0\) in \eqref{weak-sol-upsilon-improved}, we obtain
\[
0\le
\int_Q m\bigl(-\upsilon_t^{\mathrm{ac}}+H(t,x,D\upsilon,0)\bigr)\,dxdt
-\int_{\Tt^d}m_0(x)(\upsilon(0,x)-u(0,x))\,dx.
\]
The boundary term is finite, and
\[
m|H(t,x,D\upsilon,0)|
\leq C m(|D\upsilon|^\alpha+V(t,x))\in L^1(Q)
\]
by \eqref{mDupsilon-shift-L1}, \eqref{integrability.further.m}, and
\eqref{eq:assH.potential}. Since we already know
\(m(-\upsilon_t^{\mathrm{ac}})^+\in L^1(Q)\), the preceding inequality forces
\(m(\upsilon_t^{\mathrm{ac}})^+\in L^1(Q)\). Thus
\(m|\upsilon_t^{\mathrm{ac}}|\in L^1(Q)\).
\end{remark}

To prove Lemma~\ref{lem:upsilon-improved}, we approximate the test function $\upsilon$ with smooth functions. The lack of forward time regularity requires an asymmetric mollification scheme. We define this sequence and isolate its key convergence properties into two auxiliary propositions before completing the main proof.

Let $\upsilon \in BV(Q)$ satisfy the conditions of Lemma~\ref{lem:upsilon-improved}. We extend \(\upsilon\) to
\((0,\infty)\times\Tt^d\) by setting 
\begin{equation}\label{eq:upsilon-extension}
\upsilon^*(t,x):=
\begin{cases}
\upsilon(t,x), & 0<t<T,\\
u_T(x), & t>T,
\end{cases}
\end{equation}
for $x\in \Tt^d$. This extension belongs to
\(L^1_{\mathrm{loc}}((0,\infty)\times\Tt^d)\). Recalling the mollification kernel
\(K_{\xi,\tau}(t',x';t,x)\) defined in \eqref{kernel-mollification}, we define
\(\upsilon_{\xi,\tau}\in C^1(\bar Q)\) by
\begin{equation}\label{eq:upsilon-mollified}
    \upsilon_{\xi, \tau}(t,x) = \int_t^\infty \int_{\Tt^d} \upsilon^*(t',x') K_{\xi, \tau}(t',x'; t,x) \,dx'dt'.
\end{equation}
Next, we take a sequence $\xi_n \downarrow 0$. Because the total variation of the Radon measure $\upsilon_t$ on the open slice $(0, \tau) \times \Tt^d$ vanishes as $\tau \to 0^+$, we can choose another sequence $\tau_n \downarrow 0$ such that
\begin{equation}\label{xi-tau-choice}
\int_0^{\tau_n}\int_{\Tt^d}d|\upsilon_t| < (\xi_n)^{d+1},
\end{equation}
for each $n$. Then, for such a sequence, we set
\begin{equation}\label{upsilon_n.defined}
\upsilon_n(t,x) := \upsilon_{\xi_n,\tau_n}(t,x) + u_T(x) -\int_{\Tt^d} u_T(x')\cdot \frac{1}{(\xi_n)^d} \overline{\zeta}\left(\frac{x'-x}{\xi_n}\right) \,dx'.
\end{equation}
It is clear that $\upsilon_n\in C^1(\bar{Q})$, and~\eqref{eq:upsilon-mollified} ensures the exact terminal condition $\upsilon_n(T,x) = u_T(x)$.

\begin{proposition}[Approximation of Initial Trace and Spatial Derivative]\label{lem:trace-approx}
Under the assumptions of Lemma~\ref{lem:upsilon-improved}, let $\upsilon_n$ be constructed as in~\eqref{upsilon_n.defined}. Then, as $n\to \infty$,
\begin{align}
  D\upsilon_n \to D\upsilon \quad & \text{in } L^\alpha(Q;\RR^d), \label{Dupsilon.limit-aux} \\
  \upsilon_n(0,\cdot) \to \upsilon(0,\cdot) \quad &\text{in } L^1(\Tt^d). \label{upsilon-trace-limit-aux}
\end{align}
\end{proposition}

\begin{proof}
  To establish \eqref{Dupsilon.limit-aux}, note that~\eqref{upsilon_n.defined} and~\eqref{eq:upsilon-mollified} imply
  \begin{equation}\label{Dupsilon_n.formula}
     \begin{aligned}
     D\upsilon_n(t,x) & = \int_t^\infty\int_{\Tt^d} D\upsilon^*(t',x')K_{{\xi_n}, {\tau_n}}(t',x'; t,x) \,dx'dt' \\
     & \qquad + Du_T(x) - \int_{\Tt^d}Du_T(x')\cdot \frac{1}{(\xi_n)^d} \overline{\zeta}\left(\frac{x'-x}{\xi_n}\right) \,dx'.
     \end{aligned}
  \end{equation}
The first integral on the right-hand side is a space-time mollification of $D\upsilon^*$, which converges to $D\upsilon$ in $L^\alpha(Q;\RR^d)$ since $D\upsilon\in L^\alpha(Q;\RR^d)$; the remaining terms $$Du_T(x)-\int_{\Tt^d}Du_T(x')\frac{1}{(\xi_n)^d}\overline{\zeta}\left(\frac{x'-x}{\xi_n}\right)\,dx'$$ converge to zero uniformly because $u_T\in C^1(\Tt^d)$. Hence $D\upsilon_n\to D\upsilon$ in $L^\alpha(Q;\RR^d)$.

  To prove \eqref{upsilon-trace-limit-aux}, consider the smooth vector field $\psi\colon [0,\infty)\times \RR^d \to \RR\times \RR^{d}$ defined by
  \[\psi(t,x) = -\mathbf{\hat{e}_t}\cdot\left(\int_{t}^\infty \hat{\zeta}(s)\,ds\right) \cdot \overline{\zeta}(x),\]
where $\mathbf{\hat{e}_t} := (1, \mathbf{0}) \in \RR\times\RR^d$ is the standard unit vector pointing in the time direction. Given $x\in\Tt^d$ and $\xi,\tau>0$, we have
  \[\div_{(t',x')}\left(\frac{1}{\xi^d}\psi\Bigl(\frac{t'}{\tau},\frac{x'-x}{\xi}\Bigr)\right) = \frac{1}{\tau} \hat{\zeta}\left(\frac{t'}{\tau}\right)\cdot \frac{1}{\xi^d} \overline{\zeta}\left(\frac{x'-x}{\xi}\right) = K_{{\xi}, {\tau}}(t',x'; 0,x).\]
  Therefore, by~\eqref{eq:upsilon-mollified} and integration by parts in $(t',x')$, we obtain
  \begin{equation}\label{upsilon-mollified-at-t.equals.zero}
     \begin{aligned}
     \upsilon_{\xi,\tau}(0,x) & = \int_0^\infty \int_{\Tt^d} \upsilon^*(t',x')\cdot \div_{(t',x')}\left(\frac{1}{\xi^d}\psi\Bigl(\frac{t'}{\tau},\frac{x'-x}{\xi}\Bigr)\right) \,dx'dt' \\
     & = \int_0^\infty \int_{\Tt^d} \left(\int_{t'/\tau}^\infty \hat{\zeta}(s)\,ds\right) \cdot \frac{1}{\xi^d}\overline{\zeta}\left(\frac{x'-x}{\xi}\right)\,d(\upsilon^*(t',x'))_{t'} \\
     & + \int_{\Tt^d} 
     \upsilon(0,x')\cdot\frac{1}{\xi^d}\overline{\zeta}\left(\frac{x'-x}{\xi}\right)\,dx'.
  \end{aligned}
  \end{equation}
  
  Now, we estimate the second line of~\eqref{upsilon-mollified-at-t.equals.zero}, using the fact that
  \[\int_{t'/\tau}^\infty \hat{\zeta}(s)\,ds \leq 1 \qquad\text{and}\qquad \int_{t'/\tau}^\infty \hat{\zeta}(s)\,ds = 0 \text{ for } t'>\tau.\]
  This gives, for $\tau<T$, that
  \begin{equation*}
     \begin{aligned}
     & \left|\upsilon_{\xi,\tau}(0,x)-\int_{\Tt^d} 
     \upsilon(0,x')\cdot\frac{1}{\xi^d}\overline{\zeta}\left(\frac{x'-x}{\xi}\right)\,dx'\right|\\
     & \qquad\qquad \leq \int_0^\tau \int_{\Tt^d} \frac{1}{\xi^d}\overline{\zeta}\left(\frac{x'-x}{\xi}\right)\, d\left|(\upsilon(t',x'))_{t'}\right| \leq \frac{1}{\xi^d}\lVert\overline{\zeta}\rVert_{L^\infty} \int_0^\tau \int_{\Tt^d}\, d|\upsilon_t|.
  \end{aligned}
  \end{equation*}
  Hence~\eqref{xi-tau-choice} implies
  \begin{equation*}
      \left|\upsilon_{\xi_n,\tau_n}(0,x)-\int_{\Tt^d} \upsilon(0,x')\cdot\frac{1}{(\xi_n)^d}\overline{\zeta}\left(\frac{x'-x}{\xi_n}\right)\,dx'\right|\leq \xi_n\lVert\overline{\zeta}\rVert_{L^\infty},
  \end{equation*}
  which converges to zero uniformly over~$\Tt^d$. Moreover, we observe that
  \[\upsilon_n(0,x)-\upsilon_{\xi_n,\tau_n}(0,x) = u_T(x) -\int_{\Tt^d} u_T(x')\cdot \frac{1}{(\xi_n)^d} \overline{\zeta}\left(\frac{x'-x}{\xi_n}\right) \,dx'\]
  converges to zero uniformly, while
  \begin{equation*}
    \upsilon(0,x)-\int_{\Tt^d} \upsilon(0,x')\cdot\frac{1}{(\xi_n)^d}\overline{\zeta}\left(\frac{x'-x}{\xi_n}\right)\,dx'  
  \end{equation*}
  converges to zero in $L^1(\Tt^d)$. Consequently, $\upsilon_n(0,\cdot) \to \upsilon(0,\cdot)$ in $L^1(\Tt^d)$.
\end{proof}

With the convergence of the spatial gradients and initial traces established, we now turn to the Hamilton--Jacobi term. A direct limit passage is obstructed because the product of the density $m \in L^\beta(Q)$ with the singular measure $\upsilon_t^{\text{s}}$ is \emph{a priori} ill-defined. The following proposition circumvents this difficulty by exploiting the convexity of the Hamiltonian and the non-negativity of the temporal jumps ($\upsilon_t^{\text{s}} \geq 0$) to establish an asymptotic upper bound. This allows us to safely drop the singular mass and control the limit entirely in terms of the absolutely continuous density $\upsilon_t^{\text{ac}}$.

\begin{proposition}[Asymptotic Upper Bound for the Hamiltonian]\label{lem:hjb-limsup} Under the assumptions of Lemma~\ref{lem:upsilon-improved}, let $\upsilon_n$ be as in~\eqref{upsilon_n.defined}. Then
\begin{equation}\label{m-times-hjb.upsilon-limit}
\limsup_{n\to\infty} \int_Q m\bigl(-(\upsilon_n)_t + H(t,x,D\upsilon_n,\mu)\bigr)\,dxdt \leq \int_Q m\bigl(-\upsilon_t^{\text{ac}} + H(t,x,D\upsilon,\mu)\bigr)\,dxdt,
\end{equation}
for any $\mu \in L^{\infty}(Q; \RR^+_0)$.
\end{proposition}

\begin{proof}
Recalling \eqref{eq:upsilon-extension},
we observe that
  \[(\upsilon^*)_{t}^{\text{s}} = \begin{cases}
     \upsilon_t^{\text{s}} \qquad & \text{on } Q = (0,T)\times\Tt^d, \\
     (u_T(x)-\upsilon(T,x))\,dx \qquad & \text{on } \{T\}\times\Tt^d, \\
     0 & \text{on } (T,\infty)\times\Tt^d.
  \end{cases}\]
  where $dx$ in the second line denotes the $d$-dimensional Lebesgue measure. Hence, $(\upsilon^*)_{t}^{\text{s}}\geq 0$ by~\eqref{upsilon.jumps.forward}. Therefore, \eqref{upsilon_n.defined} and~\eqref{eq:upsilon-mollified} imply
  \begin{equation}\label{upsilon_t.singular.removed}
     -(\upsilon_n)_t \leq \int_t^\infty \int_{\Tt^d} \bigl(-\upsilon^*(t',x')\bigr)_{t'}^{\text{ac}}\cdot K_{{\xi_n}, {\tau_n}}(t',x'; t,x) \,dx'dt'.
  \end{equation}
  On the other hand, since $H(t,x,\cdot,\mu(t,x))$ is convex as in Remark~\ref{rmk:cxty}, \eqref{Dupsilon_n.formula} implies
  \begin{equation}\label{hjb-upsilon_n.cxty.upper-bound}
  \begin{aligned}
     & H(t,x,D\upsilon_n,\mu) \\
     & \quad \leq \int_t^\infty \int_{\Tt^d}  H\bigl(t,x,D\upsilon^*(t',x')+Du_T(x)-Du_T(x'),\mu(t,x)\bigr) K_{{\xi_n}, {\tau_n}}(t',x'; t,x) \,dx'dt',
  \end{aligned} 
  \end{equation}
  for a.e.~$(t,x)\in Q$. Now, let $\ell$ be as in \eqref{upsilon-HJB-upper}, and
  let $h\colon (0,\infty)\times\Tt^d \to \RR$ be 
  \begin{equation*}
     h = \begin{cases}
       -\upsilon_t^{\text{ac}} + H(t,x,D\upsilon+\ell,\mu), \qquad & \text{on } Q = (0,T)\times\Tt^d, \\
       0 & \text{on } (T,\infty)\times\Tt^d, 
     \end{cases}
  \end{equation*}
and 
\begin{equation*}
       h^{\text{c}}(t',x';t,x) = \begin{cases}
       \begin{aligned}
   & \Bigl(H(t,x,D\upsilon(t',x')+Du_T(x)-Du_T(x'),\mu(t,x))\\
   & \qquad\qquad - H(t',x',D\upsilon(t',x')+\ell(t',x'),\mu(t',x'))\Bigr)
       \end{aligned} \qquad & \text{for } t' < T, \\
       H(t,x,Du_T(x),\mu(t,x)) & \text{for } t' > T.
     \end{cases}
  \end{equation*}
Adding~\eqref{upsilon_t.singular.removed} and~\eqref{hjb-upsilon_n.cxty.upper-bound}, then adding and subtracting appropriate quantities, we obtain
\begin{equation}\label{upsilon_n.hjb.bounded.by.two-terms}
 \begin{aligned}
 & -(\upsilon_n)_t + H(t,x,D\upsilon_n,\mu) \leq \\
 & \qquad \int_t^\infty \int_{\Tt^d} (h(t',x')+h^{\text{c}}(t',x';t,x))\cdot K_{{\xi_n}, {\tau_n}}(t',x'; t,x) \,dx'dt' \qquad\text{a.e.~in }Q.
\end{aligned}
\end{equation}
 
  In view of~\eqref{upsilon_n.hjb.bounded.by.two-terms}, our purpose is to establish the following limits:
\begin{align}
     & \limsup_{n\to\infty} \int_Q m(t,x) \left(\int_t^\infty \int_{\Tt^d} h(t',x') K_{{\xi_n}, {\tau_n}}(t',x'; t,x) \,dx'dt'\right)\,dxdt \nonumber \\
     & \qquad \leq \int_Q m(t,x)(-\upsilon_t^{\text{ac}} + H(t,x,D\upsilon+\ell,\mu))\,dxdt, \label{mh-limsup} 
\end{align}
and
\begin{align}     
     & \lim_{n\to\infty} \int_Q m(t,x) \left(\int_t^\infty \int_{\Tt^d} h^{\text{c}}(t',x';t,x) K_{{\xi_n}, {\tau_n}}(t',x'; t,x) \,dx'dt'\right)\,dxdt \nonumber \\
     & \qquad = \int_Q m(t,x)(H(t,x,D\upsilon,\mu)-H(t,x,D\upsilon+\ell,\mu))\,dxdt. \label{mhc-lim}
  \end{align}

\medskip
 \noindent\textbf{Proof of \eqref{mh-limsup}:} We note that $h^+\in L^{\beta'}((0,\infty)\times\Tt^d)$ by~\eqref{upsilon-HJB-upper} because $H(t,x,p,\cdot)$ is non-increasing as in Remark~\ref{rmk:cxty}. Hence,
  \begin{equation*}
\int_t^\infty \int_{\Tt^d} h^+(t',x') K_{{\xi_n}, {\tau_n}}(t',x'; t,x) \,dx'dt' \to h^+(t,x) \qquad\text{in } L^{\beta'}(Q),
  \end{equation*}
  which yields
  \begin{equation}\label{hplus-limit}
  \begin{aligned}
& \lim_{n\to\infty} \int_Q m(t,x) \left(\int_t^\infty \int_{\Tt^d} h^+(t',x') K_{{\xi_n}, {\tau_n}}(t',x'; t,x) \,dx'dt'\right)\,dxdt \\
& \qquad = \int_Q m(t,x)h^+(t,x)\,dxdt, 
  \end{aligned}
  \end{equation}
  since $m\in L^\beta(Q)$.
  On the other hand, because $h^-\in L^1(Q)$,
  \begin{equation*}
\int_t^\infty \int_{\Tt^d} h^-(t',x') K_{{\xi_n}, {\tau_n}}(t',x'; t,x) \,dx'dt' \to h^-(t,x) \qquad\text{a.e.~in } Q,
  \end{equation*}
  which yields
  \begin{equation}\label{hminus-limit}
  \begin{aligned}
  & \liminf_{n\to\infty} \int_Q m(t,x) \left(\int_t^\infty \int_{\Tt^d} h^-(t',x') K_{{\xi_n}, {\tau_n}}(t',x'; t,x) \,dx'dt'\right)\,dxdt \\
& \qquad \geq \int_Q m(t,x)h^-(t,x)\,dxdt,
  \end{aligned}
  \end{equation}
  by Fatou's lemma. Subtracting~\eqref{hminus-limit} from~\eqref{hplus-limit}, we obtain~\eqref{mh-limsup}.

\medskip
\noindent\textbf{Proof of \eqref{mhc-lim}:}
Consider the measure $\mathbf{m} := m\,dxdt$ over $Q$; note that $\mathbf{m}(Q) < \infty$ because $m\in L^\beta(Q)$ and $|Q|<\infty$. Let $L^q(Q,m)$ denote the corresponding $L^q$-space. Next, let
\begin{equation*}
 I_n := \int_t^\infty \int_{\Tt^d} h^{\text{c}}(t',x';t,x) K_{{\xi_n}, {\tau_n}}(t',x'; t,x) \,dx'dt'.
 \end{equation*}
It is enough to prove that
\[I_n \to H(t,x,D\upsilon,\mu)-H(t,x,D\upsilon+\ell,\mu) \qquad\text{in } L^1(Q, m).\]

We first claim that
\begin{equation}\label{I.n-converges-ae}
 I_n \to H(t,x,D\upsilon,\mu)-H(t,x,D\upsilon+\ell,\mu) \qquad \text{a.e.~in } Q.
 \end{equation}
Indeed, define
\begin{equation*}
    \theta = \begin{cases}
        D\upsilon - Du_T, \qquad & \text{on } Q, \\
        0, & \text{on } (T,\infty)\times\Tt^d,
    \end{cases}
\end{equation*}
and $F(t,x,s) = H(t,x,s+Du_T(x), \mu(t,x))$. 
Taking into account that 
\[
h^{\mathrm c}(t',x';t,x)
=
\begin{cases}
F(t,x,\theta(t',x'))
-
H\bigl(t',x',D\upsilon(t',x')+\ell(t',x'),\mu(t',x')\bigr),
& t'<T,\\[1mm]
F(t,x,\theta(t',x')), & t'>T,
\end{cases}
\]
we have
\begin{equation*}
    \begin{aligned}
        & I_n = -\int_t^T \int_{\Tt^d} H(t',x',D\upsilon(t',x') + \ell(t',x'), \mu(t',x')) K_{{\xi_n}, {\tau_n}}(t',x'; t,x) \,dx'dt' \\
        & \qquad\qquad + \int_t^\infty\int_{\Tt^d} F(t,x, \theta(t',x')) K_{{\xi_n}, {\tau_n}}(t',x'; t,x) \,dx'dt'.
    \end{aligned}
\end{equation*}
Now, the first integral is the mollification of the function
\[(t,x)\mapsto -H(t,x,D\upsilon(t,x)+\ell(t,x),\mu(t,x)),\]
extended by $0$ for $t>T$, which belongs to $L^1(Q)$ by~\eqref{eq:assH.lower.simple}, \eqref{res:H.upper.simple}, \eqref{Dupsilon.Lalpha}, and $|\ell|, \mu\in L^\infty(Q)$. 
Hence, by the same adapted Lebesgue differentiation argument applied to this \(L^1\)-function, it converges to $-H(t,x,D\upsilon(t,x)+\ell(t,x),\mu(t,x))$ a.e.~in $Q$.
On the other hand, the second integral converges to $F(t,x,\theta(t,x)) = H(t,x,D\upsilon(t,x),\mu(t,x))$ a.e.~in $Q$ by Lemma~\ref{lem:mollification-convergence-with-adapted.Lebesgue} applied with $\sigma = \alpha$, because $\theta$ and $F$ satisfy the hypotheses by \eqref{Dupsilon.Lalpha}, \eqref{eq:assH.lower.simple}, and \eqref{res:H.upper.simple}, while $(\xi_n, \tau_n)$ satisfies the hypothesis by construction.

On the other hand, we observe that
\begin{equation}\label{hc-leftover-converges-L1m}
 \tilde{I}_n := \int_T^\infty \int_{\Tt^d} h^{\text{c}}(t',x';t,x) K_{{\xi_n}, {\tau_n}}(t',x'; t,x) \,dx'dt' \to 0 \qquad\text{in } L^1(Q,m).
 \end{equation}
Indeed, since $K_{{\xi_n}, {\tau_n}}(t',x';t,x) = 0$ for $t' > t+\tau_n$, the last integral is nonzero only when \(t\in(T-\tau_n,T)\). Thus
\[|\tilde{I}_n| \leq \chi_{\{T-\tau_n<t<T\}}|H(t,x,Du_T(x),\mu(t,x))| \qquad \text{a.e.~in } Q.\]
Hence
\[\int_Q m|\tilde{I}_n| \,dxdt \leq \int_{T-\tau_n}^T\int_{\Tt^d} m|H(t,x,Du_T(x),\mu(t,x))| \,dxdt,\]
which converges to zero because
\[
m|H(t,x,Du_T(x),\mu(t,x))|\leq Cm(1+V)\in L^1(Q).
\]

In view of \eqref{hc-leftover-converges-L1m}, it remains to prove $I_n - \tilde{I}_n$ converges to $H(t,x,D\upsilon,\mu)-H(t,x,D\upsilon+\ell,\mu)$ in $L^1(Q,m)$. 
But, by \eqref{I.n-converges-ae} and the fact that $\tilde I_n\to0$ a.e.~in $Q$, we know that $I_n - \tilde{I}_n$ converges to $H(t,x,D\upsilon,\mu)-H(t,x,D\upsilon+\ell,\mu)$ a.e.~in $Q$, hence $\mathbf{m}$-a.e.~in $Q$.
To lift this $\mathbf{m}$-a.e.~convergence to convergence in $L^1(Q,m)$, we show that the sequence is bounded in $L^q(Q,m)$, that is
\begin{equation}\label{mollification-Lqm-bound-pre.calculation}
    \int_Q m(t,x)|I_n - \tilde{I}_n|^q\,dxdt = \int_Q m(t,x) \left|\int_t^T \int_{\Tt^d} h^{\text{c}}(t',x';t,x) K_{{\xi_n}, {\tau_n}}(t',x'; t,x) \,dx'dt'\right|^q \,dxdt
\end{equation}
is bounded uniformly in $n$ for some $q>1$. Indeed, by H\"older's inequality, this bound gives that
$\int_A|I_n - \tilde{I}_n|\,d\mathbf{m}\leq C\,\mathbf{m}(A)^{1/q'}$ for every measurable $A\subset Q$, uniformly in $n$; hence $\{I_n - \tilde{I}_n\}$ is uniformly integrable. As $\mathbf{m}$ is finite, the Vitali convergence theorem upgrades the $\mathbf{m}$-a.e.\ convergence to convergence in $L^1(Q,m)$. Consequently, $I_n$ converges in $L^1(Q,m)$ and~\eqref{mhc-lim} follows.

Before specifying $q$, we estimate the quantity in \eqref{mollification-Lqm-bound-pre.calculation} as follows:
\begin{equation}\label{mollification-Lqm-bound-calculation}
 \begin{aligned}
 & \int_Q m(t,x) \left|\int_t^T \int_{\Tt^d} h^{\text{c}}(t',x';t,x) K_{{\xi_n}, {\tau_n}}(t',x'; t,x) \,dx'dt'\right|^q \,dxdt \\
 & \ \leq \int_Q m(t,x) \left(\int_t^T \int_{\Tt^d} |h^{\text{c}}(t',x';t,x)|^q \cdot K_{{\xi_n}, {\tau_n}}(t',x'; t,x) \,dx'dt'\right)\,dxdt \\
 & \ = \int_Q \frac{1}{\tau_n} \hat{\zeta}\left(\frac{s}{\tau_n}\right)\cdot \frac{1}{(\xi_n)^d} \overline{\zeta}\left(\frac{y}{\xi_n}\right)\left(\int_0^{T-s} \int_{\Tt^d} m(t,x)|h^{\text{c}}(t+s,x+y;t,x)|^q \,dxdt\right) \,dyds \\
 & \ \leq \sup_{(s,y)\in\bar{Q}} \int_0^{T-s}\int_{\Tt^d} m(t,x)\biggl|H(t,x,D\upsilon(t+s,x+y)+Du_T(x)-Du_T(x+y),\mu(t,x))\\
 & \quad - H(t+s,x+y,D\upsilon(t+s,x+y)+\ell(t+s,x+y),\mu(t+s,x+y))\biggr|^q \,dxdt.
 \end{aligned}
\end{equation}

Since the last expression does not depend on $n$, it remains to find some $q>1$ which makes it finite. For this, we use 
Assumption \ref{onH.highest-order-p.has-no-txm} via 
\eqref{eq:devH-quadruple} in Lemma \ref{lem:deviation-bound} and choose some $c > 1$ satisfying
\begin{equation*}
    \begin{aligned}
        & |H(t,x,p_1,m_1) - H(t',x',p_2,m_2)| \\
        & \qquad \leq c \bigl( R(m_1+m_2)+|p_1-p_2|\bigr)\bigl(|p_1|^\alpha + |p_2|^\alpha + V(t,x) +V(t',x')\bigr)^{1-1/c}.
    \end{aligned}
    \end{equation*}
    Then, we set $q:= c/(c-1)$. Accordingly,
    \begin{equation}\label{integrand-estimated-devH-bound}
     \begin{aligned}
       & \biggl|H(t,x,D\upsilon(t+s,x+y)+Du_T(x)-Du_T(x+y),\mu(t,x))\\
       & \qquad - H(t+s,x+y,D\upsilon(t+s,x+y)+\ell(t+s,x+y),\mu(t+s,x+y))\biggr|^q \\
       & \leq C\bigl(R(\mu(t,x)+\mu(t+s,x+y))+|Du_T(x)|+|Du_T(x+y)|+|\ell(t+s,x+y)|\bigr)^q\\
       &\qquad \cdot\Bigl(|D\upsilon(t+s,x+y)|^\alpha+V(t,x)+V(t+s,x+y)\\
       &\qquad\qquad +|Du_T(x)|^\alpha+|Du_T(x+y)|^\alpha+|\ell(t+s,x+y)|^\alpha\Bigr) \\
       & \leq C(|D\upsilon(t+s,x+y)|^\alpha+V(t,x)+V(t+s,x+y)),
     \end{aligned}
  \end{equation}
  where, in the last step, we absorbed the finite quantities $\Vert \mu\Vert_{L^\infty(Q)}$ and $\Vert \ell\Vert_{L^\infty(Q)}$, as well as the data, into the constant $C$. 
  Then we insert the bound~\eqref{integrand-estimated-devH-bound} into the last expression in~\eqref{mollification-Lqm-bound-calculation} and see that it is finite because of~\eqref{mDupsilon-shift-L1}, $m\in L^\beta(Q)$, and $V\in L^{\beta'}(Q)$. Thus, we finish the proof. 
\end{proof}

\begin{proof}[Proof of Lemma~\ref{lem:upsilon-improved}]
  Let $\upsilon_n\in C^1(\bar{Q})$ be the sequence defined in~\eqref{upsilon_n.defined}. Because $\upsilon_n(T,x) = u_T(x)$, \eqref{weak-sol-mu-improved-Linfty} holds for $\upsilon_n$, yielding:
  \begin{equation}\label{weak-sol.upsilon_n}
  \begin{aligned}
  0 \leq & \int_Q(\mu-m)\bigl((\upsilon_n)_t - H(t,x,D\upsilon_n,\mu)\bigr)\,dxdt\\
  & + \int_Q\bigl(\mu (u_t^{\text{ac}}-(\upsilon_n)_t) + \mu D_pH(t,x,D\upsilon_n,\mu)\cdot (D\upsilon_n-Du) \bigr)\,dxdt\\
  & \qquad -\int_{\Tt^d}m_0(x)(\upsilon_n(0,x)-u(0,x))\,dx \\
  = & \int_Q m\bigl(-(\upsilon_n)_t + H(t,x,D\upsilon_n,\mu)\bigr)\,dxdt\\
  &+\int_Q\mu\bigl( u_t^{\text{ac}}- H(t,x,D\upsilon_n,\mu) + D_pH(t,x,D\upsilon_n,\mu)\cdot (D\upsilon_n-Du) \bigr)\,dxdt\\
  &\qquad -\int_{\Tt^d}m_0(x)(\upsilon_n(0,x)-u(0,x))\,dx,
  \end{aligned}
  \end{equation}
  for all $\mu\in L^{\infty}(Q; \RR^+_0)$. We now pass to the limit $n\to\infty$ in~\eqref{weak-sol.upsilon_n}.

  First, by~\eqref{Dupsilon.limit-aux}, we can extract a subsequence such that $D\upsilon_n \to D\upsilon$ a.e.~and $|D\upsilon_n|\leq g \in L^\alpha(Q)$. The bounds \eqref{eq:assH.DpH.upper.simple} and \eqref{res:H.upper.simple} provide integrable majorants. Thus, the Dominated Convergence Theorem yields
  \begin{equation*}
  \begin{aligned}
  & \lim_{n\to\infty} \int_Q\mu\bigl( u_t^{\text{ac}}- H(t,x,D\upsilon_n,\mu) + D_pH(t,x,D\upsilon_n,\mu)\cdot (D\upsilon_n-Du) \bigr)\,dxdt \\
  & \qquad = \int_Q\mu\bigl( u_t^{\text{ac}}- H(t,x,D\upsilon,\mu) + D_pH(t,x,D\upsilon,\mu)\cdot (D\upsilon-Du) \bigr)\,dxdt \\
  & \qquad = \int_Q\mu\bigl(\upsilon_t^{\text{ac}} - H(t,x,D\upsilon,\mu)\bigr)\,dxdt \\
  & \qquad\quad + \int_Q\bigl(\mu (u_t^{\text{ac}}-\upsilon_t^{\text{ac}}) + \mu D_pH(t,x,D\upsilon,\mu)\cdot (D\upsilon-Du) \bigr)\,dxdt.
  \end{aligned}
  \end{equation*}

  Second, by~\eqref{upsilon-trace-limit-aux}, we have
  \begin{equation*}
  \lim_{n\to\infty} \int_{\Tt^d}m_0(x)(\upsilon_n(0,x)-u(0,x))\,dx = \int_{\Tt^d}m_0(x)(\upsilon(0,x)-u(0,x))\,dx.
  \end{equation*}

  Finally, \eqref{m-times-hjb.upsilon-limit} of Proposition~\ref{lem:hjb-limsup} provides the upper bound for the remaining term weighted by $m$:
  \begin{equation*}
  \limsup_{n\to\infty} \int_Q m\bigl(-(\upsilon_n)_t + H(t,x,D\upsilon_n,\mu)\bigr)\,dxdt \leq \int_Q m\bigl(-\upsilon_t^{\text{ac}} + H(t,x,D\upsilon,\mu)\bigr)\,dxdt.
  \end{equation*}

  Taking the limsup of both sides of~\eqref{weak-sol.upsilon_n} and substituting these three limits directly produces~\eqref{weak-sol-upsilon-improved}, completing the proof.
\end{proof}

\subsection{Further Enlarging Test Densities}

In the final step of the proof of Theorem~\ref{weak.is.strong}, we build upon the $BV(Q)$ test value framework established in the previous subsection to complete the decoupling of the system. Because the actual population density $m$ lies in $L^\beta(Q)$ rather than $L^\infty(Q)$, we must further enlarge the class of test densities to $L^\beta(Q)$. By employing a truncation argument and the reverse Fatou lemma, we extend the variational inequality to these unbounded densities. This allows us to use carefully tailored variations around the true solution components to rigorously extract the Hamilton--Jacobi equation almost everywhere, and subsequently recover the transport equation via convex variations of the test value functions.

\begin{lemma}\label{lem:vi-gives-hjb}
    Consider the setting of Problem~\ref{prob.mfg}. Suppose Assumptions \ref{onH.monotone}, \ref{onH.powergrowth} and \ref{onH.highest-order-p.has-no-txm} hold. Let $(m,u)$ be a solution to Problem~\ref{prob.vi} in the sense of Definition~\ref{def.weak.sol}, which satisfies~\eqref{integrability.further.m}, \eqref{integrability.further.du}, \eqref{integrability.further.ut}, \eqref{integrability.further.mdu}, and~\eqref{integrability.further.H}. Then $(m,u)$ satisfies \eqref{def:hjb-strong}.
\end{lemma}

\begin{proof}
    We first observe that $\upsilon := u$ is a valid test function in~\eqref{weak-sol-upsilon-improved}. Indeed, \eqref{integrability.further.du} implies~\eqref{Dupsilon.Lalpha}, \eqref{u.jumps.forward} implies~\eqref{upsilon.jumps.forward}, \eqref{integrability.further.H} implies~\eqref{upsilon-HJB-upper} with $\ell = 0$, and \eqref{integrability.further.mdu} implies~\eqref{mDupsilon-shift-L1}. Therefore, 
    \begin{equation}\label{weak-sol-upsilon-equals-u}
    0 \leq \int_Q(\mu-m)(u_t^{\text{ac}} - H(t,x,Du,\mu))\,dxdt
    \end{equation}
for all $\mu\in L^{\infty}(Q; \RR^+_0)$.
The first step is to establish that 
\eqref{weak-sol-upsilon-equals-u} holds for all $\mu\in L^{\beta}(Q; \RR^+_0)$ satisfying $\esssup(\mu-m)<\infty$. 
Then, we use this result with a suitable $\mu$ to conclude~\eqref{def:hjb-strong}.

\medskip

Consider a truncation $\mu_n := \min(\mu,n)$ of $\mu$. Then, \eqref{weak-sol-upsilon-equals-u} implies
\begin{equation*}
    0 \leq \int_Q(\mu_n-m)(u_t^{\text{ac}} - H(t,x,Du,\mu_n))\,dxdt.
    \end{equation*}
    We wish to pass to the limit in this expression. To this end, we observe
\[\begin{aligned}
& (\mu_n-m)\bigl(u_t^{\text{ac}} - H(t,x,Du,\mu_n)\bigr)\\
& \quad \leq  (\mu_n-m)^+\bigl(u_t^{\text{ac}} - H(t,x,Du,\mu_n)\bigr)^+ + (\mu_n-m)^-\bigl(u_t^{\text{ac}} - H(t,x,Du,\mu_n)\bigr)^- \\
 & \quad \leq (\mu-m)^+((u_t^{\text{ac}})^+ + C\mu^{\beta-1}+CV(t,x)) + m((u_t)^- +C|Du|^\alpha +CV(t,x)) \quad\text{a.e.~in } Q.
\end{aligned}
\]
The last function is in $L^1(Q)$ because $(\mu-m)^+\in L^\infty(Q)$ by hypothesis, $u\in BV(Q)$, $\mu\in L^\beta(Q)$, $m\in L^\beta(Q)$, $V\in L^{\beta'}(Q)$, $(u_t)^-\in L^{\beta'}(Q)$ by~\eqref{integrability.further.ut}, and $m|Du|^\alpha\in L^1(Q)$ by~\eqref{integrability.further.mdu}. Therefore, the reverse Fatou lemma applies to give
   \begin{equation*}
      \begin{aligned}
        & \limsup_{n\to\infty} \int_Q(\mu_n-m)(u_t^{\text{ac}} - H(t,x,Du,\mu_n))\,dxdt \\
        & \qquad\qquad \leq \int_Q(\mu-m)(u_t^{\text{ac}} - H(t,x,Du,\mu))\,dxdt.
      \end{aligned}
   \end{equation*}
 Since the left-hand side integrals are nonnegative for every $n$, this proves that \eqref{weak-sol-upsilon-equals-u} holds for all $\mu\in L^\beta(Q;\RR^+_0)$ satisfying $\esssup(\mu-m)<\infty$.

  \medskip
Now we address the Hamilton--Jacobi equation~\eqref{def:hjb-strong}. Let $\mu$ be defined as
  \begin{equation*}
     \begin{aligned}
      \mu(t,x) = \min\biggl(m(t,x)+1,\, \inf\Bigl\{\mu_0\in\RR^+_0\colon & u_t^{\text{ac}}(t,x)-H(t,x,Du(t,x),\mu_0)\geq \\
     & \qquad \frac{u_t^{\text{ac}}(t,x)-H(t,x,Du(t,x),m(t,x))}{2} \Bigr\}\biggr),  
     \end{aligned}
  \end{equation*}
  for a.e.~$(t,x)\in Q$. 
  The set in the infimum is nonempty by~\eqref{res:H.upper.simple}, and $\mu$ is Lebesgue measurable. Moreover, $0\leq\mu\leq m+1$, so $\mu\in L^\beta(Q;\RR^+_0)$ and $\esssup(\mu-m)\leq1$; hence \eqref{weak-sol-upsilon-equals-u} holds.
 Note that $H(t,x,p,\cdot)$ is continuous and non-increasing on $\RR^+_0$. Hence, for a.e.~$(t,x)\in Q$, one of the following alternatives must hold:
  \begin{enumerate}[label=(\roman*)]
     \item $\mu(t,x) = 0$ and
     \begin{equation}\label{mu-at-zero}
      u_t^{\text{ac}}(t,x)-H(t,x,Du(t,x),\mu(t,x)) \geq \frac{u_t^{\text{ac}}(t,x)-H(t,x,Du(t,x),m(t,x))}{2}, 
     \end{equation}
     \item $0<\mu(t,x)<m(t,x)+1$ and
     \begin{equation}\label{mu-midrange}
      u_t^{\text{ac}}(t,x)-H(t,x,Du(t,x),\mu(t,x)) = \frac{u_t^{\text{ac}}(t,x)-H(t,x,Du(t,x),m(t,x))}{2},
     \end{equation}
     \item $\mu(t,x) = m(t,x)+1$ and
     \begin{equation*}
      u_t^{\text{ac}}(t,x)-H(t,x,Du(t,x),\mu(t,x)) \leq \frac{u_t^{\text{ac}}(t,x)-H(t,x,Du(t,x),m(t,x))}{2}. 
     \end{equation*}
  \end{enumerate}
   Therefore,
   \begin{enumerate}[label=(\alph*)]
      \item $\mu(t,x) > m(t,x)$ implies the first alternative cannot hold. Hence,
      \begin{equation}\label{mu-above}
       \begin{aligned}
       & u_t^{\text{ac}}(t,x)-H(t,x,Du(t,x),m(t,x)) \leq \\
       & \qquad u_t^{\text{ac}}(t,x)-H(t,x,Du(t,x),\mu(t,x)) \leq \\
       &\qquad\qquad \frac{u_t^{\text{ac}}(t,x)-H(t,x,Du(t,x),m(t,x))}{2}. 
      \end{aligned} 
      \end{equation}
      \item $\mu(t,x) \leq m(t,x)$ implies the last alternative cannot hold. Hence, 
      \begin{equation}\label{mu-under}
       \begin{aligned}
       & u_t^{\text{ac}}(t,x)-H(t,x,Du(t,x),m(t,x)) \geq \\
       & \qquad u_t^{\text{ac}}(t,x)-H(t,x,Du(t,x),\mu(t,x)) \geq \\
       &\qquad\qquad \frac{u_t^{\text{ac}}(t,x)-H(t,x,Du(t,x),m(t,x))}{2}. 
      \end{aligned} 
      \end{equation}
   \end{enumerate}
  In either case, we have
      \[(\mu-m)\bigl(u_t^{\text{ac}} -H(t,x,Du,\mu)\bigr)\leq 0.\]
      Comparing this with~\eqref{weak-sol-upsilon-equals-u}, we obtain
\begin{equation}\label{hjb-at-test}
(\mu-m)\bigl(u_t^{\text{ac}} -H(t,x,Du,\mu)\bigr) = 0 \qquad \text{a.e.~in } Q.
\end{equation}
  Now,
  \begin{enumerate}[label=(\arabic*)]
     \item If $u_t^{\text{ac}}(t,x) -H(t,x,Du(t,x),\mu(t,x)) = 0$, either \eqref{mu-above} or \eqref{mu-under} holds and yields
     \[u_t^{\text{ac}}(t,x) -H(t,x,Du(t,x),m(t,x)) = 0.\]
     \item If $u_t^{\text{ac}}(t,x) -H(t,x,Du(t,x),\mu(t,x)) \neq 0$, \eqref{hjb-at-test} gives $\mu(t,x) = m(t,x)$. Then,~\eqref{mu-midrange} does not hold, thus~\eqref{mu-at-zero} holds and yields
     \[m(t,x) = 0 \qquad \text{and}\qquad u_t^{\text{ac}}(t,x)-H(t,x,Du(t,x),m(t,x))\geq 0.\]
  \end{enumerate}
  The conclusion of these two cases, together with $u_t^{\text{s}}\geq 0$ from~\eqref{u.jumps.forward}, is~\eqref{def:hjb-strong}.
\end{proof}

\begin{lemma}\label{full-test-functions}
  Consider the setting of Problem~\ref{prob.mfg}. Suppose Assumptions \ref{onH.monotone}, \ref{onH.powergrowth} and \ref{onH.highest-order-p.has-no-txm} hold. Let $(m,u)$ be a solution to Problem~\ref{prob.vi} in the sense of Definition~\ref{def.weak.sol}, which satisfies~\eqref{integrability.further.m}, \eqref{integrability.further.du}, \eqref{integrability.further.ut}, \eqref{integrability.further.mdu}, and~\eqref{integrability.further.H}. Then, \eqref{weak-sol-upsilon-improved} holds
  for all $(\mu,\upsilon)\in L^\beta(Q; \RR^+_0)\times BV(Q)$ satisfying $\esssup (\mu-m) < \infty$ and satisfying \eqref{Dupsilon.Lalpha}, \eqref{upsilon.jumps.forward}, \eqref{upsilon-HJB-upper}, and \eqref{mDupsilon-shift-L1}.
\end{lemma}

\begin{proof}
By Remark~\ref{rmk:m-upsilon-t-integrable}, the term
\(\int_Q m\upsilon_t^{\mathrm{ac}}\,dxdt\) is finite.
Considering a truncation $\mu_n := \min(\mu,n)$ of $\mu$ as in the proof of Lemma~\ref{lem:vi-gives-hjb}, we see that \eqref{weak-sol-upsilon-improved} gives
\begin{equation}\label{weak-sol-desired-test-functions-truncated}
\begin{aligned}
0 \leq & -\int_Qm\upsilon_t^{\text{ac}}\,dxdt + \int_Q mH(t,x,D\upsilon,\mu_n)\,dxdt\\
&+\int_Q\mu_n\bigl(u_t^{\text{ac}} + D_pH(t,x,D\upsilon,\mu_n)\cdot (D\upsilon-Du) - H(t,x,D\upsilon,\mu_n)\bigr)\,dxdt\\
&\qquad -\int_{\Tt^d}m_0(x)(\upsilon(0,x)-u(0,x))\,dx.
\end{aligned}
\end{equation}
The desired result follows once we justify passing to the limit $n\to\infty$ in \eqref{weak-sol-desired-test-functions-truncated} by the Dominated Convergence Theorem. To this end, we note that the first and the last integrals are already independent of $n$, and we carefully bound the second and the third integrands as follows. 

For the first integrand, \eqref{eq:assH.lower.simple} and \eqref{res:H.upper.simple} give
\[m|H(t,x,D\upsilon, \mu_n)|\leq Cm(|D\upsilon|^\alpha+\mu^{\beta-1}+V(t,x)),\]
which is integrable since $\mu\in L^\beta(Q)$ and $\upsilon$ satisfies \eqref{mDupsilon-shift-L1}. For the second integrand, \eqref{eq:assH.DpH.upper.simple}, \eqref{eq:assH.lower.simple}, \eqref{res:H.upper.simple}, and Young's inequality give
\[\begin{aligned}
&\mu_n\bigl(|u_t^{\text{ac}}| + |D_pH(t,x,D\upsilon,\mu_n)|(|D\upsilon|+|Du|) + |H(t,x,D\upsilon,\mu_n)|\bigr) \\
    & \qquad \leq \mu\bigl(|u_t^{\text{ac}}| + C|D\upsilon|^\alpha+C|Du|^\alpha+C\mu^{\beta-1}+CV(t,x)\bigr) \\
    & \qquad \leq C\mu(\mu^{\beta-1}+V(t,x))  + (\mu-m)^+\bigl(|u_t^{\text{ac}}| + C|D\upsilon|^\alpha+C|Du|^\alpha\bigr) \\
    & \qquad\qquad + m|u_t^{\text{ac}}| + Cm(|D\upsilon|^\alpha+|Du|^\alpha).
\end{aligned}\]
The first two terms on the right-hand side are integrable because $\mu\in L^\beta(Q)$, $V\in L^{\beta'}(Q)$, $(\mu-m)^+\in L^\infty(Q)$, $u\in BV(Q)$, and $Du,D\upsilon\in L^\alpha(Q;\RR^d)$, 
due to \eqref{integrability.further.du} and \eqref{Dupsilon.Lalpha}.
 The second term of the last line is also integrable due to \eqref{integrability.further.mdu} and \eqref{mDupsilon-shift-L1}. Thus, it only remains to show that the first term of the last line is integrable. Indeed, since the Hamilton--Jacobi equation~\eqref{def:hjb-strong} holds due to Lemma~\ref{lem:vi-gives-hjb} and $(u_t^{\text{ac}})^- = (u_t)^-$ as $u_t^{\text{s}} \geq 0$, we have
 \begin{equation}\label{mutac-L1}
 \begin{aligned}
     m|u_t^{\text{ac}}| = m(u_t^{\text{ac}} + 2(u_t^{\text{ac}})^-) & = m(H(t,x,Du,m)+2(u_t)^-) \\ & \qquad\leq Cm(|Du|^\alpha+m^{\beta-1}+V(t,x))+2m(u_t)^-.
 \end{aligned}
\end{equation}
The last expression is integrable due to~\eqref{integrability.further.m},~\eqref{integrability.further.ut}, and~\eqref{integrability.further.mdu}.
\end{proof}

\begin{lemma}\label{lem:transport-tested-upsilon.minus.u}
    Consider the setting of Problem~\ref{prob.mfg}. Suppose Assumptions \ref{onH.monotone}, \ref{onH.powergrowth} and \ref{onH.highest-order-p.has-no-txm} hold. Let $(m,u)$ be a solution to Problem~\ref{prob.vi} in the sense of Definition~\ref{def.weak.sol}, which satisfies~\eqref{integrability.further.m}, \eqref{integrability.further.du}, \eqref{integrability.further.ut}, \eqref{integrability.further.mdu}, and~\eqref{integrability.further.H}. Then
    \begin{equation}\label{transport-tested-upsilon.minus.u}
    \begin{aligned}
    & \int_Qm\bigl( -(\upsilon-u)_t^{\text{ac}} + D_pH(t,x,Du,m)\cdot D(\upsilon-u) \bigr)\,dxdt\\
    &\qquad\quad \geq \int_{\Tt^d}m_0(x)(\upsilon(0,x)-u(0,x))\,dx
    \end{aligned}
    \end{equation}
for all $\upsilon\in BV(Q)$ satisfying \eqref{Dupsilon.Lalpha}, \eqref{upsilon.jumps.forward}, \eqref{upsilon-HJB-upper},  and \eqref{mDupsilon-shift-L1}.
\end{lemma}

\begin{proof}
Due to Lemma~\ref{full-test-functions}, we use \eqref{weak-sol-upsilon-improved} with $\mu = m$ to get
\begin{equation}\label{transport-tested-pre.conv.comb}
    \begin{aligned}
    0 \leq & \int_Qm\bigl( -(\upsilon-u)_t^{\text{ac}} + D_pH(t,x,D\upsilon,m)\cdot D(\upsilon-u) \bigr)\,dxdt\\
    &\qquad -\int_{\Tt^d}m_0(x)(\upsilon(0,x)-u(0,x))\,dx.
    \end{aligned}
    \end{equation}
Here, we note that $\upsilon$ in \eqref{transport-tested-pre.conv.comb} can be replaced by $(1-\lambda)u + \lambda\upsilon$ for $0\leq\lambda\leq 1$. Indeed, \eqref{Dupsilon.Lalpha}, \eqref{upsilon.jumps.forward}, and \eqref{mDupsilon-shift-L1} are preserved under convex combinations. Moreover, if $\ell\in L^\infty(Q; \RR^d)$ is as in \eqref{upsilon-HJB-upper} for $\upsilon$, then $\lambda\ell$ is admissible for $(1-\lambda)u+\lambda\upsilon$, because
\[\begin{aligned}
& -\bigl((1-\lambda)u+\lambda\upsilon\bigr)_t^{\text{ac}}
+H\bigl(t,x,(1-\lambda)Du+\lambda(D\upsilon+\ell),0\bigr) \\
& \qquad \leq (1-\lambda)\bigl(-u_t^{\text{ac}}+H(t,x,Du,0)\bigr)
+\lambda\bigl(-\upsilon_t^{\text{ac}}+H(t,x,D\upsilon+\ell,0)\bigr) \qquad \text{a.e.~in } Q,
\end{aligned}
\]
by the convexity of $H(t,x,\cdot,0)$ as in Remark~\ref{rmk:cxty}. Consequently, we have
    \begin{equation*}
    \begin{aligned}
    0 & \leq \lambda\biggl(\int_Qm\bigl( -(\upsilon-u)_t^{\text{ac}} + D_pH(t,x,(1-\lambda)Du+\lambda D\upsilon,m)\cdot D(\upsilon-u) \bigr)\,dxdt\\
    & \qquad\qquad -\int_{\Tt^d}m_0(x)(\upsilon(0,x)-u(0,x))\,dx\biggr),
    \end{aligned}
    \end{equation*}
        for $0\leq\lambda\leq 1$. 
        Dividing by \(\lambda\) for \(\lambda>0\), we pass to the limit
\(\lambda\to0^+\). The pointwise convergence follows from the continuity of
\(p\mapsto D_pH(t,x,p,m(t,x))\) for a.e.~\((t,x)\in Q\). Moreover, domination
follows from \eqref{eq:assH.DpH.upper.simple}, combined with
\eqref{integrability.further.m}, \eqref{eq:assH.potential}, and Young's
inequality, using \eqref{integrability.further.mdu} and
\eqref{mDupsilon-shift-L1} at \((s,y)=(0,0)\). Hence the Dominated Convergence
Theorem gives \eqref{transport-tested-upsilon.minus.u}.
\end{proof}

\begin{proof}[Proof of Theorem~\ref{weak.is.strong}]
We combine the preceding results to check the requirements of Definition~\ref{def.strong.sol}. 
We begin by verifying the integrability requirements in \eqref{basic-integrability}.
The condition \(Du\in L^1(Q;\RR^d)\) follows from
\eqref{integrability.further.du}, since \(\alpha>1\) and \(Q\) has finite
measure. The growth bound \eqref{eq:assH.DpH.upper.simple}, together with
\eqref{integrability.further.m}, \eqref{integrability.further.mdu}, and
\eqref{eq:assH.potential}, implies
\[
mD_pH(t,x,Du,m)\in L^1(Q;\RR^d)
\]
and
\[
mD_pH(t,x,Du,m)\cdot Du\in L^1(Q).
\]
Finally, the estimate \(m|u_t^{\mathrm{ac}}|\in L^1(Q)\), given by \eqref{mutac-L1} in the
proof of Lemma~\ref{full-test-functions}, gives
\[
m\bigl(u_t^{\mathrm{ac}}-D_pH(t,x,Du,m)\cdot Du\bigr)\in L^1(Q).
\]
Thus \eqref{basic-integrability} holds.

Lemma~\ref{lem:vi-gives-hjb} shows that \eqref{def:hjb-strong} holds. Moreover,
\eqref{u.jumps.forward} in Lemma~\ref{lem:mu-improved} gives
\(u_t^{\mathrm{s}}\ge0\) and \(u(T,\cdot)\leq u_T\).

To obtain \eqref{def:transport-strong}, fix
\(\varphi\in C^1(\bar Q)\) with \(\varphi(T,\cdot)=0\). We claim that
\(\upsilon:=u\pm\varphi\) is an admissible test function in
\eqref{transport-tested-upsilon.minus.u}. Indeed,
\eqref{integrability.further.du} implies \eqref{Dupsilon.Lalpha}, while
\(u_t^{\mathrm{s}}\ge0\), \(u(T,\cdot)\leq u_T\), and \(\varphi(T,\cdot)=0\) imply
\eqref{upsilon.jumps.forward}. Moreover, \eqref{integrability.further.H}
implies \eqref{upsilon-HJB-upper} with \(\ell:=\mp D\varphi\), because
\[
-\upsilon_t^{\mathrm{ac}}+H(t,x,D\upsilon+\ell,0)
=
-u_t^{\mathrm{ac}}\mp\varphi_t+H(t,x,Du,0),
\]
and \(\varphi_t\in L^\infty(Q)\). Finally, \eqref{integrability.further.mdu}
implies \eqref{mDupsilon-shift-L1}, since \(D\varphi\) is bounded.

Applying \eqref{transport-tested-upsilon.minus.u} with
\(\upsilon=u+\varphi\) gives
\[
\int_Q m\bigl(-\varphi_t+D_pH(t,x,Du,m)\cdot D\varphi\bigr)\,dxdt
\ge
\int_{\Tt^d}m_0(x)\varphi(0,x)\,dx.
\]
Applying it with \(\upsilon=u-\varphi\) gives the reverse inequality. Hence
\eqref{def:transport-strong} follows.

Finally, choose \(\upsilon(t,x)=u_T(x)\). This function is admissible in
\eqref{transport-tested-upsilon.minus.u}: indeed, \(D\upsilon=Du_T\in
L^\infty(Q;\RR^d)\), \(\upsilon_t=0\), \(\upsilon(T,\cdot)=u_T\), and \eqref{res:H.upper.simple} implies 
\[
\bigl(H(t,x,Du_T,0)\bigr)^+\in L^{\beta'}(Q),
\]
while \eqref{mDupsilon-shift-L1} is immediate from \(m\in L^\beta(Q)\) and
\(Du_T\in L^\infty(Q;\RR^d)\). Therefore \eqref{transport-tested-upsilon.minus.u}
with this choice of \(\upsilon\) yields
\[
\int_Q m\bigl(u_t^{\mathrm{ac}}
+D_pH(t,x,Du,m)\cdot(Du_T-Du)\bigr)\,dxdt
\ge
\int_{\Tt^d}m_0(x)(u_T(x)-u(0,x))\,dx,
\]
which is precisely \eqref{def:singular-and-terminal}.
\end{proof}

\section{Proof of Theorem~\ref{maximal-u-exists}}
\label{sec:maximal-u}

In this section, we recall the notations $\mathcal{U}(m)$ and $\mathcal{S}(m)$ introduced in Definition~\ref{def:u-classes-at-given-m}. 

\subsection{Regularity and Stability Properties}

\begin{proposition}\label{prop:uinitial.and.m.bound:mfgsolution}
    Consider the setting of Problem~\ref{prob.mfg}. Suppose Assumptions~\ref{onH.monotone}, \ref{onH.powergrowth}, and \ref{onH.highest-order-p.has-no-txm} hold. Let $m\in L^\beta(Q;\RR^+_0)$. Then any $u\in \mathcal{U}(m)$ satisfies 
    \begin{align}
        \int_Q ((u_t)^-)^{\beta'}\,dxdt & \leq C, \label{ut.minus-Lbeta':mfgsolution} \\
        \int_{\Tt^d} u^+(0,x)^{\beta'}\,dx & \leq C, \label{uplus-initial:mfgsolution} \\
        \int_{\Tt^d} m_0(x)u^-(0,x)\,dx & \leq C, \label{uminus-initial-m0:mfgsolution}  \\
        \int_Q m|Du|^\alpha \,dxdt & \leq C, \label{mDualpha-basic:mfgsolution}
    \end{align}
    for a positive constant $C$ independent of $u$.
\end{proposition}

\begin{proof}
Combining the inequality condition in \eqref{def:hjb-strong} with the lower bound \eqref{eq:assH.lower.simple}, we get
\begin{equation}\label{hjb-estimated-wo-Du:mfgsolution}
    -u_t \leq  C(m^{\beta-1} + V(t,x))
\end{equation}
in the sense of measures. Taking the positive part of the left-hand side, while noting that the right-hand side is in $L^{\beta'}(Q;\RR^+_0)$ and independent of $u$, we obtain \eqref{ut.minus-Lbeta':mfgsolution}.

Next, we multiply \eqref{hjb-estimated-wo-Du:mfgsolution} by a test function $\varphi\in C^\infty(\Tt^d; \RR^+_0)$ and integrate over $Q$ to get
\begin{equation*}
    -\int_Q \varphi(x)\,du_t \leq \int_{\Tt^d} \varphi(x) \left(C\int_0^T (m^{\beta-1} + V(t,x))\,dt\right)\,dx.
\end{equation*}
Applying the trace Green formula to the left-hand side, we obtain
\begin{equation*}
     \int_{\Tt^d} \varphi(x) \bigl(u(0,x)-u(T,x)\bigr)\,dx \leq 
    \int_{\Tt^d} \varphi(x) \left(C\int_0^T (m^{\beta-1} + V(t,x))\,dt\right)\,dx.
\end{equation*}
Since this holds for all $\varphi\in C^\infty(\Tt^d; \RR^+_0)$ with the same value of $C$, we must have
\begin{equation*}
        u(0,x)-u(T,x) \leq C\int_0^T (m^{\beta-1} + V(t,x))\,dt \qquad \text{a.e. on }\Tt^d.
\end{equation*}
  Moving the second term of the left-hand side to the right and using the inequality $u(T,x)\leq u_T(x)$, we find
  \begin{equation*}
      u^+(0,x)\leq |u_T(x)| + C\int_0^T(m^{\beta-1}+V(t,x))\,dt \qquad\text{a.e. on } \Tt^d.
  \end{equation*}
  Now, the right-hand side is in $L^{\beta'}(\Tt^d)$ and independent of $u$, hence we obtain \eqref{uplus-initial:mfgsolution}.

    Finally, plugging the equality condition of the Hamilton--Jacobi equation~\eqref{def:hjb-strong} into~\eqref{def:singular-and-terminal}, then rearranging, we get
    \begin{equation*}
     \begin{aligned}
       &\int_Q m\bigl(-H(t,x,Du,m) + D_pH(t,x,Du,m)\cdot Du\bigr) \,dxdt \\
       & \enskip \leq \int_{\Tt^d} m_0(x)u(0,x)\,dx + \int_{\Tt^d} m_0(x)|u_T(x)|\,dx + \int_Q m|D_pH(t,x,Du,m)||Du_T| \,dxdt.
     \end{aligned}
  \end{equation*}
  Estimating the left-hand side with \eqref{dphdotp-minus-h.lower.bound} and the right-hand side with \eqref{eq:assH.DpH.upper.simple}, then absorbing the terms that do not depend on $u$ into the constant $C$, we find
  \begin{equation*}
      \frac{1}{C}\int_Q m|Du|^\alpha \,dxdt  \leq \int_{\Tt^d} m_0(x)u(0,x)\,dx  + C\int_Q m|Du|^{\alpha-1} \,dxdt + C,
  \end{equation*}
  which we rearrange to obtain
  \begin{equation*}
  \begin{aligned}
      & \int_{\Tt^d} m_0(x)u^-(0,x)\,dx + \frac{1}{C}\int_Q m|Du|^\alpha \,dxdt \\
      & \qquad\qquad \leq \int_{\Tt^d} m_0(x)u^+(0,x)\,dx  + C\int_Q m|Du|^{\alpha-1} \,dxdt + C.
      \end{aligned}
  \end{equation*}
  The first term on the right-hand side is bounded due to~\eqref{uplus-initial:mfgsolution} and the middle term can be absorbed into the left-hand side with Young's inequality, yielding \eqref{uminus-initial-m0:mfgsolution} and~\eqref{mDualpha-basic:mfgsolution}.
\end{proof}

\begin{proposition}\label{prop:u-derivatives-bound:mfgsolution}
    Consider the setting of Problem~\ref{prob.mfg}. Suppose Assumptions~\ref{onH.monotone}, \ref{onH.powergrowth}, and \ref{onH.highest-order-p.has-no-txm} hold. Let $m\in L^\beta(Q;\RR^+_0)$. Then any $u\in \mathcal{U}(m)$ satisfies
    \begin{align}
        \int_Q d(u_t)^+ & \leq C + \int_{\Tt^d} u^-(0,x)\,dx, \label{ut.plus-bound:mfgsolution} \\
        \int_Q |Du|^\alpha \,dxdt & \leq C + C\int_{\Tt^d} u^-(0,x)\,dx, \label{Du-Lalpha:mfgsolution} \\
        \sup_{s\in (0,T]} \int_{\Tt^d} u^-(s^-,x)\,dx & \leq C + \int_{\Tt^d} u^-(0,x)\,dx, \label{u-lower-uniform-in-time.in-terms-of-initial:mfgsolution}
    \end{align}
    for a positive constant $C$ independent of $u$. Here, $u(s^-,\cdot)$ refers to the trace of $u$ on the hypersurface $\{s\}\times\Tt^d$ from the side of $(0,s)\times \Tt^d$.
\end{proposition}

\begin{proof}
 First we observe
\[\begin{aligned}
    \int_Q d(u_t)^+ & = \int_Q du_t + \int_Q (u_t)^-\,dxdt \\
    & = \int_{\Tt^d} (u(T,x) - u(0,x))\,dx + \int_Q (u_t)^-\,dxdt \\
    & \leq \int_{\Tt^d} |u_T(x)|\,dx + \int_{\Tt^d} u^-(0,x)\,dx + \int_Q (u_t)^-\,dxdt,
\end{aligned}
\]
hence \eqref{ut.plus-bound:mfgsolution} follows from \eqref{ut.minus-Lbeta':mfgsolution}. On the other hand, the inequality condition of \eqref{def:hjb-strong} and the bound \eqref{eq:assH.lower.simple} imply
\begin{equation*}
    |Du|^\alpha \leq C(m^{\beta-1}+V(t,x) + (u_t)^+) 
\end{equation*}
in the sense of measures. We integrate the last expression over $Q$ and conclude \eqref{Du-Lalpha:mfgsolution} from \eqref{ut.plus-bound:mfgsolution}.

Next, we multiply \eqref{hjb-estimated-wo-Du:mfgsolution} by a test function $\varphi\in C^\infty(\Tt^d;\RR^+_0)$ as in the proof of Proposition~\ref{prop:uinitial.and.m.bound:mfgsolution} and integrate over $(0,s)\times\Tt^d$ to get
\begin{equation*}
    -\int_0^s\int_{\Tt^d} \varphi(x)\,du_t \leq \int_{\Tt^d} \varphi(x) \left(C\int_0^s (m^{\beta-1} + V(t,x))\,dt\right)\,dx.
\end{equation*}
Applying the trace Green formula to the left-hand side and expanding the bounds of time integration to $[0,T]$ on the right-hand side, we obtain
\begin{equation*}
     \int_{\Tt^d} \varphi(x) \bigl(u(0,x)-u(s^-,x)\bigr)\,dx \leq 
    \int_{\Tt^d} \varphi(x) \left(C\int_0^T (m^{\beta-1} + V(t,x))\,dt\right)\,dx.
\end{equation*}
Since this holds for all $\varphi\in C^\infty(\Tt^d; \RR^+_0)$ with the same value of $C$, we must have
\begin{equation*}
        u(0,x)-u(s^-,x) \leq C\int_0^T (m^{\beta-1} + V(t,x))\,dt \qquad \text{a.e. on } \Tt^d.
\end{equation*}
We move the first term of the left-hand side to the right and take the positive part to obtain
\begin{equation*}
    u^-(s^-,x) \leq u^-(0,x) + C\int_0^T (m^{\beta-1} + V(t,x))\,dt \qquad \text{a.e. on } \Tt^d.
\end{equation*}
Now we integrate the last expression over $\Tt^d$ and conclude \eqref{u-lower-uniform-in-time.in-terms-of-initial:mfgsolution}.
\end{proof}

\begin{proposition}\label{cross-estimate:mfgsolution}
    Consider the setting of Problem~\ref{prob.mfg}. Suppose Assumptions~\ref{onH.monotone}, \ref{onH.powergrowth},  
    \ref{onH.remedy-bound}, and
    \ref{onH.highest-order-p.has-no-txm} hold. Let $m\in L^\beta(Q;\RR^+_0)$. Then
    \begin{equation}\label{mDu-shift:mfgsolution}
        \int_0^{T-s}\int_{\Tt^d} m(t,x)|Du(t+s,x+y)|^\alpha\,dxdt \leq C + C\int_{\Tt^d} u^-(0,x)\,dx
    \end{equation}
    for a positive constant $C$ independent of $u\in \mathcal{U}(m)$, $s\in [0,T)$, and $y\in\Tt^d$.
\end{proposition}

\begin{proof}
    Consider $\upsilon\in BV(Q)$ defined by
    \[\upsilon(t,x) = \begin{cases}
      u(t+s,x+y) - 2\|u_T\|_{L^\infty(\Tt^d)},\enskip &\text{for } t<T-s,\\
      -\|u_T\|_{L^\infty(\Tt^d)}, &\text{for } t>T-s.
   \end{cases}\]
The conditions \eqref{Dupsilon.Lalpha} and \eqref{mDupsilon-shift-L1} follow from \eqref{integrability.further.du} and \eqref{integrability.further.mdu} by translation on the torus and by the definition of $\upsilon$ on $(T-s,T)\times\Tt^d$.
To verify
\eqref{upsilon-HJB-upper}, take \(\ell=0\). On \((0,T-s)\times\Tt^d\),
\[
-\upsilon_t^{\mathrm{ac}}+H(t,x,D\upsilon,0)
=
-u_t^{\mathrm{ac}}(t+s,x+y)+H(t,x,Du(t+s,x+y),0).
\]
Thus
\[
\begin{aligned}
&\bigl(-\upsilon_t^{\mathrm{ac}}+H(t,x,D\upsilon,0)\bigr)^+ \\
&\quad \leq
\bigl(-u_t^{\mathrm{ac}}(t+s,x+y)+H(t+s,x+y,Du(t+s,x+y),0)\bigr)^+ \\
&\qquad
+\bigl|H(t,x,Du(t+s,x+y),0)
-H(t+s,x+y,Du(t+s,x+y),0)\bigr|.
\end{aligned}
\]
The first term belongs to \(L^{\beta'}((0,T-s)\times\Tt^d)\) by
\eqref{integrability.further.H}. The second term is controlled by
Assumption \ref{onH.remedy-bound}, namely 
\eqref{remedy-bound}
together with \(Du\in L^\alpha(Q;\RR^d)\) and \(V\in L^{\beta'}(Q)\). Thus, this second term also
belongs to \(L^{\beta'}(Q)\). On \((T-s,T)\times\Tt^d\), \(\upsilon\) is
constant, and \((H(t,x,0,0))^+\in L^{\beta'}(Q)\) follows from
\eqref{res:H.upper.simple}. Hence \eqref{upsilon-HJB-upper} follows.

Moreover, $\upsilon$ also satisfies \eqref{upsilon.jumps.forward} because
   \[u_t^{\text{s}}\geq 0,\qquad u(T,x) - 2\|u_T\|_{L^\infty(\Tt^d)} \leq - \|u_T\|_{L^\infty(\Tt^d)},\qquad -\|u_T\|_{L^\infty(\Tt^d)}\leq u_T(x). \]
   
   Consequently, we may apply \eqref{transport-tested-upsilon.minus.u} to get
   \begin{equation*}
    \begin{aligned}
    & \int_0^{T-s}\int_{\Tt^d}m(t,x)\bigl( -u_t^{\text{ac}}(t+s,x+y) + D_pH(t,x,Du(t,x),m(t,x))\cdot Du(t+s,x+y) \bigr)\,dxdt\\
    & \qquad + \int_Q m\bigl(u_t^{\text{ac}} - D_pH(t,x,Du,m)\cdot Du \bigr)\,dxdt \\
    &\qquad \geq \int_{\Tt^d}m_0(x)(u(s^+,x+y)-u(0,x))\,dx - 2\|u_T\|_{L^\infty(\Tt^d)}\int_{\Tt^d} m_0(x)\,dx,
    \end{aligned}
    \end{equation*}
    where $u(s^+,\cdot)$ appearing on the right-hand side refers to the trace of $u\in BV(Q)$ on the hypersurface $\{s\}\times\Tt^d$ from the side of $(s,T)\times\Tt^d$. Since $u_t^{\text{s}}\geq 0$, we have $u(s^+,\cdot) \geq u(s^-,\cdot)$ a.e.~on $\Tt^d$, thus we replace $u(s^+,\cdot)$ by $u(s^-,\cdot)$ in the last expression. Moreover, we use the Hamilton--Jacobi conditions
    \begin{align*}
     & m(-u_t^{\text{ac}}+H(t,x,Du,m)) = 0, \\
     & -u_t^{\text{ac}}(t+s,x+y) + H(t+s,x+y,Du(t+s,x+y),m(t+s,x+y)) \leq 0,
    \end{align*}
    and rearrange to find
    \begin{equation}
     \label{ineq}
        \begin{aligned}
    & \int_0^{T-s}\int_{\Tt^d} m(t,x)H(t+s,x+y,Du(t+s,x+y),m(t+s,x+y))\,dxdt \\
    &  \leq \int_0^{T-s}\int_{\Tt^d} m(t,x)|D_pH(t,x,Du(t,x),m(t,x))||Du(t+s,x+y)|\,dxdt\\
    & \quad + \int_Q m\bigl(|H(t,x,Du,m)| + |D_pH(t,x,Du,m)||Du| \bigr)\,dxdt \\
    &\quad + \int_{\Tt^d} m_0(x)u^+(0,x)\,dx + \int_{\Tt^d} m_0(x)u^-(s^-,x+y)\,dx  \\&\quad+ 2\|u_T\|_{L^\infty(\Tt^d)}\int_{\Tt^d} m_0(x)\,dx.
    \end{aligned}
    \end{equation}
Now we estimate the left-hand side with \eqref{eq:assH.lower.simple}. This produces, besides the coercive term
\[
\frac1C\int_0^{T-s}\int_{\Tt^d}m(t,x)|Du(t+s,x+y)|^\alpha\,dxdt,
\]
the error terms
\[
C\int_0^{T-s}\int_{\Tt^d}m(t,x)m(t+s,x+y)^{\beta-1}\,dxdt
\]
and
\[
C\int_0^{T-s}\int_{\Tt^d}m(t,x)V(t+s,x+y)\,dxdt.
\]
Both are bounded uniformly in \((s,y)\) by Hölder's inequality, using
\(m\in L^\beta(Q)\) and \(V\in L^{\beta'}(Q)\).
For the right-hand side of \eqref{ineq}, we use for the first line, \eqref{eq:assH.DpH.upper.simple}, the first term on the second line, i.e.~the one involving $|H(t,x,Du,m)|$ with \eqref{eq:assH.lower.simple} and \eqref{res:H.upper.simple}, the next term by first noting $|D_pH(t,x,Du,m)||Du| \leq |D_pH(t,x,Du,m)|^{\frac{\alpha}{\alpha-1}} + |Du|^\alpha$ and then using \eqref{eq:assH.DpH.upper.simple}, and finally the remaining terms of the right-hand side, i.e.~the last line, with \eqref{uplus-initial:mfgsolution} and \eqref{u-lower-uniform-in-time.in-terms-of-initial:mfgsolution}. Meanwhile, we absorb the terms that do not depend on the solution $u$ and the coordinates $(s,y)$ into the constant $C$, hence we get
    \begin{equation*}
    \begin{aligned}
    & \frac{1}{C}\int_0^{T-s}\int_{\Tt^d} m(t,x)|Du(t+s,x+y)|^\alpha\,dxdt \\
    & \qquad\leq C\int_0^{T-s}\int_{\Tt^d} m(t,x)|Du(t+s,x+y)||Du(t,x)|^{\alpha-1}\,dxdt\\
    & \qquad + C\int_0^{T-s}\int_{\Tt^d} m(t,x)|Du(t+s,x+y)|(m(t,x)^{\frac{(\alpha-1)}{\alpha}(\beta-1)} + V(t,x)^{\frac{\alpha-1}{\alpha}})\,dxdt\\
    & \qquad + C\int_Q m|Du|^\alpha\,dxdt + C\int_{\Tt^d} u^-(0,x)\,dx +  C.
    \end{aligned}
    \end{equation*}
    Next we use Young's inequality to absorb the terms involving $|Du(t+s,x+y)|$ on the right-hand side into the left-hand side, then we once again suppress the terms that are independent of $u$ and $(s,y)$, hence we obtain
    \begin{equation*}
    \begin{aligned}
    & \int_0^{T-s}\int_{\Tt^d} m(t,x)|Du(t+s,x+y)|^\alpha\,dxdt  \leq C\int_Q m|Du|^\alpha\,dxdt + C\int_{\Tt^d} u^-(0,x)\,dx +  C.
    \end{aligned}
    \end{equation*}
    Since the first term on the right-hand side is bounded by \eqref{mDualpha-basic:mfgsolution}, we conclude.
\end{proof}

\begin{lemma}\label{strong-ae-stable}
    Consider the setting of Problem~\ref{prob.mfg}. Suppose Assumptions~\ref{onH.monotone}, \ref{onH.powergrowth}, \ref{onH.high-order-p.has-no-m}, \ref{onH.remedy-bound}, and \ref{onH.highest-order-p.has-no-txm} hold. Let $m\in L^\beta(Q;\RR^+_0)$ and let $u_1, u_2, \ldots \in \mathcal{U}(m)$ be such that $u_1 \leq u_2 \leq u_3 \leq \ldots$ a.e. Then the pointwise a.e.~limit, $u_\infty$, of the sequence $\{u_n\}$ satisfies $u_\infty\in \mathcal{U}(m)$.
\end{lemma}

\begin{proof}
    Since
    \begin{equation}\label{un-trivial-lowerbound}
        \int_{\Tt^d} u_n^-(0,x)\,dx \leq \int_{\Tt^d} u_1^-(0,x)\,dx
    \end{equation}
and the right-hand side is independent of~$n$, \eqref{ut.minus-Lbeta':mfgsolution}, \eqref{ut.plus-bound:mfgsolution}, and \eqref{Du-Lalpha:mfgsolution} show that the total variation norm of $((u_n)_t, Du_n)$ is bounded uniformly over $n$. Moreover,
    \[\int_{\Tt^d} |u_n(0,x)|\,dx = \int_{\Tt^d} u_n^+(0,x)\,dx + \int_{\Tt^d} u_n^-(0,x)\,dx \leq C,\]
    because of H\"older's inequality applied to \eqref{uplus-initial:mfgsolution}, and the uniform bound \eqref{un-trivial-lowerbound}. Hence, the Poincar\'e--Wirtinger trace inequality bounds $\{u_n\}$ uniformly in $BV(Q)$. Consequently, $u_n$ must converge to $u_\infty$ weakly-$\ast$ in $BV(Q)$.

    Now, \eqref{Du-Lalpha:mfgsolution} implies that $Du_n$ converges weakly to $Du_\infty$ in $L^\alpha(Q;\RR^d)$. 
    Thus, $u_\infty$ satisfies \eqref{integrability.further.du}. Passing to the limit in \eqref{weak-sol} for $(m,u_n)$, using Green's formula to rewrite the term $\int_Q\mu\,d(u_n)_t$ and the cancellation of the endpoint traces because $\mu(0,\cdot)=m_0$, we deduce that $(m,u_\infty)$ is a solution to Problem~\ref{prob.vi} in the sense of Definition~\ref{def.weak.sol}.
    Next, \eqref{ut.minus-Lbeta':mfgsolution} implies that $u_\infty$ satisfies \eqref{integrability.further.ut}, and \eqref{mDu-shift:mfgsolution} together with \eqref{un-trivial-lowerbound} imply that $u_\infty$ satisfies \eqref{integrability.further.mdu}. Therefore, it only remains to prove that $u_\infty$ satisfies \eqref{integrability.further.H}.

We prove this in an analogous way to the corresponding argument in the proof of Theorem~\ref{weak.exists}. Namely, the convexity of $H(t,x,\cdot,m)$ as in Remark~\ref{rmk:cxty}, as well as \eqref{def:hjb-strong} and~\eqref{eq:high-order-p.has-no-m} imply
\[\begin{aligned}
 & -(u_n)_t+H(t,x,Du_\infty,0) + D_pH(t,x,Du_\infty,0)\cdot (Du_n-Du_\infty) \\
 & \qquad \leq -(u_n)_t+H(t,x,Du_n,0) \\
& \qquad\qquad \leq H(t,x,Du_n,0) - H(t,x,Du_n,m) \\
& \qquad\qquad\qquad \leq C(|Du_n|^{\alpha\frac{(\beta-1)}{\beta}} + m^{{\beta-1}}+V(t,x)).
\end{aligned}
\]
We note that the last line is uniformly bounded in $L^{\beta'}(Q)$ and conclude as in the proof of Theorem~\ref{weak.exists}.
\end{proof}

\subsection{Comparison and Lattice Properties}

We begin with a comparison principle between solutions of Problem~\ref{prob.mfg} in the sense of Definition~\ref{def.strong.sol} and subsolutions of the Hamilton--Jacobi equation.

\begin{lemma}\label{lem:hjb-sub-under-value}
  Consider the setting of Problem~\ref{prob.mfg}. Suppose Assumptions~\ref{onH.monotone}, \ref{onH.powergrowth}, and \ref{onH.highest-order-p.has-no-txm} hold. Let $m\in L^\beta(Q;\RR^+_0)$ and let $u\in \mathcal{U}(m)$ and $\tilde{u}\in \mathcal{S}(m)$ be such that $\tilde{u}\geq u$ a.e. Then $\tilde{u} = u$ a.e.~in the set $\{m>0\}$ and $\tilde{u}(0,\cdot) = u(0,\cdot)$ a.e.~in the set $\{m_0 > 0\}$. In particular, $\tilde{u}\in \mathcal{U}(m)$.
\end{lemma}

\begin{proof}
  Let $\lambda \colon [0,T] \to [0,1]$ be a non-decreasing, smooth function. 
  We claim that  $\upsilon := (1-\lambda(t))u+\lambda(t)\tilde{u}$ satisfies \eqref{Dupsilon.Lalpha}, \eqref{upsilon.jumps.forward}, \eqref{upsilon-HJB-upper}, and \eqref{mDupsilon-shift-L1}, which allow us to apply \eqref{transport-tested-upsilon.minus.u} in Lemma~\ref{lem:transport-tested-upsilon.minus.u}. To prove the claim, we compute
  \begin{equation*}
 \begin{aligned}
D\upsilon & = (1-\lambda(t))Du+\lambda(t)D\tilde{u}, \\
\upsilon_t^{\text{s}} & = (1-\lambda(t))u_t^{\text{s}}+\lambda(t)\tilde{u}_t^{\text{s}}, \\
\upsilon_t^{\text{ac}} & = (1-\lambda(t))u_t^{\text{ac}}+\lambda(t)\tilde{u}_t^{\text{ac}} + \lambda'(t)(\tilde{u}-u),
\end{aligned}
\end{equation*}
and note that \eqref{Dupsilon.Lalpha} follows from \eqref{integrability.further.du} for $u$ and $\tilde{u}$, similarly \eqref{upsilon.jumps.forward} follows from the corresponding conditions for $u$ and $\tilde{u}$, and finally \eqref{mDupsilon-shift-L1} follows in the same manner. Moreover,
\[\begin{aligned} -\upsilon_t^{\text{ac}}+H(t,x,D\upsilon,0) & \leq -\lambda'(t)(\tilde{u}-u) \\
& \quad + (1-\lambda(t))(-u_t^{\text{ac}}+H(t,x,Du,0)) \\
& \qquad + \lambda(t)(-\tilde{u}_t^{\text{ac}}+H(t,x,D\tilde{u},0)) \qquad \text{a.e.~in } Q,
\end{aligned}
 \]
 by the convexity of $H(t,x,\cdot,0)$ as in Remark~\ref{rmk:cxty}; hence \eqref{upsilon-HJB-upper} follows from \eqref{integrability.further.H} for $u$ and $\tilde{u}$, as well as the fact that $\lambda'(t)(\tilde{u}-u) \geq 0$ a.e.

 Consequently, plugging $\upsilon := (1-\lambda(t))u+\lambda(t)\tilde{u}$ in \eqref{transport-tested-upsilon.minus.u}, we obtain
  \begin{equation*}
  \begin{aligned}
  & \lambda(0)\int_{\Tt^d}m_0(x)(\tilde{u}(0,x)-u(0,x))\,dx + \int_Q \lambda'(t) m(\tilde{u}-u)\,dxdt\\
  & \qquad \leq \int_Q \lambda(t) m\bigl( (u_t^{\text{ac}}-\tilde{u}_t^{\text{ac}}) + D_pH(t,x,Du,m)\cdot (D\tilde{u}-Du) \bigr)\,dxdt.
  \end{aligned}
  \end{equation*}
  Combining this
  with the Hamilton--Jacobi conditions $m(-u_t^{\text{ac}}+H(t,x,Du,m)) = 0$ and $-\tilde{u}_t^{\text{ac}}+H(t,x,D\tilde{u},m) \leq 0$, we get
  \begin{equation*}
  \begin{aligned}
  & \lambda(0)\int_{\Tt^d}m_0(x)(\tilde{u}(0,x)-u(0,x))\,dx + \int_Q \lambda'(t) m(\tilde{u}-u)\,dxdt\\
  & \enskip \leq \int_Q \lambda(t) m\bigl(H(t,x,Du,m)-H(t,x,D\tilde{u},m) + D_pH(t,x,Du,m)\cdot (D\tilde{u}-Du) \bigr)\,dxdt.
  \end{aligned}
  \end{equation*}
 The integrands on the left-hand side are non-negative by $\tilde{u}\geq u$ a.e., and by preservation of the order under BV traces, while the integrand on the right-hand side is non-positive by the convexity of $H(t,x,\cdot,m)$. 
The left-hand side is nonnegative, whereas the right-hand side is
nonpositive. Hence both sides are equal to zero. In particular,
\[
\lambda(0)\int_{\Tt^d}m_0(x)(\tilde u(0,x)-u(0,x))\,dx
+
\int_Q\lambda'(t)m(\tilde u-u)\,dxdt=0.
\]
Choosing, for instance, \(\lambda(t)=(t+1)/(T+1)\), we have
\(\lambda(0)>0\) and \(\lambda'(t)>0\). Since
\(\tilde u\geq u\) a.e.~in \(Q\), and since the BV trace preserves this
order, both integrands are nonnegative. Therefore
\[
m(\tilde u-u)=0 \quad\text{a.e. in }Q
\]
and
\[
m_0(\tilde u(0,\cdot)-u(0,\cdot))=0
\quad\text{a.e. on }\Tt^d.
\]
This gives \(\tilde u=u\) a.e.~on \(\{m>0\}\) and
\(\tilde u(0,\cdot)=u(0,\cdot)\) a.e.~on \(\{m_0>0\}\).

It remains to justify the final assertion. Since \(\tilde u\in\mathcal{S}(m)\),
the required integrability bounds, the Hamilton--Jacobi inequality, and
the terminal inequality are already part of the definition of
\(\mathcal{S}(m)\). Moreover, by the locality of the absolutely continuous
part of the distributional derivative for \(BV\) functions, the identity
\(\tilde u=u\) a.e.~on \(\{m>0\}\) implies
\[
\tilde u_t^{\rm ac}=u_t^{\rm ac},
\qquad
D\tilde u=Du
\quad\text{a.e. on }\{m>0\}.
\]
Consequently, all terms multiplied by \(m\) in the transport equation and
in \eqref{def:singular-and-terminal} agree for \(u\) and \(\tilde u\).
Similarly, the identity
\(\tilde u(0,\cdot)=u(0,\cdot)\) a.e.~on \(\{m_0>0\}\) implies
\[
m_0\tilde u(0,\cdot)=m_0u(0,\cdot)
\quad\text{a.e.~on }\Tt^d.
\]
Thus the transport equation and \eqref{def:singular-and-terminal} for
\((m,\tilde u)\) follow from the corresponding properties of \((m,u)\).
The Hamilton--Jacobi equality on \(\{m>0\}\) also follows from the same
locality argument. Therefore \(\tilde u\in \mathcal{U}(m)\).
\end{proof}

The next lemma provides the lattice property needed in the construction of
the maximal representative \(u^*\): Hamilton--Jacobi subsolutions are stable
under maxima, and taking the maximum with an MFG value function preserves the
full MFG structure.

\begin{lemma}\label{max-of-subsolutions}
    Consider the setting of Problem~\ref{prob.mfg}. Suppose Assumptions~\ref{onH.monotone}, \ref{onH.powergrowth}, and \ref{onH.highest-order-p.has-no-txm} hold. Let $m\in L^\beta(Q;\RR^+_0)$ and let $u_1, u_2\in \mathcal{S}(m)$. Then $\max(u_1, u_2)\in \mathcal{S}(m)$. Moreover, if either $u_1\in \mathcal{U}(m)$ or $u_2\in \mathcal{U}(m)$, then $\max(u_1, u_2)\in \mathcal{U}(m)$.
\end{lemma}

\begin{proof}
We prove the first statement by verifying that $\tilde{u}:=\max(u_1, u_2)$ satisfies the conditions in~\ref{def:Um-tilde}. 
Applying Corollary~\ref{cor:bv-chain} to \(u_1,u_2\) as \(BV\)-functions
on \(Q\), and then projecting the full space-time derivative onto its time
and spatial components, we obtain
\begin{equation}\label{maxut-s-conv.comb}
\tilde{u}_t^{\text{s}} = h(u_1)_t^{\text{s}} + (1-h)(u_2)_t^{\text{s}},
\end{equation}
as an identity of Radon measures 
for a Borel measurable function, $h \colon Q \to [0,1]$, and
\begin{equation}\label{maxuDt-conv.comb}
(\tilde{u}_t^{\text{ac}}, D\tilde{u}) \in \bigl\{ ((u_1)_t^{\text{ac}}, Du_1), \, ((u_2)_t^{\text{ac}}, Du_2) \bigr\} \qquad \text{a.e.~in } Q.
\end{equation}
Now, \eqref{integrability.further.du} and \eqref{integrability.further.mdu} for $\tilde{u}$ immediately follow from the same conditions for $u_1$ and $u_2$, considering \eqref{maxuDt-conv.comb}. Moreover, \eqref{maxut-s-conv.comb} implies $\tilde{u}_t^{\text{s}} \geq 0$ because $(u_1)_t^{\text{s}} \geq 0$ and $(u_2)_t^{\text{s}} \geq 0$. Then \eqref{integrability.further.ut} for $\tilde{u}$ follows from the same condition for $u_1$ and $u_2$, as before. Next, since \(d\tilde u_t^{\mathrm{s}}\geq0\), we have, as measures,
\[
-d\tilde u_t+H(t,x,D\tilde u,0)\,dxdt
\leq
\bigl(-\tilde u_t^{\mathrm{ac}}+H(t,x,D\tilde u,0)\bigr)\,dxdt.
\]
Therefore the positive part of the measure on the left is controlled by
\[
\bigl(-\tilde u_t^{\mathrm{ac}}+H(t,x,D\tilde u,0)\bigr)^+dxdt.
\]
By \eqref{maxuDt-conv.comb}, the density
\[
-\tilde u_t^{\mathrm{ac}}+H(t,x,D\tilde u,0)
\]
is a.e.~equal to one of the two densities
\[
-(u_1)_t^{\mathrm{ac}}+H(t,x,Du_1,0),
\qquad
-(u_2)_t^{\mathrm{ac}}+H(t,x,Du_2,0).
\]
Hence its positive part belongs to \(L^{\beta'}(Q)\), because \(u_1,u_2\in
\mathcal{S}(m)\). Thus \eqref{integrability.further.H} holds for \(\tilde u\).
 
Similarly, we have
 \begin{equation*}
-\tilde{u}_t^{\text{ac}}+H(t,x,D\tilde{u},m) \in \bigl\{-(u_1)_t^{\text{ac}}+H(t,x,Du_1,m), \enskip -(u_2)_t^{\text{ac}}+H(t,x,Du_2,m)\bigr\} \qquad \text{a.e.,}
\end{equation*}
which implies $-\tilde{u}_t+H(t,x,D\tilde{u},m) \leq 0$ by the same condition for $u_1$ and $u_2$, after noting $\tilde{u}_t^{\text{s}} \geq 0$ as above.
Finally, the trace operator is compatible with Lipschitz compositions, so
\[
\tilde u(T,\cdot)=\max\{u_1(T,\cdot),u_2(T,\cdot)\}
\quad\text{a.e. on }\Tt^d.
\]
Since \(u_1(T,\cdot)\leq u_T\) and \(u_2(T,\cdot)\leq u_T\), we obtain
\(\tilde u(T,\cdot)\leq u_T\).

For the second statement, suppose without loss of generality that $u_1\in \mathcal{U}(m)$. The desired result follows from Lemma~\ref{lem:hjb-sub-under-value} applied with $u = u_1$ and $\tilde{u} = \max(u_1,u_2)\geq u_1$, since $\tilde{u}\in \mathcal{S}(m)$ by the first statement.
\end{proof}

\begin{proof}[Proof of Theorem~\ref{maximal-u-exists}, Item~\ref{thmitem:maximal.u}]
Since \(\mathcal{U}(m)\) is non-empty and \(\mathcal{U}(m)\subset \mathcal{S}(m)\), the family
\(\mathcal{S}(m)\) is non-empty. Applying Lemma~\ref{lub-lemma} with
\(\mathcal{F}=\mathcal{S}(m)\), we obtain a sequence
\(u_1,u_2,\ldots\in\mathcal{S}(m)\) and a measurable function \(u^*\) on \(Q\)
such that
\[
\max(u_1,u_2,\ldots,u_n)\to u^*
\quad\text{a.e. in }Q,
\]
and
\[
u\leq u^* \quad\text{a.e. in }Q
\qquad\text{for every }u\in\mathcal{S}(m).
\]

Fix \(u_0\in \mathcal{U}(m)\). By Lemma~\ref{max-of-subsolutions}, the finite maximum
\[
v_n:=\max(u_0,u_1,\ldots,u_n)
\]
belongs to \(\mathcal{U}(m)\) for every \(n\). Moreover, \(v_n\) is non-decreasing in
\(n\). Since the finite maximum \(v_n\) also belongs to \(\mathcal{S}(m)\), the
maximality property above gives \(v_n\leq u^*\) a.e.~in \(Q\). On the other
hand,
\[
\max(u_1,\ldots,u_n)\leq v_n.
\]
Thus \(v_n\to u^*\) a.e.~in \(Q\). Lemma~\ref{strong-ae-stable} therefore
implies \(u^*\in \mathcal{U}(m)\). By construction, \(u^*\) dominates every element of
\(\mathcal{S}(m)\).

The uniqueness of such a maximal element is immediate. Indeed, if
\(v^*\in \mathcal{U}(m)\) also dominates every element of \(\mathcal{S}(m)\), then, since
both \(u^*\) and \(v^*\) belong to \(\mathcal{S}(m)\), we have \(u^*\leq v^*\)
and \(v^*\leq u^*\) a.e.~in \(Q\).

Now let \(u\in \mathcal{U}(m)\). Since \(u^*\in \mathcal{U}(m)\subset\mathcal{S}(m)\) and
\(u\leq u^*\) a.e.~in \(Q\), Lemma~\ref{lem:hjb-sub-under-value}, applied
with \(\tilde u=u^*\), gives
\[
u=u^* \quad\text{a.e. on }\{m>0\},
\]
and
\[
u(0,\cdot)=u^*(0,\cdot)
\quad\text{a.e. on }\{m_0>0\}.
\]
Conversely, let \(u\in\mathcal{S}(m)\) satisfy
\[
u=u^* \quad\text{a.e. on }\{m>0\},
\qquad
u(0,\cdot)=u^*(0,\cdot)
\quad\text{a.e. on }\{m_0>0\}.
\]
We prove that \(u\in \mathcal{U}(m)\). Since \(u\in\mathcal{S}(m)\), the required
integrability bounds, the Hamilton--Jacobi inequality, and the terminal
inequality are already available. By the locality of the absolutely continuous
part of the distributional derivative for \(BV\) functions, the identity
\(u=u^*\) a.e.~on \(\{m>0\}\) implies
\[
u_t^{\mathrm{ac}}=(u^*)_t^{\mathrm{ac}},
\qquad
Du=Du^*
\quad\text{a.e. on }\{m>0\}.
\]
Hence the Hamilton--Jacobi equality on \(\{m>0\}\) for \(u\) follows from
the corresponding equality for \(u^*\).

The same locality observation also shows that all terms multiplied by \(m\)
in the transport equation and in \eqref{def:singular-and-terminal} agree for
\(u\) and \(u^*\). Moreover,
\[
u(0,\cdot)=u^*(0,\cdot)
\quad\text{a.e. on }\{m_0>0\}
\]
implies
\[
m_0u(0,\cdot)=m_0u^*(0,\cdot)
\quad\text{a.e. on }\Tt^d.
\]
Therefore the transport equation and \eqref{def:singular-and-terminal} for
\((m,u)\) follow from the corresponding properties of \((m,u^*)\). Thus
\((m,u)\) is an MFG-solution with the prescribed integrability bounds, and so
\(u\in \mathcal{U}(m)\). This completes the characterization of \(\mathcal{U}(m)\).
\end{proof}

\subsection{Uniqueness of Density}

The uniqueness proof compares two solutions corresponding to two possibly
distinct densities. In this setting, the shift-integrability condition
\eqref{mDupsilon-shift-L1} is not available for one solution tested against
the other. We therefore prove a variant of
Lemma~\ref{lem:transport-tested-upsilon.minus.u} in which
\eqref{mDupsilon-shift-L1} is replaced by the stronger Hamilton--Jacobi upper
bound obtained by taking the spatial perturbation \(\ell=0\). We first
establish the approximation estimates needed for this variant.

For
\(\upsilon\in BV(Q)\), we use the same one-sided mollification introduced in \eqref{eq:upsilon-mollified}
consider the extension in \eqref{eq:upsilon-extension}, and we define
\(\upsilon_{\xi,\tau}\) by \eqref{eq:upsilon-mollified}. 
In contrast with \eqref{upsilon_n.defined}, we
now correct the terminal value only by a constant.
Let
\(\xi_n\downarrow0\), and choose \(\tau_n\downarrow0\) so that
\eqref{xi-tau-choice} holds.  For such a sequence, set
\begin{equation}\label{upsilon_n.redefined}
\upsilon_n(t,x)
:=
\upsilon_{\xi_n,\tau_n}(t,x)-\delta_n,
\end{equation}
where
\[
\delta_n:=
\sup_{x\in\Tt^d}
\left|
u_T(x)
-
\int_{\Tt^d}
u_T(x')
\frac{1}{(\xi_n)^d}
\overline{\zeta}\left(\frac{x'-x}{\xi_n}\right)
\,dx'
\right|.
\]
Since \(u_T\in C^1(\Tt^d)\), we have \(\delta_n\to0\). Moreover,
\(\upsilon_n\in C^1(\bar Q)\), and the definition gives the terminal
inequality
\[
\upsilon_n(T,x)\leq u_T(x)
\qquad\text{for all }x\in\Tt^d.
\]
The following approximation result is the analogue of
Proposition~\ref{lem:trace-approx} for this terminal-inequality
regularization.

\begin{proposition}\label{prop:terminal-ineq-approx}
Let \(\upsilon\in BV(Q)\) satisfy \(D\upsilon\in L^\alpha(Q;\RR^d)\). Let
\(\upsilon_n\) be defined by \eqref{upsilon_n.redefined}. Then, as \(n\to\infty\),
\begin{align}
  D\upsilon_n \to D\upsilon
  \quad &\text{in } L^\alpha(Q;\RR^d), \label{Dupsilon.limit-aux:repeated}\\
  \upsilon_n(0,\cdot)\to\upsilon(0,\cdot)
  \quad &\text{in } L^1(\Tt^d). \label{upsilon-trace-limit-aux:repeated}
\end{align}
\end{proposition}

\begin{proof}
Differentiating
\eqref{upsilon_n.redefined} in the spatial variables gives
\begin{equation}\label{Dupsilon_n.formula-new}
D\upsilon_n(t,x)
=
\int_t^\infty\int_{\Tt^d}
D\upsilon^*(t',x')
K_{{\xi_n}, {\tau_n}}(t',x'; t,x)
\,dx'dt'.
\end{equation}
Thus \(D\upsilon_n\) is the one-sided space-time mollification of
\(D\upsilon^*\). Since \(D\upsilon^*=D\upsilon\) on \(Q\), while
\(D\upsilon^*=Du_T\) for \(t>T\), the standard convergence of approximate
identities, together with \(D\upsilon\in L^\alpha(Q;\RR^d)\) and
\(Du_T\in L^\infty(\Tt^d;\RR^d)\), yields
\eqref{Dupsilon.limit-aux:repeated}.

For the initial trace, the integration-by-parts argument used in the proof of
Proposition~\ref{lem:trace-approx} gives
\[
\left|
\upsilon_{\xi_n,\tau_n}(0,x)
-
\int_{\Tt^d}
\upsilon(0,x')
\frac{1}{(\xi_n)^d}
\overline{\zeta}\left(\frac{x'-x}{\xi_n}\right)
\,dx'
\right|
\leq
\xi_n\|\overline{\zeta}\|_{L^\infty},
\]
where \eqref{xi-tau-choice} is used in the last estimate. The right-hand side
converges to zero uniformly in \(x\). Since \(\delta_n\to0\), and since the
spatial mollification of \(\upsilon(0,\cdot)\) converges to
\(\upsilon(0,\cdot)\) in \(L^1(\Tt^d)\), we obtain
\eqref{upsilon-trace-limit-aux:repeated}.
\end{proof}

\begin{proposition}[Hamilton--Jacobi limsup without shift-integrability]
\label{prop:hjb-limsup-no-shift}
Consider the setting of Problem~\ref{prob.mfg}. Suppose that
Assumptions~\ref{onH.monotone}, \ref{onH.powergrowth}, and
\ref{onH.high-order-pm.has-no-tx} hold. Let
\(m\in L^\beta(Q;\RR^+_0)\), and let \(\upsilon\in BV(Q)\) satisfy
\eqref{Dupsilon.Lalpha}, \eqref{upsilon.jumps.forward}, and
 \begin{equation}\label{upsilon-HJB-upper:stricter}
        \bigl(-\upsilon_t^{\text{ac}}+H(t,x,D\upsilon,0)\bigr)^+ \in L^{\beta'}(Q).
    \end{equation}
Let \(\upsilon_n\) be defined by
\eqref{upsilon_n.redefined}. Then
\begin{equation}\label{m-times-hjb.upsilon-limit:repeated}
\limsup_{n\to\infty}
\int_Q m\bigl(-(\upsilon_n)_t+H(t,x,D\upsilon_n,m)\bigr)\,dxdt
\leq
\int_Q m\bigl(-\upsilon_t^{\mathrm{ac}}+H(t,x,D\upsilon,m)\bigr)\,dxdt.
\end{equation}
\end{proposition}

\begin{remark}
The integral on the right-hand side of
\eqref{m-times-hjb.upsilon-limit:repeated} is well-defined as an extended
real number in \([-\infty,\infty)\). Indeed, since
\(H(t,x,p,\cdot)\) is non-increasing by Remark~\ref{rmk:cxty}, we have
\[
-\upsilon_t^{\mathrm{ac}}+H(t,x,D\upsilon,m)
\leq
-\upsilon_t^{\mathrm{ac}}+H(t,x,D\upsilon,0).
\]
Hence
\[
\bigl(-\upsilon_t^{\mathrm{ac}}+H(t,x,D\upsilon,m)\bigr)^+
\leq
\bigl(-\upsilon_t^{\mathrm{ac}}+H(t,x,D\upsilon,0)\bigr)^+ .
\]
The right-hand side belongs to \(L^{\beta'}(Q)\) by
\eqref{upsilon-HJB-upper:stricter}; therefore, multiplying by
\(m\in L^\beta(Q)\) and using Hölder's inequality shows that the positive
part of the integrand in
\[
\int_Q m\bigl(-\upsilon_t^{\mathrm{ac}}+H(t,x,D\upsilon,m)\bigr)\,dxdt
\]
is integrable.  In the application below, the inequality obtained from
\eqref{m-times-hjb.upsilon-limit:repeated} provides a finite lower bound,
which also excludes the value \(-\infty\).
\end{remark}

    \begin{proof}
We follow the strategy of Proposition~\ref{lem:hjb-limsup}, but the error
term is treated using Assumption~\ref{onH.high-order-pm.has-no-tx} rather
than the shift-integrability condition~\eqref{mDupsilon-shift-L1}. Since \(\upsilon_t^{\mathrm{s}}\geq0\) and
\(\upsilon(T,\cdot)\leq u_T\), the singular part of the time derivative of
the extension \(\upsilon^*\) is nonnegative. Hence, as in
\eqref{upsilon_t.singular.removed},
\begin{equation}\label{upsilon_t.singular.removed:repeated}
-(\upsilon_n)_t
\leq
\int_t^\infty\int_{\Tt^d}
\bigl(-\upsilon^*(t',x')\bigr)_{t'}^{\mathrm{ac}}
\cdot K_{{\xi_n}, {\tau_n}}(t',x';t,x)\,dx'dt' .
\end{equation}
Moreover, by \eqref{Dupsilon_n.formula-new} and the convexity of
\(H(t,x,\cdot,m(t,x))\),
\begin{equation}\label{hjb-upsilon_n.cxty.upper-bound-new}
H(t,x,D\upsilon_n(t,x),m(t,x))
\leq
\int_t^\infty\int_{\Tt^d}
H(t,x,D\upsilon^*(t',x'),m(t,x))
K_{{\xi_n}, {\tau_n}}(t',x';t,x)\,dx'dt' .
\end{equation}

We take $f\colon\RR^+_0\to \RR$ as in Lemma~\ref{lem:highest-order-m.separated} and define
\[
h(t,x)
:=
\begin{cases}
-\upsilon_t^{\mathrm{ac}}(t,x)+H(t,x,D\upsilon(t,x),m(t,x)) + f(m(t,x)),
& (t,x)\in Q,\\
0, & t>T,
\end{cases}
\]
and define \(h^{\mathrm{c}}\) by

\begin{equation*}
 h^{\mathrm{c}}(t',x';t,x)
:= \begin{cases}
            \begin{aligned}
   & \Bigl(H(t,x,D\upsilon(t',x'),m(t,x)) + f(m(t,x)) \\
   & \mathrlap{\quad - H(t',x',D\upsilon(t',x'),m(t',x')) - f(m(t',x'))\Bigr),}\hphantom{- H(t',x',D\upsilon(t',x'),m(t',x')) - f(m(t',x'))\Bigr)} 
       \end{aligned} \quad\qquad &  t' < T, \\
       H(t,x,Du_T(x'),m(t,x))+f(m(t,x)), \qquad &  t' > T.
        \end{cases}
\end{equation*}
Adding \eqref{upsilon_t.singular.removed:repeated} and
\eqref{hjb-upsilon_n.cxty.upper-bound-new}, and then adding and subtracting
the corresponding terms, gives
\begin{equation}\label{upsilon_n.hjb.bounded.by.two-terms-new}
\begin{aligned}
&-(\upsilon_n)_t+H(t,x,D\upsilon_n,m)\\
&\qquad\leq
\int_t^\infty\int_{\Tt^d}
\bigl(h(t',x')+h^{\mathrm{c}}(t',x';t,x) -f(m(t,x))\bigr)
K_{{\xi_n}, {\tau_n}}(t',x';t,x)\,dx'dt' \qquad\text{a.e.~in } Q.
\end{aligned}
\end{equation}

In view of~\eqref{upsilon_n.hjb.bounded.by.two-terms-new}, it is enough to show
\begin{align}
& \limsup_{n\to\infty}
\int_Q m(t,x)
\left(\int_t^\infty\int_{\Tt^d}
h(t',x')K_{{\xi_n}, {\tau_n}}(t',x';t,x)\,dx'dt'\right)\,dxdt \nonumber\\
&\qquad\leq
\int_Q m(t,x)\bigl(-\upsilon_t^{\mathrm{ac}}
+H(t,x,D\upsilon,m) +  f(m(t,x))\bigr)\,dxdt,
\label{mh-limsup-new}
\end{align}
and
\begin{align}
& \lim_{n\to\infty}
\int_Q m(t,x)
\left(\int_t^\infty\int_{\Tt^d}
h^{\mathrm{c}}(t',x';t,x)K_{{\xi_n}, {\tau_n}}(t',x';t,x)\,dx'dt'\right)\,dxdt
=0.
\label{mhc-lim-new}
\end{align}

\medskip\noindent\textbf{Proof of \eqref{mh-limsup-new}:} We note that $h^+\in L^{\beta'}((0,\infty)\times\Tt^d)$ by~\eqref{upsilon-HJB-upper:stricter}  and \eqref{eq:highest-order-m.separated} because $H(t,x,p,\cdot)$ is non-increasing as in Remark~\ref{rmk:cxty}. Hence,
  \begin{equation*}
\int_t^\infty \int_{\Tt^d} h^+(t',x') K_{{\xi_n}, {\tau_n}}(t',x';t,x)\,dx'dt' \to h^+(t,x) \qquad\text{in } L^{\beta'}(Q),
  \end{equation*}
  which yields
  \begin{equation}\label{hplus-limit-new}
  \begin{aligned}
& \lim_{n\to\infty} \int_Q m(t,x) \left(\int_t^\infty \int_{\Tt^d} h^+(t',x') K_{{\xi_n}, {\tau_n}}(t',x';t,x)\,dx'dt'\right)\,dxdt \\
& \qquad = \int_Q m(t,x)h^+(t,x)\,dxdt.
  \end{aligned}
  \end{equation}
  On the other hand, \(h^-\in L^1((0,\infty)\times\Tt^d)\) by the \(BV\)
  regularity of \(\upsilon\), the power-growth bounds on \(H\), and
  \eqref{eq:highest-order-m.separated}. Hence, by the adapted Lebesgue
  differentiation argument,
  \begin{equation*}
\int_t^\infty \int_{\Tt^d} h^-(t',x') K_{{\xi_n}, {\tau_n}}(t',x';t,x)\,dx'dt' \to h^-(t,x) \qquad\text{a.e.~in } Q,
  \end{equation*}
  which yields
  \begin{equation}\label{hminus-limit-new}
  \begin{aligned}
  & \liminf_{n\to\infty} \int_Q m(t,x) \left(\int_t^\infty \int_{\Tt^d} h^-(t',x') K_{{\xi_n}, {\tau_n}}(t',x';t,x)\,dx'dt'\right)\,dxdt \\
& \qquad \geq \int_Q m(t,x)h^-(t,x)\,dxdt,
  \end{aligned}
  \end{equation}
  by Fatou's lemma. Subtracting~\eqref{hminus-limit-new} from~\eqref{hplus-limit-new}, we obtain~\eqref{mh-limsup-new}.

\medskip\noindent\textbf{Proof of \eqref{mhc-lim-new}:}
It is enough to prove
\[I_{n} := \int_t^\infty\int_{\Tt^d}
h^{\mathrm{c}}(t',x';t,x)K_{{\xi_n}, {\tau_n}}(t',x';t,x)\,dx'dt' \to 0 \qquad \text{in } L^{\beta'}(Q),
\]
because then \eqref{mhc-lim-new} follows from Hölder's inequality
and \(m\in L^\beta(Q)\). We first observe that
\begin{equation}\label{hc2-converges-ae}
 I_{n} \to 0 \qquad \text{a.e.~in } Q,
 \end{equation}
which can be justified in the same way as \eqref{I.n-converges-ae} by defining
\begin{equation*}
    \theta = \begin{cases}
        D\upsilon, \qquad & \text{on } Q, \\
        Du_T, & \text{on } (T,T+\tau_1)\times\Tt^d, \\
        0, & \text{on } (T+\tau_1,\infty)\times\Tt^d,
    \end{cases}
\end{equation*}
and $F(t,x,s) = H(t,x,s,m(t,x)) + f(m(t,x))$, then expressing
\begin{equation*}
    \begin{aligned}
        & I_{n} = -\int_t^T \int_{\Tt^d} \bigl(H(t',x',D\upsilon(t',x'), m(t',x')) + f(m(t',x'))\bigr) K_{{\xi_n}, {\tau_n}}(t',x'; t,x) \,dx'dt' \\
        & \qquad\qquad + \int_t^\infty\int_{\Tt^d} F(t,x, \theta(t',x')) K_{{\xi_n}, {\tau_n}}(t',x'; t,x) \,dx'dt'.
    \end{aligned}
\end{equation*}

On the other hand, we claim that
\begin{equation}\label{hc2-leftover-converges-Lbeta'}
    \tilde{I}_{n} := \int_T^\infty\int_{\Tt^d}
h^{\mathrm{c}}(t',x';t,x)K_{{\xi_n}, {\tau_n}}(t',x';t,x)\,dx'dt' \to 0 \qquad \text{in } L^{\beta'}(Q).
\end{equation}
Indeed, since $K_{{\xi_n}, {\tau_n}}(t',x';t,x) = 0$ for $t' > t+\tau_n$, the last integral is nonzero only when \(t\in(T-\tau_n,T)\). Thus
\[|\tilde{I}_{n}|  \leq \chi_{\{T-\tau_n<t<T\}}
\int_{\Tt^d}
\bigl|H(t,x,Du_T(x'),m(t,x))+f(m(t,x))\bigr|\cdot
\frac{1}{(\xi_n)^d}
\overline{\zeta}\left(\frac{x'-x}{\xi_n}\right)\,dx'.\]
Since \(Du_T\in L^\infty(\Tt^d;\RR^d)\), the bounds
\eqref{eq:assH.lower.simple}, \eqref{res:H.upper.simple}, and
\eqref{eq:highest-order-m.separated} imply
\[
\bigl|H(t,x,Du_T(x'),m(t,x))+f(m(t,x))\bigr|
\leq C\bigl(m(t,x)^{\beta-1}+V(t,x)+1\bigr),
\]
uniformly in \(x'\). Hence
\[|\tilde{I}_{n}| \leq
C\chi_{\{T-\tau_n<t<T\}}
\bigl(m(t,x)^{\beta-1}+V(t,x)+1\bigr) \qquad\text{a.e.~in } Q,
\]
uniformly in $n$. Therefore
\[
\int_Q |\tilde{I}_{n}(t,x)|^{\beta'}\,dxdt
\leq
C\int_{T-\tau_n}^T\int_{\Tt^d}
\bigl(m(t,x)^\beta+V(t,x)^{\beta'}+1\bigr)\,dxdt,
\]
which converges to zero as \(n\to\infty\), because
\(m^\beta+V^{\beta'}+1\in L^1(Q)\). Thus
\(\tilde{I}_{n}\to 0\) in \(L^{\beta'}(Q)\).

In view of \eqref{hc2-leftover-converges-Lbeta'}, it remains to prove
\(I_n-\tilde I_n\to0\) in \(L^{\beta'}(Q)\). But, by
\eqref{hc2-converges-ae} and the fact that \(\tilde I_n\to0\) a.e.~in \(Q\),
we know that \(I_n-\tilde I_n\to0\) a.e.~in \(Q\). To lift this a.e.\
convergence to convergence in \(L^{\beta'}(Q)\), we show that the sequence is
uniformly bounded in \(L^{q\beta'}(Q)\) for some \(q>1\). Indeed, since
\(q>1\), the uniform \(L^{q\beta'}(Q)\)-bound implies uniform integrability
in \(L^{\beta'}(Q)\), and Vitali's convergence theorem then gives
\(I_n-\tilde I_n\to0\) in \(L^{\beta'}(Q)\).

Before specifying $q$, we estimate the desired quantity as follows:
\begin{equation}\label{mollification-Lqbeta'-bound-calculation}
 \begin{aligned}
 & \int_Q |I_{n} - \tilde{I}_{n}|^{q\beta'}\,dxdt = \int_Q \left|\int_t^T \int_{\Tt^d} h^{\text{c}}(t',x';t,x)  K_{{\xi_n}, {\tau_n}}(t',x';t,x)\,dx'dt'\right|^{q\beta'} dxdt \\
 & \ \leq \int_Q  \left(\int_t^T \int_{\Tt^d} |h^{\text{c}}(t',x';t,x)|^{q\beta'}  K_{{\xi_n}, {\tau_n}}(t',x';t,x)\,dx'dt'\right)dxdt \\
 & \ = \int_Q \frac{1}{\tau_n} \hat{\zeta}\left(\frac{s}{\tau_n}\right)\cdot \frac{1}{(\xi_n)^d} \overline{\zeta}\left(\frac{y}{\xi_n}\right)\left(\int_0^{T-s} \int_{\Tt^d} |h^{\text{c}}(t+s,x+y;t,x)|^{q\beta'} dxdt\right) dyds \\
 & \ \leq \sup_{(s,y)\in \bar{Q}} \int_0^{T-s} \int_{\Tt^d} |h^{\text{c}}(t+s,x+y;t,x)|^{q\beta'} dxdt
 \\
& \qquad = \sup_{(s,y)\in\bar{Q}} \int_0^{T-s}\int_{\Tt^d} \Bigl|H(t,x,D\upsilon(t+s,x+y),m(t,x)) + f(m(t,x)) \\
   & \qquad\qquad - H(t+s,x+y,D\upsilon(t+s,x+y),m(t+s,x+y)) - f(m(t+s,x+y))\Bigr|^{q\beta'} \,dxdt.
 \end{aligned}
\end{equation}

Since the last expression does not depend on $n$, it remains to find some $q>1$ which makes it finite. For this, we use \eqref{eq:deviation-at-fixed-pm-improved-to-m1m2} and choose some $c>1$ satisfying
\begin{equation*}
    \begin{aligned}
        & |H(t,x,p,m_1) + f(m_1) - H(t',x',p,m_2) - f(m_2)|\\
        & \qquad \leq c\bigl(|p|^{\alpha\frac{(\beta-1)}{\beta}}+ m_1^{\beta-1} + m_2^{\beta-1} + V(t,x) + V(t',x')\bigr)^{1-1/c},
    \end{aligned}
    \end{equation*}
then we set $q := c/(c-1)$. Accordingly,
    \begin{equation*}
     \begin{aligned}
       & \Bigl|H(t,x,D\upsilon(t+s,x+y),m(t,x)) + f(m(t,x)) \\
   & \qquad - H(t+s,x+y,D\upsilon(t+s,x+y),m(t+s,x+y)) - f(m(t+s,x+y))\Bigr|^{q\beta'} \\
       & \quad \leq C\Bigl(|D\upsilon(t+s,x+y)|^{\alpha\frac{(\beta-1)}{\beta}} + m(t,x)^{\beta-1} + m(t+s,x+y)^{\beta-1}\\
       &\quad\qquad\qquad + V(t,x) + V(t+s,x+y)\Bigr)^{\beta'}\\
       & \quad \leq C\bigl(|D\upsilon(t+s,x+y)|^\alpha + m(t,x)^{\beta} + m(t+s,x+y)^{\beta} + V(t,x)^{\beta'} + V(t+s,x+y)^{\beta'}\bigr).
     \end{aligned}
  \end{equation*}
Thus the right-hand side of
\eqref{mollification-Lqbeta'-bound-calculation} is bounded uniformly in
\(n\), because \(D\upsilon\in L^\alpha(Q;\RR^d)\),
\(m\in L^\beta(Q)\), and \(V\in L^{\beta'}(Q)\).

Combining this uniform \(L^{q\beta'}(Q)\)-bound with the preceding Vitali
argument gives \(I_n-\tilde I_n\to0\) in \(L^{\beta'}(Q)\). Together with
\eqref{hc2-leftover-converges-Lbeta'}, this yields \(I_n\to0\) in
\(L^{\beta'}(Q)\), and therefore \eqref{mhc-lim-new}. Finally, combining
\eqref{mh-limsup-new} and \eqref{mhc-lim-new} with
\eqref{upsilon_n.hjb.bounded.by.two-terms-new} gives
\eqref{m-times-hjb.upsilon-limit:repeated}.
\end{proof} 

We can now prove the testing result needed for the uniqueness of the
density.

\begin{lemma}\label{lem:upsilontest-wo-mintegr}
    Consider the setting of Problem~\ref{prob.mfg}. Suppose Assumptions \ref{onH.monotone}, \ref{onH.powergrowth},  and \ref{onH.high-order-pm.has-no-tx} hold. Let $m\in L^\beta(Q;\RR^+_0)$ and let $u\in \mathcal{U}(m)$. Then
    \begin{equation}\label{transport-tested-upsilon.minus.u:repeated}
    \begin{aligned}
    & \int_Qm\bigl( -(\upsilon-u)_t^{\text{ac}} + D_pH(t,x,Du,m)\cdot D(\upsilon-u) \bigr)\,dxdt\\
    &\qquad\quad \geq \int_{\Tt^d}m_0(x)(\upsilon(0,x)-u(0,x))\,dx
    \end{aligned}
    \end{equation}
    for all $\upsilon\in BV(Q)$ satisfying \eqref{Dupsilon.Lalpha}, \eqref{upsilon.jumps.forward}, and \eqref{upsilon-HJB-upper:stricter}.
\end{lemma}

\begin{remark}
Lemma~\ref{lem:upsilontest-wo-mintegr} is the testing tool needed for the
uniqueness of the density. It allows us to test the equation for one pair
\((m,u)\) against a value function associated with another density, without
assuming the cross shift-integrability condition~\eqref{mDupsilon-shift-L1}.

The difference from Lemma~\ref{lem:transport-tested-upsilon.minus.u} is the
trade-off between the hypotheses on the test function. Here
\eqref{mDupsilon-shift-L1} is replaced by the stricter Hamilton--Jacobi upper
bound~\eqref{upsilon-HJB-upper:stricter}, corresponding to the case
\(\ell=0\) in~\eqref{upsilon-HJB-upper}. This restriction lets the
Hamiltonians in the error term \(h^{\mathrm{c}}\) be evaluated at the same
momentum variable. Consequently, Assumption~\ref{onH.high-order-pm.has-no-tx}
and Lemma~\ref{lem:highest-order-m.separated} give the required
\(L^{q\beta'}\)-bound without using~\eqref{mDupsilon-shift-L1}.
\end{remark}

\begin{proof}
By Remark~\ref{rmk:last-deviation-implies-the-earlier-two},
Assumption~\ref{onH.high-order-pm.has-no-tx} implies
Assumption~\ref{onH.highest-order-p.has-no-txm} under
Assumption~\ref{onH.powergrowth}. Hence Lemma~\ref{lem:transport-tested-upsilon.minus.u}
is applicable to \((m,u)\).

Let $\upsilon_n$ be constructed as in~\eqref{upsilon_n.redefined}. For each
fixed \(n\), we have $\upsilon_n\in C^1(\bar{Q})$ and
$\upsilon_n(T,x) \leq u_T(x)$; hence $\upsilon_n$ satisfies
\eqref{Dupsilon.Lalpha}, \eqref{upsilon.jumps.forward},
\eqref{upsilon-HJB-upper}, and \eqref{mDupsilon-shift-L1}. Thus,
by Lemma~\ref{lem:transport-tested-upsilon.minus.u},
    \[\begin{aligned}
    & \int_Qm\bigl( -(\upsilon_n)_t + u_t^{\text{ac}} + D_pH(t,x,Du,m)\cdot D(\upsilon_n-u) \bigr)\,dxdt\\
    &\qquad\quad \geq \int_{\Tt^d}m_0(x)(\upsilon_n(0,x)-u(0,x))\,dx.
    \end{aligned}
    \]
    Since $m(-u_t^{\text{ac}}+H(t,x,Du,m)) = 0$ by~\eqref{def:hjb-strong}, and
    \[H(t,x,Du,m) + D_pH(t,x,Du,m)\cdot(D\upsilon_n-Du) \leq H(t,x,D\upsilon_n,m)\]
    by the convexity of $H(t,x,\cdot,m)$ as in Remark~\ref{rmk:cxty}, we get
    \begin{equation*}
        \int_Qm\bigl( -(\upsilon_n)_t + H(t,x,D\upsilon_n,m) \bigr)\,dxdt \geq \int_{\Tt^d}m_0(x)(\upsilon_n(0,x)-u(0,x))\,dx.
    \end{equation*} 
    Consequently, \eqref{upsilon-trace-limit-aux:repeated} and \eqref{m-times-hjb.upsilon-limit:repeated} imply
    \begin{equation}\label{transport-test-convexity.flip}
        \int_Qm\bigl( -\upsilon_t^{\text{ac}} + H(t,x,D\upsilon,m) \bigr)\,dxdt \geq \int_{\Tt^d}m_0(x)(\upsilon(0,x)-u(0,x))\,dx.
    \end{equation}
    
    Now, we have proven that \eqref{transport-test-convexity.flip} holds for all $\upsilon\in BV(Q)$ satisfying \eqref{Dupsilon.Lalpha}, \eqref{upsilon.jumps.forward}, and \eqref{upsilon-HJB-upper:stricter}. 
    Since these three conditions are preserved under convex combinations, using
the convexity of \(H(t,x,\cdot,0)\) for
\eqref{upsilon-HJB-upper:stricter}, we can replace \(\upsilon\) by
\((1-\lambda)u+\lambda\upsilon\) for \(0\leq\lambda\leq1\).
    Consequently, we get
    \begin{equation*}
        \begin{aligned}
    & \int_Qm\bigl( -\lambda(\upsilon-u)_t^{\text{ac}} - u_t^{\text{ac}} + H(t,x, Du + \lambda D(\upsilon-u), m) \bigr)\,dxdt \\
    & \qquad \geq \lambda\int_{\Tt^d}m_0(x)(\upsilon(0,x)-u(0,x))\,dx.
    \end{aligned}
    \end{equation*}
    Using the Hamilton--Jacobi equation \eqref{def:hjb-strong} and dividing by $\lambda$ for $\lambda>0$, we obtain
    \begin{equation}\label{transport-test-flipped-conv.comb}
        \begin{aligned}
    & \int_Qm\Bigl( -(\upsilon-u)_t^{\text{ac}} + \frac{H(t,x, Du + \lambda D(\upsilon-u),m) - H(t,x,Du,m)}{\lambda} \Bigr)\,dxdt \\
    & \qquad \geq \int_{\Tt^d}m_0(x)(\upsilon(0,x)-u(0,x))\,dx.
    \end{aligned}
    \end{equation}
    
    In the limit $\lambda\to 0$, the integrand on the left-hand side converges to
    \[m\bigl( -(\upsilon-u)_t^{\text{ac}} + D_pH(t,x,Du,m)\cdot D(\upsilon-u) \bigr),\]
    for a.e.~$(t,x)\in Q$. On the other hand, this integrand has an integrable upper bound that is uniform in $\lambda$. Indeed, the convexity of $H(t,x,\cdot,m)$ as in Remark~\ref{rmk:cxty} and the Hamilton--Jacobi equation~\eqref{def:hjb-strong} imply
    \[\begin{aligned}
        & m\Bigl( -(\upsilon-u)_t^{\text{ac}} + \frac{H(t,x, Du + \lambda D(\upsilon-u),m) - H(t,x,Du,m)}{\lambda} \Bigr) \\
        & \qquad \leq m\bigl( -(\upsilon-u)_t^{\text{ac}} + H(t,x, D\upsilon,m) - H(t,x,Du,m) \bigr) = m\bigl( -\upsilon_t^{\text{ac}} + H(t,x, D\upsilon,m)\bigr),
    \end{aligned}\]
Moreover, because \(H(t,x,p,\cdot)\) is non-increasing by
Remark~\ref{rmk:cxty}, we have
\[
m\bigl(-\upsilon_t^{\mathrm{ac}}+H(t,x,D\upsilon,m)\bigr)
\leq
m\bigl(-\upsilon_t^{\mathrm{ac}}+H(t,x,D\upsilon,0)\bigr)^+.
\]
The right-hand side belongs to \(L^1(Q)\), since \(m\in L^\beta(Q)\) and
\[
\bigl(-\upsilon_t^{\mathrm{ac}}+H(t,x,D\upsilon,0)\bigr)^+
\in L^{\beta'}(Q)
\]
by \eqref{upsilon-HJB-upper:stricter}. Thus the integrands in
\eqref{transport-test-flipped-conv.comb} are bounded above by an
\(L^1\)-function independent of \(\lambda\). Applying the reverse Fatou
lemma as \(\lambda\to 0^+\), we obtain
\eqref{transport-tested-upsilon.minus.u:repeated}.
\end{proof}

\begin{proof}[Proof of Theorem~\ref{maximal-u-exists}, item~\ref{thmitem:unique.m}]
    Let $m_1, m_2\in L^\beta(Q)$ and let $u_1\in \mathcal{U}(m_1)$ and $u_2\in \mathcal{U}(m_2)$. We can apply Lemma~\ref{lem:upsilontest-wo-mintegr} with
    \[(m,u,\upsilon) = (m_1,u_1,u_2) \qquad \text{and} \qquad  (m,u,\upsilon) = (m_2,u_2,u_1),\]
    because \eqref{integrability.further.du}, the inequalities in Definition~\ref{def.strong.sol}, and \eqref{integrability.further.H} imply that $\upsilon = u_1$ and $\upsilon = u_2$ satisfy \eqref{Dupsilon.Lalpha}, \eqref{upsilon.jumps.forward}, and \eqref{upsilon-HJB-upper:stricter}, respectively. Consequently, we obtain
    \[\begin{aligned}
    & \int_Qm_1\bigl( -(u_2)_t^{\text{ac}} + (u_1)_t^{\text{ac}} + D_pH(t,x,Du_1,m_1)\cdot (Du_2-Du_1) \bigr)\,dxdt\\
    &\qquad\quad \geq \int_{\Tt^d}m_0(x)(u_2(0,x)-u_1(0,x))\,dx 
    \end{aligned}\]
    and
        \[\begin{aligned}
    & \int_Qm_2\bigl( -(u_1)_t^{\text{ac}} + (u_2)_t^{\text{ac}} + D_pH(t,x,Du_2,m_2)\cdot (Du_1-Du_2) \bigr)\,dxdt\\
    &\qquad\quad \geq \int_{\Tt^d}m_0(x)(u_1(0,x)-u_2(0,x))\,dx.
    \end{aligned}\]
    Adding the two inequalities, we find
    \[\begin{aligned}
        & \int_Q\biggl((m_1-m_2)\bigl((u_1)_t^{\text{ac}} -(u_2)_t^{\text{ac}}\bigr) \\
        & \qquad\qquad + \bigl(m_1D_pH(t,x,Du_1,m_1) - m_2D_pH(t,x,Du_2,m_2)\bigr)\cdot (Du_2-Du_1) \biggr)\,dxdt \geq 0.
    \end{aligned}\]
    Then, using the Hamilton--Jacobi inequality and equality conditions in \eqref{def:hjb-strong} for $(u_1)_t^{\text{ac}}$ and $(u_2)_t^{\text{ac}}$, we obtain
    \[\begin{aligned}
        & \int_Q\biggl((m_1-m_2)\bigl(H(t,x,Du_1,m_1) - H(t,x,Du_2,m_2)\bigr) \\
        & \qquad\qquad + \bigl(m_1D_pH(t,x,Du_1,m_1) - m_2D_pH(t,x,Du_2,m_2)\bigr)\cdot (Du_2-Du_1) \biggr)\,dxdt \geq 0.
    \end{aligned}\]
    The last integrand is non-positive by \eqref{hmon}, hence it must vanish a.e. This implies $m_1 = m_2$ a.e.~because of Assumption~\ref{onH.strict-monotone}.
\end{proof}

\bibliographystyle{abbrv}
\bibliography{mfgv8_nn, supbib}
     
\end{document}